\documentclass[article]{amsart}
\usepackage{amssymb,amsfonts,amsmath,amsthm}
\usepackage[all,arc]{xy}
\usepackage{enumerate}
\usepackage{mathrsfs}
\usepackage[toc,page]{appendix}
\usepackage[left=3cm, right=3cm, bottom=3cm]{geometry}
\usepackage{graphicx}
\usepackage{tabularx}
\usepackage{url}
\usepackage{color}
\usepackage{tikz-cd}  
\usepackage{mathdots} 
\usepackage{romannum} 
\usepackage{hyperref} 
\usepackage{float} 

\newtheorem{thm}{Theorem}[section]

\newtheorem*{thm*}{Theorem}
\newtheorem*{cor*}{Corollary}
\newtheorem*{prop*}{Proposition}
\newtheorem{cor}[thm]{Corollary}
\newtheorem{prop}[thm]{Proposition}
\newtheorem{lem}[thm]{Lemma}

\theoremstyle{definition}
\newtheorem{defn}[thm]{Definition}

\newtheorem{exmp}[thm]{Example}

\newtheorem{notn}[thm]{Notation}
\newtheorem*{notn*}{Notation}

\theoremstyle{remark}
\newtheorem{rem}[thm]{Remark}

\newtheorem*{idea*}{Idea}

\makeatletter
\let\c@equation\c@thm
\makeatother
\numberwithin{thm}{section}
\numberwithin{equation}{section}

\title[Tame parahoric nonabelian Hodge correspondence]{Tame parahoric nonabelian Hodge correspondence on curves}

\author{Pengfei Huang, Georgios Kydonakis, Hao Sun and Lutian Zhao}

\begin{document}
\pagenumbering{arabic}
\maketitle
\begin{abstract}
The nonabelian Hodge correspondence for vector bundles over noncompact curves is adequately described by implementing a weighted filtration on the objects involved. In order to establish a full correspondence between a Dolbeault and a de Rham space for a general complex reductive group $G$, we introduce torsors given by parahoric group schemes in the sense of Bruhat--Tits. Combined with existing results on the Riemann--Hilbert correspondence for logarithmic parahoric connections, this gives a full nonabelian Hodge correspondence from Higgs bundles to fundamental group representations over a noncompact curve beyond the $\text{GL}_n(\mathbb{C})$-case.\\

\end{abstract}
	
\flushbottom
	
	
	
\renewcommand{\thefootnote}{\fnsymbol{footnote}}
	
\section{Introduction}

Classical Hodge theory involves the study of the abelian sheaf cohomology of a complex algebraic variety. The underlying complex, algebraic and topological structures of the variety give rise to the Dolbeault, de Rham and Betti cohomology groups with coefficients in the abelian group of the complex numbers. Nonabelian Hodge theory on the other hand concerns the cohomology with nonabelian coefficients and essentially occurs only in degree one, since the homotopy groups in higher dimensions are abelian. Unlike the classical first cohomologies (Dolbeault, de Rham, Betti) which are vector spaces, the nonabelian Hodge cohomologies are stacks. The formulation and development of the correct analogues of classical notions from the usual Hodge theory to the nonabelian setting is largely due to Simpson \cite{Simp1988, Simp, Simp-naH, Simp Higgs Local, Simp2, Simp3} building on prototypical works by several authors. The implementation of suitable stability conditions for the objects in these stacks leads to the construction of moduli spaces and these give, among a profound field of relevance, important examples of integrable systems.

The nonabelian Hodge correspondence on smooth projective curves $X$ is a fundamental prototype where one can see how the algebraic objects in these moduli spaces are related, namely, stable Higgs bundles, stable $D_X$-modules and irreducible representations of the fundamental group into the group $\text{GL}_n(\mathbb{C})$. Stable Higgs bundles bijectively correspond to solutions of a certain system of nonlinear PDEs now called the Hitchin equations. This correspondence is an instance of a Kobayashi--Hitchin correspondence and was established by Hitchin \cite{Hit87} and more generally by Simpson \cite{Simp1988, Simp Higgs Local}. Another instance of a Kobayashi--Hitchin correspondence appears in the works of Donaldson \cite{Donaldson} and more generally of Corlette \cite{Corlette}, where the stability condition for a $D_X$-module can be interpreted as the condition for the existence of a harmonic bundle, and the nonlinear equations for a harmonic bundle are yet another way of writing the Hitchin equations. The classical Riemann--Hilbert correspondence for stable $D_X$-modules on a smooth projective curve completes the passage to irreducible representations.

In the present article, we are concerned with the problem of establishing a complete correspondence on a punctured curve $X_{\boldsymbol D}:=X \backslash \boldsymbol D$ between the Dolbeault, de Rham and Betti spaces, for a general complex reductive group under the tameness condition, where $\boldsymbol D$ is a reduced effective divisor on $X$, in other words, a set of punctures on $X$. In the $\text{GL}_n(\mathbb{C})$-version of this story, a categorical equivalence was explained by Simpson in \cite{Simp} and the right objects have been identified in this case as stable filtered Higgs bundles, stable filtered $D_X$-modules, both with regular singularities at the punctures, and irreducible representations of fundamental group; an analytic construction of the relevant moduli space was provided later on by Konno \cite{Konno}. In the work of Simpson, the passage from the Dolbeault to the de Rham space, that is, from filtered Higgs bundles to filtered $D_X$-modules, goes through a class of analytic objects on the punctured curves $X_{\boldsymbol D}$, namely, harmonic bundles on $X_{\boldsymbol D}$ satisfying a certain growth condition for the curvature at the punctures called tameness. Then, the equivalence between the de Rham and Betti space, that is, from $D_X$-modules to representations, pertains a version of the Riemann--Hilbert correspondence on punctured algebraic curves with regular singularities. It is important to remark here that this step is much more involved compared to the standard Riemann--Hilbert in the compact curve case.

The full problem for general complex reductive groups on noncompact curves turns out to be significantly more subtle. The main reason is that a generalization of the structure of a filtered Higgs bundle does not suffice to provide the right holomorphic objects that will one-to-one correspond to fundamental group representations. In particular, when $G=\text{GL}_n(\mathbb{C})$, the extra structure at the points of a divisor is adequately described by a weighted filtration, usually called a \emph{parabolic structure}, in the fiber of a vector bundle. However, beyond the group $\text{GL}_n(\mathbb{C})$, the parabolic structure does not suffice to provide the full picture; one has to go further out in a formal neighborhood of the puncture and this requires a different approach.

\subsection{Examples}
In order to stress further the fact that the parabolic structure is not enough to fully describe the correspondence other than $\text{GL}_n(\mathbb{C})$, we include below basic examples in the case of $\text{SL}_2(\mathbb{C})$. We first fix some notation: $R=\mathbb{C}[[z]]$ and $K=\mathbb{C}((z))$.

\subsection*{\textbf{Connections}}
Let $A(z)=\sum_{i=0} a_i z^i$ be an element in $\mathfrak{gl}_n(K)$, and we consider $A(z)\frac{dz}{z}$ as a connection form with regular singularities. The action of $g \in {\rm GL}_n(K)$ on $A(z)\frac{dz}{z}$ is defined as
\begin{align*}
g \circ A(z)\frac{dz}{z} := \left( {\rm Ad}(g)A(z) \right) \frac{dz}{z}+dg \cdot g^{-1},
\end{align*}
which is the \emph{gauge action}. It is well-known that for any $A(z)$, there exists an element $g$ such that
\begin{align*}
g \circ A(z)\frac{dz}{z} = a \frac{dz}{z},
\end{align*}
where $a \in \mathfrak{gl}_n(\mathbb{C})$ is a constant matrix (see \cite[Theorem 5.1.4]{HoTakTan} for instance). This result is very powerful, and it implies that the information around punctures can be understood simply from that over the punctures. With the help of this property, the monodromy can be described directly around punctures, which is exactly $\exp(- 2 \pi \sqrt{-1} a)$, and filtered (or parabolic) $D_X$-modules appear in the Riemann--Hilbert correspondence with regular singularities.

However, this property no longer holds for general reductive groups. Consider, for instance,
\begin{align*}
A(z) = \begin{pmatrix} \frac{m}{2} & 0 \\ 0 & -\frac{m}{2} \end{pmatrix} +  \begin{pmatrix} 0 & 1 \\ 0 & 0 \end{pmatrix}z^m,
\end{align*}
as an element in $\mathfrak{sl}_2(K)$. Then, the connection form $A(z)\frac{dz}{z}$ is gauge equivalent to $a\frac{dz}{z}$ for some $a \in \mathfrak{sl}_2(\mathbb{C})$ if and only if $m$ is even \cite[\S 8.2]{BaVa}. This shows that the connection form $A(z)\frac{dz}{z}$ is not gauge equivalent to the form $a\frac{dz}{z}$ in general. Therefore, a more careful treatment is required to describe the monodromy data in this case.

\subsection*{\textbf{Higgs fields}}
Consider the matrix $B(z) = \begin{pmatrix} 0 & z \\ z^{-1} & 0 \end{pmatrix}$. Denote by
\begin{align*}
\phi(z) = B(z)\frac{dz}{z} = \left( \begin{pmatrix} 0 & 0 \\ 1 & 0 \end{pmatrix} \frac{1}{z^2} + \begin{pmatrix} 0 & 1 \\ 0 & 0 \end{pmatrix} \right) dz
\end{align*}
the corresponding Higgs field. As an element in $\mathfrak{gl}_2(K) \frac{dz}{z}$, the poles of all eigenvalues of $\phi(z)$ have order one, thus the tameness condition is satisfied (see \cite{Simp}). Note that the Higgs field has a nonzero coefficient of the term $\frac{1}{z^2}$, and it seems to be hard to find parabolic structures preserved by $\phi(z)$ in this form. Nonetheless, taking $g= \begin{pmatrix} -1 & 1 \\ 1 & 1 \end{pmatrix} \cdot \begin{pmatrix} 1 & 0 \\ 0 & z \end{pmatrix}$, it is easy to check that
\begin{align*}
{\rm Ad}(g) \phi(z) = \begin{pmatrix} -1 & 0 \\ 0 & 1 \end{pmatrix} \frac{dz}{z} \in \mathfrak{gl}_2(R) \frac{dz}{z},
\end{align*}
which is a Higgs field with trivial coefficients of higher order terms. In this form, the Higgs field is much more related to a parabolic structure on holomorphic bundles.

However, if we consider $B(z)$ as an element in $\mathfrak{sl}_2(K)$, the above calculation does not make sense because the matrix $\begin{pmatrix} 1 & 0 \\ 0 & z \end{pmatrix}$ is not a well-defined element in ${\rm SL}_2(K)$.  Moreover, the matrix $\begin{pmatrix} 0 & z \\ z^{-1} & 0 \end{pmatrix}$ is not conjugate to any matrix in $\mathfrak{sl}_2(R)$. Therefore, although the Higgs field $\phi(z)$ satisfies the tameness condition, there is a nontrivial coefficient of a higher order term in the case of ${\rm SL}_2(\mathbb{C})$, which is an obstacle to defining a parabolic structure directly.

\vspace{3mm}

The inadequacy of the parabolic structure on the fiber for providing a nonabelian Hodge correspondence for general reductive groups on noncompact curves, has been pointed out in the studies for establishing a Riemann--Hilbert correspondence on noncompact curves for groups other than $\text{GL}_n(\mathbb{C})$. It first became apparent through the work of Boalch \cite{Bo} that the right objects to replace the (parabolic) vector bundles on the de  Rham (and Dolbeault side) for an arbitrary group $G$ are \emph{torsors given by parahoric group schemes} in the sense of Bruhat--Tits \cite{BT1, BT2}. In that work, a notion of weight was introduced, which was used to define parahoric groups, and a local Riemann--Hilbert correspondence for logarithmic parahoric connections was established, called briefly \emph{logahoric connections}. Furthermore, a prediction of the local data in the extension of the nonabelian Hodge correspondence for a complex reductive group was given (see \cite[\S 6]{Bo}).



\subsection{Main Results}

In this article we establish a correspondence among stable logahoric Higgs torsors of degree zero, stable logahoric connections of degree zero and stable filtered local systems of degree zero on a punctured curve, for an arbitrary complex reductive group. This correspondence not only holds at the level of categories, but also for the corresponding moduli spaces. Before we state the idea and result, we first introduce some notation. 

A \emph{real weight} is an element in $Y(T)\otimes_{\mathbb{Z}} \mathbb{R}$, where $Y(T):={\rm Hom}(\mathbb{G}_m,T)$ and $T$ is a maximal torus of a complex reductive group $G$. A \emph{rational weight} is a weight with coefficient in $\mathbb{Q}$. In this paper, a real weight is called a weight if there is no ambiguity. Weights will be applied to define the central objects discussed in this article, namely, parahoric group schemes (Definition \ref{defn parahoric_group}), adapted metrics (Definition \ref{defn adapt metric}), tame adapted $G$-Higgs bundles (Definition \ref{defn tame ana Higgs}) and tame adapted $G$-connections (Definition \ref{defn tame ana conn}). We will be using the following notation in the sequel of the article: 
\begin{itemize}
	\item $\boldsymbol\bullet=\{\bullet_x, x \in \boldsymbol D\}$ is a collection of weights labelled by punctures in $\boldsymbol D$, where $\bullet=\alpha,\beta,\gamma$.
\item $\mathcal{G}_{\boldsymbol\bullet}$ is the corresponding parahoric group scheme, where $\bullet=\alpha,\beta,\gamma$;
	\item $\phi_{\boldsymbol\alpha},\varphi_{\boldsymbol\alpha},d'_{\boldsymbol\beta},\nabla_{\boldsymbol\beta}$ are collections of elements in $\mathfrak{g}$;
	\item $M_{\boldsymbol\gamma}$ is a collection of elements in $G$;
\end{itemize}
Based on this notation, we introduce five categories for corresponding notions of stability defined within the text:
\begin{itemize}
	\item $\mathcal{C}_{\rm Dol}(X_{\boldsymbol D},G,\boldsymbol\alpha, \phi_{\boldsymbol\alpha})$: the category of $R_h$-stable tame metrized $\boldsymbol\alpha$-adapted $G$-Higgs bundles of degree zero on $X_{\boldsymbol D}$, and the Levi factors of residues of the Higgs field are $\phi_{\boldsymbol\alpha}$ at punctures (see \S\ref{subsect ana stab Higgs and Dmod});
	\item $\mathcal{C}_{\rm dR}(X_{\boldsymbol D},G,\boldsymbol\beta,d'_{\boldsymbol\beta})$: the category of $R_h$-stable tame metrized $\boldsymbol\beta$-adapted $G$-connections of degree zero on $X_{\boldsymbol D}$, and the Levi factors of residues of the connection are $d'_{\boldsymbol\beta}$ at punctures (see \S\ref{subsect ana stab Higgs and Dmod});
	\item $\mathcal{C}_{\rm Dol}(X,\mathcal{G}_{\boldsymbol \alpha},\varphi_{\boldsymbol\alpha})$: the category of $R$-stable logahoric $\mathcal{G}_{\boldsymbol\alpha}$-Higgs torsors of degree zero on $X$, and the Levi factors of residues of the Higgs field are $\varphi_{\boldsymbol\alpha}$ at punctures (see \S\ref{sect alg Higgs and conn});
	\item $\mathcal{C}_{\rm dR}(X,\mathcal{G}_{\boldsymbol\beta},\nabla_{\boldsymbol\beta})$: the category of $R$-stable logahoric $\mathcal{G}_{\boldsymbol\beta}$-connections of degree zero on $X$ and the Levi factors of residues of the connection are $\nabla_{\boldsymbol\beta}$ at punctures (see \S\ref{sect alg Higgs and conn});
	\item $\mathcal{C}_{\rm B}(X_{\boldsymbol D}, G,\boldsymbol\gamma, M_{\boldsymbol\gamma})$: the category of $R$-stable $\boldsymbol\gamma$-filtered $G$-local systems of degree zero, of which the Levi factors of monodromies around punctures are $M_{\boldsymbol\gamma}$ (see \S\ref{subsect main result in cat}).
\end{itemize}

The relation among $\alpha$, $\beta$ and $\gamma$ is given from a careful study of local data (see \ref{subsect local study}). Implementing a class of model metrics $h_0$ locally around a puncture, we have the following Main Table describing the permutations of the jump values of the metric $h_0$ (weights) and the data of Levi factors of residues and monodromies (residues $\backslash$ monodromies), in the case of an arbitrary complex reductive group $G$:
\begin{table}[H]
	\begin{tabular}{|c|c|c|c|}
		\hline
		& Dolbeault & de Rham & Betti \\
		\hline
		weights & $\alpha$ & $\beta=\alpha-(s_\alpha+\bar{s}_\alpha)$ & $\gamma=-(s_\alpha+\bar{s}_\alpha)$ \\
		\hline
		residues $\backslash$ monodromies & $\varphi_\alpha=s_\alpha+Y_\alpha$ & $\nabla_\beta$ & $M_\gamma$\\
		\hline
	\end{tabular}
\end{table}
\noindent where
\begin{align*}
\nabla_\beta =\alpha+(s_\alpha-\bar{s}_\alpha)-( H_\alpha + X_\alpha - Y_\alpha)
\end{align*}
and
\begin{align*}
M_\gamma=\exp\left( - 2 \pi i ( \alpha + s_\alpha - \bar{s}_\alpha) \right)\exp \left( 2 \pi i( H_\alpha + X_\alpha - Y_\alpha ) \right).
\end{align*}
In the above, we consider a tame $\alpha$-adapted Higgs field with residue $\varphi_{\alpha}$ with Jordan decomposition $\varphi_{\alpha} = s_{\alpha} + Y_{\alpha}$, where $s_{\alpha}$ is the semisimple part and $Y_{\alpha}$ the nilpotent part, and $(X_{\alpha}, H_{\alpha}, Y_{\alpha})$ is a Kostant--Rallis triple completing $Y_{\alpha}$ (see Section \ref{subsect local study} for full descriptions). This table in the case when the Higgs field $\phi$ does not have a nilpotent part coincides with the table predicted in \cite[\S 6]{Bo}. Moreover, in the case of the group ${\rm GL}_n(\mathbb{C})$, the table coincides with that of Simpson in \cite[\S 5]{Simp}.

Now we are ready to give the idea of the main result. The starting objects are $R_h$-stable tame metrized $G$-Higgs bundles and $R_h$-stable tame metrized $G$-connections on $X_{\boldsymbol D}$, where a $G$-connection is a $G$-bundle together with an integrable connection. We prove a version of Kobayashi--Hitchin correspondence (Theorem \ref{thm KH corr Higgs}) which guarantees the equivalence between the existence of a harmonic metric and the stability condition we consider. Furthermore, this correspondence implies the equivalence between $R_h$-stable tame metrized $G$-Higgs bundles and $R_h$-stable tame metrized $G$-connections on $X_{\boldsymbol D}$:
\begin{center}
	\begin{tikzcd}
	\mathcal{C}_{\rm Dol}(X_{\boldsymbol D},G,\boldsymbol\alpha,\phi_{\boldsymbol\alpha}) \arrow[rr, equal, "\text{Theorem \ref{thm KH corr Higgs}}"]  & &
	\mathcal{C}_{\rm dR}(X_{\boldsymbol D},G,\boldsymbol\beta,d'_{\boldsymbol\beta}).
	\end{tikzcd}
\end{center}

Next, motivated by the transparency of the parahoric language in describing a full local Riemann--Hilbert correspondence as outlined above, we introduce entirely algebraic objects along with stability conditions from the parahoric point of view: $R$-stable logahoric Higgs torsors and $R$-stable logahoric connections. We prove that there is a one-to-one correspondence between the analytic objects and the algebraic (parahoric) objects by following the scheme of the original work of Simpson \cite{Simp},
\begin{center}
	\begin{tikzcd}
	\mathcal{C}_{\rm Dol}(X_{\boldsymbol D},G,\boldsymbol\alpha,\phi_{\boldsymbol\alpha}) \arrow[rr, equal] \arrow[d, equal,"\text{Lemma \ref{lem ang=alg Dol (dR)}}"] & &
	\mathcal{C}_{\rm dR}(X_{\boldsymbol D},G,\boldsymbol\beta,d'_{\boldsymbol\beta}) \arrow[d, equal,"\text{Lemma \ref{lem ang=alg Dol (dR)}}"]\\
	\mathcal{C}_{\rm Dol}(X,\mathcal{G}_{\boldsymbol\alpha},\varphi_{\boldsymbol\alpha}) \arrow[rr, equal,"\text{Theorem \ref{thm alg Dol and DR}}"]& & \mathcal{C}_{\rm dR}(X,\mathcal{G}_{\boldsymbol\beta},\nabla_{\boldsymbol\beta})
	\end{tikzcd}
\end{center}
where $\phi_{\boldsymbol\alpha}=\varphi_{\boldsymbol\alpha}$ and $d'_{\boldsymbol\beta}=\nabla_{\boldsymbol\beta}$.

Finally, by a global version of Riemann--Hilbert correspondence for logarithmic parahoric connections, we have
	\begin{center}
	\begin{tikzcd}
	\mathcal{C}_{\rm Dol}(X_{\boldsymbol D},G,\boldsymbol\alpha,\phi_{\boldsymbol\alpha}) \arrow[rr, equal] \arrow[d, equal] & &
	\mathcal{C}_{\rm dR}(X_{\boldsymbol D},G,\boldsymbol\beta,d'_{\boldsymbol\beta}) \arrow[d, equal] & &\\
	\mathcal{C}_{\rm Dol}(X,\mathcal{G}_{\boldsymbol\alpha},\varphi_{\boldsymbol\alpha}) \arrow[rr, equal]& & \mathcal{C}_{\rm dR}(X,\mathcal{G}_{\boldsymbol\beta},\nabla_{\boldsymbol\beta}) \arrow[rr, equal, "\text{Corollary \ref{cor tame parah RH corr}}"] & & \mathcal{C}_{\rm B}(X_{\boldsymbol D}, G,\boldsymbol\gamma, M_{\boldsymbol\gamma}),
	\end{tikzcd}
\end{center}

In total, we get a full passage from logahoric Higgs torsors to fundamental group representations into a general complex reductive group $G$, and we obtain the correspondence at the level of categories that fully reduces to the main theorem of Simpson in \cite{Simp}; :
\begin{thm}[Theorem \ref{thm main thm in cat}]\label{cat_thm_introduc}
	The following categories are equivalent
	\begin{align*}
	\mathcal{C}_{\rm Dol}(X,\mathcal{G}_{\boldsymbol\alpha},\varphi_{\boldsymbol\alpha}) \cong \mathcal{C}_{\rm dR}(X,\mathcal{G}_{\boldsymbol\beta},\nabla_{\boldsymbol\beta}) \cong \mathcal{C}_{\rm B}(X_{\boldsymbol D}, G,\boldsymbol\gamma, M_{\boldsymbol\gamma}).
	\end{align*}
\end{thm}

In a prelude by the second, third and fourth author \cite{KSZparh}, an algebraic construction of the moduli space of $R$-stable logahoric Higgs torsors was given (Dolbeault space) for rational weights, and a similar argument applies to construct the corresponding algebraic moduli space of $R$-stable logahoric connections as well (de Rham space). An important feature of this moduli space is that it reduces to the category of filtered Higgs bundles of Simpson \cite{Simp} when $G=\text{GL}_n(\mathbb{C})$. The paper \cite{KSZparh} includes a more careful bibliographical overview of the various approaches that have appeared in the literature on the problem of introducing the right holomorphic objects for establishing a nonabelian Hodge correspondence beyond the group $\text{GL}_n(\mathbb{C})$, thus we shall not repeat this here. Based on the construction of the relevant moduli spaces as quasi-projective schemes provided in \cite{KSZparh}, we show that this moreover gives an equivalence between moduli spaces under the condition that weights are rational, and we also believe that the statement also holds for real weights with slight modifications.

\begin{thm}[Theorem \ref{thm main thm in mod space}]
	There is an isomorphism of complex analytic spaces
	\begin{align*}
	\mathcal{M}^{\rm (an)}_{\rm B}(X_{\boldsymbol D},G, \boldsymbol \gamma, M_{\boldsymbol \gamma} ) \cong \mathcal{M}^{\rm (an)}_{\rm dR}(X, \mathcal{G}_{\boldsymbol \beta}, \nabla_{\boldsymbol \beta}),
	\end{align*}
	and we also have a homeomorphism of topological spaces
	\begin{align*}
	\mathcal{M}^{\rm (top)}_{\rm Dol}(X, \mathcal{G}_{\boldsymbol \alpha}, \varphi_{\boldsymbol \alpha}) \cong \mathcal{M}^{\rm (top)}_{\rm dR}(X, \mathcal{G}_{\boldsymbol \beta}, \nabla_{\boldsymbol \beta}).
	\end{align*}
\end{thm}
\noindent The letter $\mathcal{C}$ in Theorem \ref{cat_thm_introduc} is used to represent the \emph{categories}, and the letter $\mathcal{M}$ in the above theorem is used to denote the \emph{moduli spaces}.

\subsection{Relation to Previous Works}
We believe it is useful for the reader to comment on the comparison of the results in this paper with the work of Biquard--Garc\'{i}a-Prada--Mundet i Riera \cite{BGM}. In that work, the authors introduce a notion of stable parabolic $G$-Higgs bundle and stable parabolic $G$-local system, and prove analytically an explicit correspondence between the two notions. However in that language, to a representation corresponds not a single holomorphic bundle, but rather a class of holomorphic bundles equivalent under gauge transformations which can have meromorphic singularities (see \cite[Sections 1 and 3]{BGM}).  Moreover, for this version of the correspondence it seems difficult to reduce to the main Theorem of Simpson \cite{Simp} for the case when the group is $\text{GL}_n(\mathbb{C})$ by taking a faithful representation, and the main problem is sorting out how the stability conditions would relate in particular (see \cite[\S 5.2]{BGM}). It was predicted, nonetheless, in their work that the point of view of parahoric torsors would be more transparent for introducing the right algebraic objects and get a full equivalence of the corresponding categories, and this is exactly what is done in the present article.

For establishing the correspondence between the analytic counterparts of our algebraic objects, namely, in our sections \S\ref{sect ana Higgs and conn} and \S\ref{subsect Kob-Hitchin}, we have adapted in our setting the very explicit analytic descriptions from \cite{BGM}. The model hermitian metric description we get is not subject to any condition and we use it to rewrite and complete the Table for the associated weights and monodromies as appears in \cite[Table 1]{BGM}, in particular including the data on the Betti side as well (cf. Remarks \ref{rem_functor}, \ref{rem_question_BGM}, \ref{comparison_model_real} and  \ref{rem_reduce_to_Simpson} within the text for more details).

The techniques followed in this article can be also applied to establish the correspondence for the analogous categories involving a real reductive group $G$ with the use of appropriate initial $\boldsymbol \theta$-adapted model metrics $h_0$ on an analytic tame $G$-Higgs bundle that give an approximate solution to the Hermite--Einstein equation near the punctures. In \cite[\S 5.1]{BGM}, the authors provide such models subject to a certain condition on the nilpotent part of the graded pieces of the residue of the Higgs field. Since it was not clear to us how to get past this condition and establish the analytic part of the correspondence in general for our parahoric objects, we did not deal with the real group case here.

The nonabelian Hodge theory for the character varieties considered here provides a way to study the topology of these moduli spaces and compute their Poincar\'{e} polynomials and $E$-polynomials, extending the works of Hausel--Thaddeus \cite{HaTa1, HaTa2}, Hausel--Rodr\'{i}guez-Villegas \cite{HaRo}, Garc\'{i}a-Prada--Heinloth \cite{GPHe}, Garc\'{i}a-Prada--Heinloth--Schmitt \cite{GPHeSch}, Mellit \cite{Me1, Me2} and Gothen--Oliveira \cite{GoOl} among others. Moreover, the methods and objects introduced in this tame framework already indicate the strategy of the approach for studying the wild nonabelian Hodge correspondence for complex reductive groups (see \cite{BiBo}, \cite{Sabbah} for the $\text{GL}_n(\mathbb{C})$-case).

\subsection{Organization of the Paper}

The main body of the paper is organized as follows. In \S\ref{sect parah grp}, we review the definition of parahoric group schemes and some related properties. In \S\ref{sect parah tor}, we define adapted $G$-bundles (Definition \ref{defn herm G-bun}), which are $G$-bundles with an adapted hermitian metric. Then, in \S\ref{subsect functor Xi}, we define a functor $\Xi$, which sends adapted $G$-bundles to parahoric torsors. This functor is an analogue of Simpson's construction in \cite[\S 3]{Simp}. Furthermore, this functor not only preserves the degree (Proposition \ref{prop analytic=algebraic degree}) but also the stability condition (Proposition \ref{prop stab of parah tor}). In \S\ref{sect ana Higgs and conn}, analytic $G$-Higgs bundles and analytic $G$-connections (principal bundles equipped with a connection) are studied. Two important classes of objects are introduced and studied: \emph{$R_h$-stable tame metrized $\boldsymbol\alpha$-adapted $G$-Higgs bundles of degree zero} and \emph{$R_h$-stable tame metrized $\boldsymbol\beta$-adapted $G$-connections of degree zero}. Under the assumption of the existence of harmonic metrics, we prove that the corresponding categories are equivalent (Theorem \ref{thm ana Dol and DR}). In \S\ref{sect alg Higgs and conn}, we define two algebraic objects: logahoric Higgs torsors and logahoric connections. We generalize the functor $\Xi$ given in \S\ref{subsect functor Xi} to Higgs bundles and connections (Proposition \ref{prop ana=alg Higgs and conn}). This is a bridge connecting analytic and algebraic objects. In \S\ref{sec_categorical naHc}, we prove the categorical version of the main result (Theorem \ref{thm main thm in cat}) while the version concerning moduli spaces is given in \S\ref{sec_naHc moduli spaces}.

\vspace{2mm}

\subsection*{\textbf{Terminology.}}  For a smooth projective algebraic curve, the term \emph{nonabelian Hodge correspondence} is widely used to refer to the passage from Higgs bundles to local systems, since the Riemann--Hilbert correspondence is well-understood in this setting. In the noncompact case, the Riemann--Hilbert requires a far more delicate treatment to be established and the term nonabelian Hodge correspondence sometimes refers to the passage from the Dolbeault to the de Rham space only; see for instance the survey description in \cite[\S1.3]{Bo17}. Yet, we use here the term \emph{tame parahoric nonabelian Hodge correspondence} for the full passage from Higgs bundles to fundamental group representations in order to keep up with the nomenclature from the smooth projective case.

\vspace{2mm}

\begin{notn*} Throughout the article, we will be distinguishing the notation between the parahoric objects on $X$ and holomorphic principal bundles on $X_{\boldsymbol D}$ as follows:
	
	\vspace{2mm}
	
	\begin{tabularx}{\textwidth}{XXX}
		Curve: & $X$ & $X_{\boldsymbol D}$ \\
		\hline
		Torsor/Bundle: & $\mathcal{E}$ & $E$ \\
		Reduction of structure group: & $\varsigma$ & $\sigma$ \\
		Character: & $\kappa$ & $\chi$ \\
		Higgs field: & $\varphi$ & $\phi$ \\
		Connection: & $\nabla$ & $d'$\\
		\hline
	\end{tabularx}
	\vspace{2mm}
\end{notn*}

\section{Parahoric Group Schemes}\label{sect parah grp}

Let $G$ be a connected complex reductive group. We fix a maximal torus $T$ in $G$ with Lie algebras $\mathfrak{t}$ and $\mathfrak{g}$. Let $X(T):={\rm Hom}(T,\mathbb{G}_m)$ be the character group and $Y(T):={\rm Hom}(\mathbb{G}_m,T)$ be the co-character group (or group of one-parameter subgroups of $T$). Let
\begin{align*}
	\langle \cdot, \cdot \rangle: Y(T) \times X(T) \rightarrow \mathbb{Z}
\end{align*}
be the canonical pairing, which can be extended to $\mathbb{R}$ by tensoring $Y(T)$ and $X(T)$ with $\mathbb{R}$. A co-character with coefficients in $\mathbb{R}$ (resp. $\mathbb{Q}$) is called a \emph{real weight} (resp. \emph{rational weight}). If there is no ambiguity, a \emph{weight} is usually regarded as a \emph{real weight}. We denote by $\mathcal{R}$ the root system with respect to the maximal torus $T$, and let $\mathcal{R}_+ \subseteq \mathcal{R}$ be the set of positive roots. Given a root $r \in \mathcal{R}$, there is an isomorphism of Lie algebras
\begin{align*}
	{\rm Lie}(\mathbb{G}_a) \rightarrow ({\rm Lie}(G))_r,
\end{align*}
which induces a natural homomorphism of groups
\begin{align*}
	u_r: \mathbb{G}_a \rightarrow G,
\end{align*}
such that $t u_r(a) t^{-1}= u_r(r(t)a)$ for $t \in T$ and $a \in \mathbb{G}_a$. Denote by $U_r$ the image of the homomorphism $u_r$, which is a closed subgroup. Furthermore, a reductive group $G$ is generated by its subgroups $T$ and $U_r$ for $r \in \mathcal{R}$. Let $g$ be an element in $G$, and $g$ can be written as a product $g=g_t \prod_{r \in \mathcal{R}} g_r$, where $g_t \in T$ and $g_r \in U_r$. Sometimes, we write $g$ as a tuple $(g_t,g_r)_{r \in \mathcal{R}}$ for convenience.

\subsection{Parahoric subgroups}
Given a weight $\theta$, under differentiation, we can consider $\theta$ as an element in $\mathfrak{t}$, which is the Lie algebra of $T$. We define the integer
\begin{align*}
	m_r(\theta):=\lceil -r(\theta) \rceil,
\end{align*}
where $\lceil \cdot \rceil$ is the ceiling function and $r(\theta):=\langle \theta, r \rangle$. We then introduce the following:
	
\begin{defn}\label{defn parahoric_group}
Let $R:=\mathbb{C}[[z]]$ and $K:=\mathbb{C}((z))$. With respect to the above data, we define the \emph{parahoric subgroup $G_{\theta}(K)$} of $G(K)$ as
\begin{align*}
G_{\theta}(K):=\langle T(R), U_r(z^{m_r(\theta)}R), r \in \mathcal{R} \rangle.
\end{align*}
Denote by $\mathcal{G}_{\theta}$ the corresponding group scheme of $G_{\theta}(K)$, which is called the \emph{parahoric group scheme}.
\end{defn}
	
Now we consider another definition of a subgroup of $G(K)$ determined by $\theta$:
\begin{align*}
	G'_{\theta}(K):=\{g(z) \in G(K) \text{ }|\text{ } z^{\theta} g(z) z^{-\theta} \text{ has a limit as $z \rightarrow 0$}\},
\end{align*}
where $z^{\theta}:=e^{\theta \ln z}$. This definition is a more analytic one. The following lemma shows that the two definitions are equivalent.
\begin{lem}\label{lem two defn parah grp}
	Given a weight $\theta$, we have $G_{\theta}(K)=G'_\theta(K)$.
\end{lem}
	
\begin{proof}
We first prove that $G_{\theta}(K) \subseteq G'_\theta(K)$. Since $G_{\theta}(K)$ is generated by $T(R)$ and $U_r(z^{m_r(\theta)}R)$, it is enough to check that the generators are included in $G'_\theta(K)$. Since $\theta \in \mathfrak{t}$, the element $z^{\theta}$ commutes with elements in $T(R)$. Thus, taking $g_t(z) \in T(R)$, we have
\begin{align*}
	\lim_{z \rightarrow 0} z^{\theta} g_t(z) z^{-\theta}= \lim_{z \rightarrow 0} g_t(z),
\end{align*}
which is bounded when $z$ approaches zero. Now let $g_r(z) \in U_r(z^{m_r(\theta)}R)$, and consider
\begin{align*}
	z^{\theta} g_r(z) z^{-\theta}= z^{r(\theta)} g_r(z).
\end{align*}
The condition $g_r(z) \in U_r(z^{m_r(\theta)}R)$ means that the ``degree" of $g_r(z)$ is always greater than $m_r(\theta)$, i.e.
\begin{align*}
	r(\theta)+m_r(\theta) = r(\theta) + \lceil -r(\theta) \rceil \geq 0,
\end{align*}
which implies that $z^{r(\theta)} g_r(z) \in U_r(R)$. Therefore, the limit
\begin{align*}
	\lim_{z\rightarrow 0} z^{\theta} g_r(z) z^{-\theta}= \lim_{z \rightarrow 0}z^{r(\theta)} g_r(z)
\end{align*}
is also bounded as $z$ approaches zero. In conclusion, we have $G_{\theta}(K) \subseteq G'_\theta(K)$.
		
Now we consider the other direction. The reductive group $G$ is generated by $T$ and $U_r$ for $r \in \mathcal{R}$, and then $G(K)$ is generated by $T(K)$ and $U_r(K)$. Given an element $g \in G'_\theta$, we write it as a product $g(z)=g_t(z) \prod_{r \in \mathcal{R}} g_r(z)$, where $g_t(z) \in T(K)$ and $g_r(z) \in U_r(K)$. Note that
\begin{align*}
	z^{\theta} g(z) z^{-\theta}= (z^{\theta} g_t(z) z^{-\theta}) \prod_{r \in \mathcal{R}} (z^{\theta} g_r(z) z^{-\theta}).
\end{align*}
Therefore, as $z$ approaches zero, $z^{\theta} g z^{-\theta}$ is bounded if and only if $z^{\theta} g_t(z) z^{-\theta}$ and $z^{\theta} g_r(z) z^{-\theta}$ are bounded for $r \in \mathcal{R}$. This implies that $g_t(z) \in T(R)$ and $g_r(z) \in U_r(z^{m_r(\theta)}R)$ for $r \in \mathcal{R}$. This finishes the proof of the lemma.
\end{proof}
	
\begin{rem}
Definition \ref{defn parahoric_group} is an algebraic definition of parahoric groups $G_{\theta}(K)$, which is given in \cite{BS}. The analytic definition of $G'_\theta(K)$ is introduced in \cite{Bo}, and plays a very important role in proving the Riemann--Hilbert correspondence for logarithmic parahoric connections, which are called \emph{logahoric connections} in this paper. Moreover, the analytic definition is also studied in \cite{BGM} as a comparison to the parabolic case.
\end{rem}
	
The above construction gives a local picture of parahoric group schemes. Now we will define parahoric group schemes globally. Let $X$ be a smooth projective curve over $\mathbb{C}$, and we also fix a reduced effective divisor $\boldsymbol D$ on $X$. In fact, the divisor $\boldsymbol D$ is a sum of, say, $s$ many distinct points on $X$. For each point $x \in \boldsymbol D$, we equip it with a weight $\theta_x$. Denote by ${\boldsymbol\theta}:=\{\theta_x, x \in \boldsymbol D\}$ the collection of weights over points in $\boldsymbol D$. Let $X_{\boldsymbol D}:=X \backslash \boldsymbol D$ be the complement of $\boldsymbol D$ in $X$, which is considered as a noncompact curve.
	
\begin{defn}\label{202}
Let ${\boldsymbol\theta}$ be a collection of weights over points in $\boldsymbol D$, for a curve $X$ and a group $G$ as above. We define a group scheme $\mathcal{G}_{\boldsymbol\theta}$ over $X$ by gluing the following local data
\begin{align*}
	\mathcal{G}_{\boldsymbol\theta}|_{X_{\boldsymbol D}} \cong G \times X_{\boldsymbol D}, \quad \mathcal{G}_{\boldsymbol\theta}|_{\mathbb{D}_x} \cong \mathcal{G}_{\theta_x}, x \in \boldsymbol{D},
\end{align*}
where $\mathbb{D}_x$ is a formal disc around $x$. This group scheme $\mathcal{G}_{\boldsymbol\theta}$ is called a \emph{parahoric Bruhat--Tits group scheme}.
\end{defn}
\noindent By \cite[Lemma 3.18]{ChGP}, the group scheme $\mathcal{G}_{\boldsymbol\theta}$ defined above is a smooth affine group scheme of finite type, flat over $X$.

\subsection{Parahoric Lie algebras}
Now we move to Lie algebras. Let $\mathfrak{g}(K)$ be the loop Lie algebra of $G(K)$. The \emph{adjoint action} of $G(K)$ on $\mathfrak{g}(K)$ is induced from that of $G$ on $\mathfrak{g}$, and we use the notation ${\rm Ad}(g)A$ for the adjoint action, where $g \in G(K)$ and $A \in \mathfrak{g}(K)$. This action is used to define the equivalence classes of Higgs bundles later on. Then, we consider the set of connection forms $\mathfrak{g}(K)\frac{dz}{z}$. The \emph{gauge action} of $G(K)$ on $\mathfrak{g}(K)\frac{dz}{z}$ is given as follows. Let $A,B \in \mathfrak{g}(K)$. The elements $A\frac{dz}{z}$ and $B\frac{dz}{z}$ are \emph{gauge equivalent}, if there exists an element $g \in G(K)$, such that
\begin{align*}
	B \frac{dz}{z} = dg \cdot g^{-1} + {\rm Ad}(g) A \frac{dz}{z}.
\end{align*}
The gauge action will be applied to define the equivalence classes of connections locally. For simplicity, the above action is called the \emph{gauge action} of $G(K)$ on $\mathfrak{g}(K)$, and $A$ is \emph{gauge equivalent} to $B$.
	
Let $\theta$ be a weight for the group $G$. With respect to the weight $\theta$, there is a natural decomposition of the Lie algebra $\mathfrak{g}$
\begin{align*}
	\mathfrak{g} = \bigoplus\limits_{\lambda \in \mathbb{R}} \mathfrak{g}_\lambda,
\end{align*}
where $\mathfrak{g}_\lambda$ is the eigenspace of the operator ${\rm ad}(\theta)$. Then, we define
\begin{align*}
\mathfrak{g}_\theta(K) = \{ \sum\limits_{i \in \mathbb{Z}, \lambda \in \mathbb{R}}  a_{i \lambda} z^i \text{ } | \text{ } a_{i \lambda} \in \mathfrak{g}_\lambda \text{ and } i+\lambda \geq 0 \}
\end{align*}
as a subset of $\mathfrak{g}(K)$. Moreover, we have an alternative way to define $\mathfrak{g}_\theta(K)$:
\begin{align*}
	\mathfrak{g}'_\theta(K) : =\{ g(z) \in \mathfrak{g}(K) \text{ } | \text{ } z^{\theta} g(z) z^{-\theta} \text{ has a limit as $z \rightarrow 0$}\}.
\end{align*}
	
\begin{lem}\label{lem two defn parah Lie alg}
	Let $\theta$ be a weight for the group $G$. We have $\mathfrak{g}_\theta(K) =  \mathfrak{g}'_\theta(K)$.
\end{lem}
	
\begin{proof}
	The proof is similar to the one of Lemma \ref{lem two defn parah grp}, and we will leave it to the reader.
\end{proof}
	
In fact, $\mathfrak{g}_\theta(K)$ is usually understood as the Lie algebra of $G_{\theta}(K)$ and is called the \emph{parahoric Lie algebra}. Furthermore, the adjoint action and the gauge action are well-defined:
\begin{lem}\cite[Lemma 3]{Bo}
	The natural adjoint action and the gauge action of $G_{\theta}(K)$ on $\mathfrak{g}(K)$ preserve $\mathfrak{g}_\theta(K)$.
\end{lem}
Given a collection of weights $\boldsymbol\theta$ over the punctures $\boldsymbol D \subseteq X$, we can define a parahoric Lie algebra $\mathfrak{g}_{\boldsymbol\theta}$ of $\mathcal{G}_{\boldsymbol\theta}$ by gluing the local data, and we can define adjoint action and gauge action (representation) of $\mathcal{G}_{\boldsymbol\theta}$ on $\mathfrak{g}_{\boldsymbol\theta}$ naturally.

\subsection{Parahoric torsors}
As a group scheme on $X$, we can define parahoric $\mathcal{G}_{\boldsymbol\theta}$-torsors on $X$ in a natural way. However, we would like to make the definition much clearer by considering the local picture. Let $E$ be a $G$-bundle (equivalently, a $G$-torsor) on $X_{\boldsymbol D}$, and let $E_x$ be a $G_{\theta_x}(K)$-torsor on $\mathbb{D}_x$ for each puncture $x \in \boldsymbol D$. Thus, defining a parahoric $\mathcal{G}_{\boldsymbol\theta}$-torsor $\mathcal{E}$ on $X$ is equivalent to giving a transition function
\begin{align*}
	\Theta_x : E_x|_{\mathbb{D}_x \cap X_{\boldsymbol D}} \rightarrow E|_{X_{\boldsymbol D} \cap \mathbb{D}_x}
\end{align*}
for each puncture $x \in \boldsymbol D$. It is easy to check that the transition function $\Theta_x$ is an element in $G(K)$. Given two parahoric $\mathcal{G}_{\boldsymbol\theta}$-torsors $\mathcal{E}_1$ and $\mathcal{E}_2$ defined by transition functions $\{\Theta_{1x}, x \in \boldsymbol D\}$ and $\{ \Theta_{2x}, x \in \boldsymbol D\}$ respectively, we say $\mathcal{E}_1$ and $\mathcal{E}_2$ are \emph{equivalent} (or \emph{isomorphic}) if there exist isomorphisms
\begin{align*}
	\psi: E_1 \rightarrow E_2, \quad \psi_x: E_{1x} \rightarrow E_{2x}, \text{ } x \in \boldsymbol D
\end{align*}
such that for each $x \in \boldsymbol D$, the following diagram commutes
\begin{center}
\begin{tikzcd}
	E_{1x} \arrow[r, "\psi_x"] \arrow[d, "\Theta_{1x}"]& E_{2x} \arrow[d, "\Theta_{2x}"] \\
	E_{1} \arrow[r, "\psi"] & E_{2}.
\end{tikzcd}
\end{center}
Therefore, an equivalence class of parahoric $\mathcal{G}_{\boldsymbol\theta}$-torsors is given by an element in
\begin{align*}
	\left[ \prod_{x \in \boldsymbol D} G_{\theta_x}(K) \backslash \prod_{x \in \boldsymbol D} G(K) / G(\mathbb{C}[X_{\boldsymbol D}])  \right],
\end{align*}
and we refer the reader to \cite[\S 2]{BS} for more details.

\section{Holomorphic Principal Bundles on $X_{\boldsymbol D}$ and Parahoric Torsors on $X$}\label{sect parah tor}
	
For the case of vector bundles, Simpson defined a functor between the category of \emph{acceptable holomorphic bundles} on $X_{\boldsymbol D}$ and the category of \emph{parabolic bundles} (also called \emph{filtered bundles}) on $X$ (see \cite[\S 10]{Simp1988} or \cite[\S 3]{Simp}). This functor, as a bridge, connects analytic objects and algebraic objects. In this section, we generalize this result to principal bundles. Based on the concept of adapted metrics (Definition \ref{defn adapt metric}), we establish a functor from the category of \emph{adapted principal bundles} on $X_{\boldsymbol D}$ to the category of \emph{parahoric torsors} on $X$ (see \S\ref{subsect functor Xi}). In fact, when $G={\rm GL}_n(\mathbb{C})$, an adapted principal bundle is exactly an acceptable holomorphic bundle. We also prove that this functor preserves the degree (Proposition \ref{prop analytic=algebraic degree}) as well as the stability condition (Proposition \ref{prop stab of parah tor}). Furthermore, we introduce two notions at the end of this section: \emph{degree zero} and \emph{polystability condition} for adapted principal bundles and parahoric torsors. The first one is crucial in non-abelian Hodge theory and the second will be used in the Kobayashi--Hitchin correspondence in \S\ref{subsect Kob-Hitchin}. Before we start, we introduce the following notation
\begin{itemize}
	\item $\mathscr{A}^p$: sheaf of $C^{\infty}$ $p$-forms;
	\item $\mathscr{A}^p(\bullet)$: sheaf of $C^{\infty}$ $\bullet$-valued $p$-forms;
	\item $\mathscr{A}^{p,q}$: sheaf of $C^{\infty}$ forms of type $(p,q)$;
	\item $\mathscr{A}^{p,q}(\bullet)$: sheaf of $C^{\infty}$ $\bullet$-valued forms of type $(p,q)$,
\end{itemize}
where $\bullet$ represents a vector bundle or a principal bundle on $X_{\boldsymbol D}$. For an open subset $U \subseteq X_{\boldsymbol D}$, we have $\mathscr{A}^0(U)=C^{\infty}(U)$, and we also use the notation $\Gamma$ for $C^{\infty}$-functions.
	
\subsection{Functor $\Xi$}\label{subsect functor Xi}
Let $X$ be a smooth projective curve over $\mathbb{C}$ together with a given reduced effective divisor $\boldsymbol D$. With respect to the data $(X,\boldsymbol D)$, we can define an open cover
\begin{align*}
	\mathfrak{U}=(X_{\boldsymbol D}, \text{  } \mathbb{D}_x, \text{ } x \in \boldsymbol D)
\end{align*}
of $X$, where $\mathbb{D}_x$ is a formal disc around the puncture $x$.
	
Let $G$ be a connected complex reductive group with a maximal compact subgroup $H$ with Lie algebra $\mathfrak{h}$. Let $g$ be an element in $G$, and we consider $g$ as a tuple $(g_t,g_r)_{r \in \mathcal{R}}$, where $g_t \in T$ and $g_r \in U_r$. Its \emph{transpose} is the tuple $(g_t, g_{-r})_{r \in \mathcal{R}}$, defined by switching the position of $g_r$ and $g_{-r}$, and will be denoted by $g^{\rm T}$. We say that $g$ is \emph{symmetric} if $g=g^{\rm T}$. Also, $g$ is \emph{hermitian} if $g=\bar{g}^{\rm T}$, that is,
\begin{align*}
	g_t=\bar{g_t} \text{ and } g_r= \bar{g}_{-r}.
\end{align*}
Note that when $G={\rm GL}_n(\mathbb{C})$, a hermitian element is exactly a hermitian matrix.
	
Let $E$ be a $G$-bundle on $X_{\boldsymbol D}$. Let $M$ be a vector space, and suppose that there is a natural $G$-action on $M$. Denote by $E(M):=E \times_G M$ the corresponding vector bundle on $X_{\boldsymbol D}$. As a special case, $E(\mathfrak{g})$ is the corresponding adjoint bundle with respect to the adjoint action on $\mathfrak{g}$. Similarly, denote by $E(G):=E \times_G G$ the associated $G$-bundle under the adjoint action of $G$ (on itself). Note that there is a natural action of $E(G)$ on $E(\mathfrak{g})$, which is induced by the adjoint action of $G$ on $\mathfrak{g}$.
	
A \emph{metric} on $E$ is considered as a section $h \in \Gamma(X_{\boldsymbol D}, E/H)$. Furthermore, we say that $h$ is a \emph{hermitian metric}, if the value of $h$ is always hermitian. In this paper, we only care about hermitian metrics, thus a metric in the sequel will always refer to a hermitian metric. Taking $g \in \Gamma(X_{\boldsymbol D},E(G))$, there is a natural action on $h$ such that $g \cdot h$ is a metric (i.e. an element in $\Gamma(X_{\boldsymbol D},E/H)$). Now given a local trivialization $e$ of $E$, we can define a metric $h_0$ such that $e$ is $h_0$-orthonormal, and this metric is called the \emph{standard metric}. Furthermore, any other metric is given by $h=h_0 \cdot g$ for some $g \in \Gamma(X_{\boldsymbol D},E(G))$. We refer the reader to \cite[\S 2.4]{BGM} for more details.

Now we consider a holomorphic $G$-bundle $(E,\partial''_E)$ on $X_{\boldsymbol D}$, where $E$ is a $G$-bundle and $\partial''_E$ is a holomorphic structure of $E$. A \emph{holomorphic structure} is a first order differential operator
\begin{align*}
	\partial''_E: \mathscr{A}^{0}(E) \rightarrow \mathscr{A}^{0,1}(E)
\end{align*}
satisfying the Leibniz rule and integrability condition (see Definition \ref{defn holomorphic struct}). Furthermore, a holomorphic structure can be regarded as an element in $\mathscr{A}^{0,1}(E(\mathfrak{g}))$ locally, and this fact comes from an equivalent definition of connections in the viewpoint of (Atiyah) Lie algebroids (see Appendix \ref{sect appendix} for more details). Sometimes, we say that $E$ is a holomorphic $G$-bundle and omit the holomorphic structure if there is no ambiguity. A \emph{metrized $G$-bundle} on $X_{\boldsymbol D}$ is a triple $(E,\partial''_E,h)$, where $(E,\partial''_E)$ is a holomorphic $G$-bundle and $h$ is a hermitian metric on $E$.
		
\begin{rem}\label{rem conv holo and alg}
When we work on the noncompact curve $X_{\boldsymbol D}$, a section in $\Gamma(X_{\boldsymbol D}, E(G))$ (or $\Gamma(X_{\boldsymbol D}, E/H))$ is not necessarily holomorphic. We will say a \emph{holomorphic section} to emphasize the condition. However, when we work on $X$, we consider everything algebraically. Therefore, a section of some objects (e.g. torsors) on $X$ is always a holomorphic one. This is an important convention in this paper, especially when we discuss Higgs fields and connections later on in \S\ref{sect alg Higgs and conn}.
\end{rem}

We can write $\mathfrak{g}= \mathfrak{h}+\mathfrak{m}$ as a direct sum, where $\mathfrak{m}$ is the complement (as vector spaces). Also, we have a well-known decomposition of the Lie group $G=H \cdot {\rm exp}(\mathfrak{m})$. Let $(E,\partial''_E,h)$ be a metrized $G$-bundle on $X_{\boldsymbol D}$. Since the metric $h$ is a section of $E/H$, it is a function with values in ${\rm exp}(\mathfrak{m})$ locally. Therefore, the term $\partial h \cdot h^{-1}$ is a well-defined element in $\mathscr{A}^{1,0} (E(\mathfrak{m}))$ (and therefore in $\mathscr{A}^{1,0} (E(\mathfrak{g}))$), where $\partial$ is the $(1,0)$-part of the usual exterior differential operator $d$. With respect to the holomorphic structure $\partial''_E$, we obtain an operator $\partial'_h:  \mathscr{A}^0(E) \rightarrow  \mathscr{A}^{1,0}(E)$ of type $(1,0)$ by the equation
\begin{align*}
	\partial h \cdot h^{-1} = (\partial'_{h})^{\rm T} + {\rm Ad}(h)\partial''_E.
\end{align*}
With respect to the above information, the curvature $F_h$ is defined as $(D_h)^2$, where $D_h=\partial'_h + \partial''_E$ is the Chern connection. We will further comment on that later in \S\ref{sect ana Higgs and conn}.

\begin{defn}\label{defn adapt metric}
For a pair $(X, \boldsymbol D)$ as above, equip each puncture $x \in \boldsymbol D$ with a rational weight $\theta_x$ and denote by $\boldsymbol\theta=\{\theta_x, x \in \boldsymbol D\}$ the set of these weights. Moreover, for each puncture $x \in \boldsymbol D$, fix a local coordinate $z$ around $x$. We say that $h$ is a \emph{$\boldsymbol\theta$-adapted metric} if it satisfies the following conditions
\begin{itemize}
	\item $h$ is a hermitian metric;
	\item for any point $x \in \boldsymbol D$, the metric $h$ can be written as
    \begin{align*}
		h_0 \cdot |z|^{2 \theta_x} e^c
	\end{align*}
	under some suitable trivialization $e$ of $E$ around $x$, where $h_0$ is the standard constant metric (with respect to the trivialization $e$), multiplied
    by a scalar factor $e^c$ describing perturbation;
	\item for the data described above, we have ${\rm Ad}(|z|^{2 \theta_x})c = o (\ln |z|)$ and the curvature of the associated connection to $h$ is in $L^1$.
	\end{itemize}
\end{defn}

\begin{rem}\label{rem notation herm}
	The associated curvature $F_h$ of a $\boldsymbol\theta$-adapted metric $h$ is $L^1$, namely
	\begin{align*}
	\| F_h \|_{L^1}:= \int_{X_{\boldsymbol{D}}} |F_h|_{g,h} < \infty
	\end{align*}
	where $|F_h|_{g,h}^2:=( F_h,F_h )_{g,h}$ and $( \cdot,\cdot )_{g,h}$ stands for the hermitian inner product on the space of bundle-valued forms induced from the Riemannian metric $g$ and the hermitian metric $h$. From now on we use the notation $(\cdot,\cdot)_h$ if there is no ambiguity. In particular, on the space of sections of bundles (i.e. bundle-valued 0-forms), we apply the notation $\langle\cdot,\cdot\rangle_{h}$. This notation will be used in \S\ref{subsect Kob-Hitchin}.
\end{rem}

The data $\boldsymbol\theta$ is usually called the \emph{jump value} or \emph{weight} of $h$ with respect to the trivialization $e$. An important fact is that the jump value of a given hermitian metric at a puncture depends on the trivialization we choose.

\begin{defn}\label{defn herm G-bun}
Given a collection of weights $\boldsymbol\theta$, a \emph{$\boldsymbol\theta$-adapted $G$-bundle} is a metrized $G$-bundle $(E,\partial''_E,h)$, of which the metric $h$ is $\boldsymbol\theta$-adapted.
\end{defn}

Given a $\boldsymbol\theta$-adapted metric $h$, for each point $x\in \boldsymbol D$, we consider the following holomorphic map
\begin{align*}
	\xi: \mathbb{D}^{*}_x \rightarrow G,
\end{align*}
where $\mathbb{D}^{*}_x$ is a punctured disc around $x$, such that
\begin{align*}
	{\rm Ad}(z^{\theta_x  }) \xi(z) :=  z^{\theta_x} \xi(z) z^{-\theta_x}
\end{align*}
is a well-defined holomorphic morphism $\mathbb{D}_x \rightarrow G$. Considering the set of all such holomorphic maps on $\mathbb{D}^{*}_x$, it gives the parahoric group $\mathcal{G}_{\theta_x}$ (see Definition \ref{defn parahoric_group} and Lemma \ref{lem two defn parah grp}). With respect to the above discussion, $G_{X_{\boldsymbol D}}:=G \times X_{\boldsymbol D}$ is a group scheme over $X_{\boldsymbol D}$ and $\mathcal{G}_{\theta_x}$ is a group scheme on $\mathbb{D}_x$. By taking appropriate transition functions
\begin{align*}
	\Theta_{G,\theta_x}: \mathcal{G}_{\theta_x} |_{\mathbb{D}_x \cap X_{\boldsymbol D}} \rightarrow G_{X_{\boldsymbol D}}|_{X_{\boldsymbol D} \cap \mathbb{D}_x},
\end{align*}
we get a scheme over $X$, which is exactly the parahoric group scheme $\mathcal{G}_{\boldsymbol\theta}$.
	
Let $E$ be a $\boldsymbol\theta$-adapted $G$-bundle on $X_{\boldsymbol D}$. Similarly, for each puncture $x \in \boldsymbol D$, we can consider the set $E_{\theta_x}$ of holomorphic sections
\begin{align*}
	\xi: \mathbb{D}^{*}_x \rightarrow E
\end{align*}
such that
\begin{align*}
	{\rm Ad}(z^{\theta_x  }) \xi(z) =  z^{\theta_x} \xi(z) z^{-\theta_x}
\end{align*}
is a well-defined holomorphic morphism on $\mathbb{D}_x$. It is easy to check that $E_{\theta_x}$ is a $\mathcal{G}_{\theta_x}$-torsor on $\mathbb{D}_x$. Now we have a holomorphic $G$-bundle $E$ on $X_{\boldsymbol D}$, and a $\mathcal{G}_{\theta_x}$-torsor $E_{\theta_x}$ on $\mathbb{D}_x$ for each $x \in \boldsymbol D$. Denote by
\begin{align*}
	\rho: G_{X_{\boldsymbol D}} \times_{X_{\boldsymbol D}} E =  G \times E \rightarrow E, \quad \rho_x: \mathcal{G}_{\theta_x} \times_{\mathbb{D}_x} E_{\theta_x} \rightarrow E_{\theta_x}
\end{align*}
the corresponding actions of group schemes, where $G_{X_{\boldsymbol D}} =  T \times X_{\boldsymbol D}$. If we want to construct a parahoric $\mathcal{G}_{\boldsymbol\theta}$-torsor $\mathcal{E}$ on $X$ with respect to the above local data, it is equivalent to finding transition functions
\begin{align*}
	\Theta_x: E_{\theta_x}|_{\mathbb{D}_x \cap X_{\boldsymbol D}} \rightarrow E|_{X_{\boldsymbol D} \cap \mathbb{D}_x}
\end{align*}
such that the following diagram commutes
\begin{center}
\begin{tikzcd}
	\mathcal{G}_{\theta_x} \times_{\mathbb{D}_x} E_{\theta_x} \arrow[r, "\rho_x"] \arrow[d, "\Theta_{G, \theta_x} \times \Theta_x"]& E_{\theta_x} \arrow[d, "\Theta_{x}"] \\
	G \times E \arrow[r, "\rho"] & E
\end{tikzcd}
\end{center}
over the intersection $X_D \cap \mathbb{D}_x$. Fixing a trivialization $e_x$ of $E_{\theta_x}$, the transition function $\Theta_x$ is uniquely determined by the value $\Theta_x(e_x)$, which is an element in $G(K)$. From \S\ref{sect parah grp}, we know that parahoric $\mathcal{G}_{\boldsymbol\theta}$-torsors are parameterized by
\begin{align*}
	\left[ \prod_{x \in \boldsymbol D} G_{\theta_x}(K) \backslash \prod_{x \in \boldsymbol D} G(K) / G(\mathbb{C}[X_{\boldsymbol D}])  \right].
\end{align*}
This fact shows that the choice of transition functions $\Theta_x$ is not unique. Thus, we choose the identity element in this double coset and denote by $\mathcal{E}$ the corresponding parahoric $\mathcal{G}_{\boldsymbol\theta}$-torsor on $X$, which is defined by the local data $E$ on $X_{\boldsymbol D}$ and $E_{\theta_x}$ on $\mathbb{D}_x$.
	
With respect to the above construction, we define a functor
\begin{align*}
	\Xi : \mathcal{C}(X_{\boldsymbol D}, G, \boldsymbol\theta) \rightarrow \mathcal{C}(X,\mathcal{G}_{\boldsymbol\theta}),
\end{align*}
where $\mathcal{C}(X_{\boldsymbol D},G,\boldsymbol\theta)$ is the category of $\boldsymbol\theta$-adapted $G$-bundles on $X_{\boldsymbol D}$, and $\mathcal{C}(X,\mathcal{G}_{\boldsymbol\theta})$ is the category of parahoric $\mathcal{G}_{\boldsymbol\theta}$-torsors on $X$.
	
\begin{rem}\label{rem_functor}
In \cite[\S 3]{Simp}, Simpson constructed a functor from the category of adapted bundles on $X_{\boldsymbol D}$ to the category of parabolic bundles on $X$, which preserves duals, direct sums, determinants and tensor products.  The key point is that Simpson did not fix a particular parabolic structure for the category of parabolic bundles.
		
In this subsection, the above functor $\Xi$ is defined for a given collection of weights $\boldsymbol\theta$. Furthermore, this functor can be extended naturally to the category of parahoric torsors over $X$, where we do not fix the parahoric group scheme (or the weights), thus
\begin{align*}
	\Xi : \{ \text{ category of adapted principal bundles on $X_{\boldsymbol D}$ }   \} \rightarrow \{ \text{ category of parahoric torsors on $X$ } \}.
\end{align*}
If we consider the larger categories (where we do not fix at least the weights), then we can discuss whether the functor preserves tensor products, duals and other properties similarly. Furthermore, this functor also preserves reductions of group structures, and this property can be obtained easily in \S\ref{subsect reduct}.
\end{rem}
	
\begin{exmp}\label{exmp line bundle}
In this example, we consider the special case when $G={\rm GL}_1(\mathbb{C}) \cong \mathbb{C}^{*}$. In this special case, the holomorphic $G$-bundles are exactly line bundles. Fixing a collection of weights $\boldsymbol\theta$, let $L$ be a holomorphic ${\rm GL}_1(\mathbb{C})$-bundle on $X_{\boldsymbol D}$. Consider the holomorphic morphism $\xi: \mathbb{D}_x^{*} \rightarrow L$ such that ${\rm Ad}(z^{\theta_x}) \xi(z)$ is a holomorphic morphism $\mathbb{D}_x \rightarrow L$. Since $G={\rm GL}_1(\mathbb{C})$, we have
\begin{align*}
	{\rm Ad}(z^{\theta_x}) \xi(z)=\xi(z).
\end{align*}
Therefore, there is a natural extension of $L$ to a line bundle $\mathcal{L}$ on $X$ \cite[\S 10]{Simp1988}. In this special case, we find that the resulting torsor (or line bundle) $\mathcal{L}$ does not depend on the choice of weights $\boldsymbol\theta$, and the parahoric language cannot distinguish the pair $(L,h)$, where $L$ is a ${\rm GL}_1(\mathbb{C})$-bundle and $h$ is a $\boldsymbol\theta$-adapted metric. Therefore, it is reasonable to still use the pair $(\mathcal{L},\boldsymbol\theta)$ for the corresponding object on $X$, and $\boldsymbol\theta$ is exactly regarded as ``parabolic weights" in the language of parabolic bundles (see \cite{KSZ20211,Meht,Yoko2} for instance). Furthermore, the property that the parabolic structure of line bundles only depends on the weights also implies this phenomenon.
\end{exmp}
	
\subsection{Parahoric Degree}\label{subsect alg parah deg}
Let $G$ be a connected complex reductive group. Let $\theta$ be a weight, and denote by $G_{\theta}(K) \subseteq G(K)$ the parahoric group corresponding to $\theta$. Recall that a parabolic subgroup $P$ (with Lie algebra $\mathfrak{p}$) of $G$ can be determined by a subset of roots $\mathcal{R}_P \subseteq \mathcal{R}$. We define the following parahoric group as a subgroup of $P(K)$
\begin{align*}
	P_\theta(K) := \langle T(R), U_r(z^{m_r(\theta)}R), r \in \mathcal{R}_P \rangle.
\end{align*}
Denote by $\mathcal{P}_\theta$ the corresponding group scheme on $\mathbb{D}=\text{Spec}(R)$. This construction is a special case of the one in \cite{HaiRap}. Furthermore, let $ev: G_\theta (K) \rightarrow G$ be the evaluation map, and denote by $P_\theta \subset G$ the parabolic subgroup, of which the inverse image under the evaluation map is $P_\theta(K)$.
	
Now we consider the global picture. We define the group scheme $\mathcal{P}_{\boldsymbol\theta}$ on $X$ by gluing the local data
\begin{align*}
	\mathcal{P}_{\boldsymbol\theta}|_{\mathbb{D}_x} \cong P \times X_{\boldsymbol D}, \quad \mathcal{P}_{\boldsymbol\theta}|_{\mathbb{D}_x} \cong \mathcal{P}_{\theta_x}, x \in \boldsymbol D.
\end{align*}
By \cite[Lemma 3.18]{ChGP}, the group scheme $\mathcal{P}_{\boldsymbol\theta}$ is a smooth affine group scheme of finite type, flat over $X$ and we have that $\mathcal{P}_{\boldsymbol\theta} \subseteq \mathcal{G}_{\boldsymbol\theta}$.
	
Now let $\mathcal{E}$ be a parahoric $\mathcal{G}_{\boldsymbol\theta}$-torsor on $X$. Let $\varsigma \in \Gamma(X,\mathcal{E}/\mathcal{P}_{\boldsymbol\theta})$ be a section (see Remark \ref{rem conv holo and alg} for the convention), which is called a \emph{reduction of structure group}. Denote by $\mathcal{E}_{\varsigma}$ the pullback of the diagram
\begin{center}
	\begin{tikzcd}
	\mathcal{E}_{\varsigma} \arrow[r, dotted] \arrow[d, dotted] & \mathcal{E} \arrow[d]\\
	X \arrow[r, "\varsigma"]  & \mathcal{E}/\mathcal{P}_{\boldsymbol\theta}.
	\end{tikzcd}
\end{center}
Then, the pullback $\mathcal{E}_{\varsigma}$ is a parahoric $\mathcal{P}_{\boldsymbol\theta}$-torsor on $X$. Let $\kappa: \mathcal{P}_{\boldsymbol\theta} \rightarrow \mathbb{G}_m$ be a morphism of group schemes over $X$, and we call it a \emph{character} of $\mathcal{P}_{\boldsymbol\theta}$. Furthermore, there is a one-to-one correspondence between characters of $\mathcal{P}_{\boldsymbol\theta}$ and characters of $P$:
\begin{lem}\cite[Lemma 4.2]{KSZparh}\label{lem char}
It holds that 
    \begin{align*}
	{\rm Hom}(\mathcal{P}_{\boldsymbol\theta},\mathbb{G}_m) = {\rm Hom}(P,\mathbb{C}^*).
	\end{align*}
\end{lem}
Denote by $\chi: P \rightarrow \mathbb{C}^*$ the corresponding character of $\kappa$. We define the following inner product
\begin{align*}
	\langle \theta, \kappa \rangle := \langle \theta, \chi \rangle,
\end{align*}
for any weight $\theta$. We return to the parahoric $\mathcal{P}_{\boldsymbol\theta}$-torsor $\mathcal{E}_{\varsigma}$. Given a character $\kappa$, we define a line bundle $\kappa_* (\mathcal{E}_{\varsigma})$ on $X$, and denote it by $L(\varsigma,\kappa)$. As a special case, let $P=G$. If $\varsigma: X \rightarrow \mathcal{E} / \mathcal{G}_{\boldsymbol\theta}$ is trivial, then the pushforward $L(\kappa):=\kappa_* \mathcal{E}$ directly gives a line bundle on $X$. In \cite{KSZparh}, the authors introduced a notion of algebraic degree of parahoric torsors as follows:
\begin{defn}\label{defn alg deg}
With respect to the above data, we define the \emph{parahoric degree} of a parahoric $\mathcal{G}_{\boldsymbol\theta}$-torsor $\mathcal{E}$ with respect to a given reduction $\varsigma$ and a character $\kappa$ as follows
\begin{align*}
	parh\deg \mathcal{E}(\varsigma,\kappa)=\deg L(\varsigma,\kappa)+ \langle  \boldsymbol\theta, \kappa \rangle,
\end{align*}
where $\langle  \boldsymbol\theta, \kappa \rangle:=\sum_{x\in \boldsymbol{D}} \langle  \theta_x, \kappa \rangle$. If $\varsigma: X \rightarrow \mathcal{E} / \mathcal{G}_{\boldsymbol\theta}$ is trivial, we define
\begin{align*}
	parh\deg \mathcal{E}(\kappa)=\deg L(\kappa)+ \langle  \boldsymbol\theta, \kappa \rangle.
\end{align*}
Sometimes, we say the degree of a parahoric torsor for simplicity.
\end{defn}
	
\subsection{Analytic Degree}\label{subsect ana deg on bundle}
Let $(E,\partial''_E,h)$ be a $\boldsymbol\theta$-adapted $G$-bundle on $X_{\boldsymbol D}$. Recall that the metric $h$ is considered as a section of $E/H$, where $H$ is the maximal compact subgroup of $G$.	Let $P \subseteq G$ be a parabolic subgroup of $G$, and denote by $\sigma: X_{\boldsymbol D} \rightarrow E/P$ a reduction of structure group. Similarly, a reduction of structure group can be considered as a section of $E/P$. Denote by $E_{\sigma}$ the pullback in the following diagram
\begin{center}
\begin{tikzcd}
	E_{\sigma} \arrow[r, dotted] \arrow[d, dotted] & E \arrow[d]\\
	X_{\boldsymbol D} \arrow[r, "\sigma"]  & E/P,
\end{tikzcd}
\end{center}
which is a $P$-bundle on $X_{\boldsymbol D}$. For convenience, we use the same notation $\sigma: E_\sigma \rightarrow E$ for the induced morphism. Then, taking a character $\chi: P \rightarrow \mathbb{C}^{\ast}$, we get a line bundle $\chi_* E_{\sigma}$ on $X_{\boldsymbol D}$ and denote it by $L(\sigma,\chi)$. Similar to the parahoric case, if $\sigma: X_{\boldsymbol D} \rightarrow E/G$ is trivial, then the pushforward $L(\chi):=\chi_* E$ directly gives a line bundle on $X_{\boldsymbol D}$.
	
Let $\partial''_E$ be a holomorphic structure on $E$. Given a reduction of structure group $\sigma$, we have an induced operator $\partial''_{E_\sigma}:\mathscr{A}^{0}(E_\sigma) \rightarrow \mathscr{A}^{1}(E_\sigma)$ given by
\begin{align*}
	\partial''_{E_\sigma}:=\sigma^* \partial''_E.
\end{align*}
If the reduction of structure group $\sigma$ is a (not necessarily holomorphic) section of $E/P$, the induced operator $\partial''_{E_\sigma}$ may not give a holomorphic structure of the $P$-bundle $E_{\sigma}$. More precisely, it is not of type $(0,1)$. Therefore, we would like to take a \emph{holomorphic reduction of structure group}, which is a holomorphic section of $E/P$. Then, the $P$-bundle $E_{\sigma}$ has a natural holomorphic structure $\partial''_{E_\sigma}$ induced from $\partial''_E$.
	
Next, we want to find a metric $h_{\sigma}: X_{\boldsymbol D} \rightarrow E_{\sigma} / (H \cap P)$ induced by $h$ on $E_{\sigma} $. Let $\langle H, P \rangle$ be the group generated by $H$ and $P$. Then, we have a diagram
\begin{center}
\begin{tikzcd}
	X_{\boldsymbol D} \arrow[rd, dotted] \arrow[rrd, bend left, "h"] \arrow[rdd, bend right, "{\rm id}"] & & \\
	& E/H \times_{E/\langle H,P \rangle} X_{\boldsymbol D} \arrow[r] \arrow[d] & E/H \arrow[d]\\
	& X_{\boldsymbol D} \arrow[r, "\sigma"]  & E/\langle H,P \rangle,
\end{tikzcd}
\end{center}
where the map at the bottom is actually the composition of the maps
\begin{align*}
	X_{\boldsymbol D} \xrightarrow{\sigma} E / P \rightarrow E/\langle H,P \rangle
\end{align*}
and we use the same notation $\sigma$ to emphasize that it is induced by the reduction of structure group $\sigma$. Since $H$ is the maximal compact subgroup of $G$ and $P$ is a parabolic subgroup of $G$, the group $\langle H, P \rangle$ is exactly $G$. Thus, the above diagram commutes and there exists a unique morphism
\begin{align*}
	X_{\boldsymbol D} \rightarrow E/H \times_{E/\langle H,P \rangle} X_{\boldsymbol D}.
\end{align*}
Furthermore, we have
\begin{align*}
h_\sigma: X_{\boldsymbol D} \rightarrow E/H \times_{E/\langle H,P \rangle} X_{\boldsymbol D} \cong E \times_{E/P} X_{\boldsymbol D} / (H \cap P)= E_{\sigma} / (H \cap P).
\end{align*}
This morphism gives a well-defined metric on the $P$-bundle $E_{\sigma}$, and denote by $h_\sigma$ the induced metric. We define the curvature $F_{h_\sigma}:=(D_{h_\sigma})^2$, where $D_{h_\sigma}=\partial'_{h_\sigma} + \partial''_{E_\sigma}$.
	
Composing with a character $\chi: P \rightarrow \mathbb{C}^*$, we induce a metric $\chi_* h_\sigma$ on the line bundle $\chi_* E_{\sigma}$. The same argument also holds for the operator $\partial'_{h_\sigma}$ and the curvature $F_{h_\sigma}$. Furthermore, $\chi_* \partial'_{h_\sigma}$ and $\chi_* F_{h_\sigma}$ are the induced operator and curvature on $\chi_* E_{\sigma}$ respectively.
	
\begin{defn}\label{defn ana deg}
Let $(E,\partial''_E,h)$ be a $\boldsymbol\theta$-adapted $G$-bundle on $X_{\boldsymbol D}$. For $P \subseteq G$ a parabolic subgroup of $G$, let $\sigma: X_{\boldsymbol D} \rightarrow E/P$ be a holomorphic reduction of structure group and we take a character $\chi: P \rightarrow \mathbb{C}^{\ast}$. With respect to the above data, we define the \emph{analytic degree} of $(E,\partial''_E,h)$ as
\begin{equation*}\label{analytic degree}
	\deg^{\rm an} E(h, \sigma, \chi):= \frac{\sqrt{-1}}{2\pi} \int_{X_{\boldsymbol{D}}} \chi_* F_{h_\sigma}.
\end{equation*}
If $\sigma: X \rightarrow E/G$ is trivial, we define
\begin{align*}
\deg^{\rm an} E(h, \chi):=\frac{\sqrt{-1}}{2\pi} \int_{X_{\boldsymbol{D}}} \chi_* F_{h}.
\end{align*}
Sometimes, we say the degree of a metrized $G$-bundle for simplicity.
\end{defn}

\begin{lem}[Chern--Weil Formula]\label{lem Chern-Weil}
Let $(E,\partial''_E,h)$ be a $\boldsymbol\theta$-adapted $G$-bundle on $X_{\boldsymbol D}$. For  $P \subseteq G$ a parabolic subgroup of $G$, let $\sigma: X_{\boldsymbol D} \rightarrow E/P$ be a holomorphic reduction of structure group and take $\chi: P \rightarrow \mathbb{C}^{\ast}$ a character. Then, the following identity holds
\begin{align*}
	\int_{X_{\boldsymbol{D}}} \chi_* F_{h_\sigma} = \int_{X_{\boldsymbol{D}}} \chi_* (\sigma^* F_h).
\end{align*}
\end{lem}
	
\begin{proof}
It is easy to check that
\begin{align*}
	\partial h_\sigma \cdot h_\sigma^{-1} = \sigma^* (\partial h \cdot h^{-1}).
\end{align*}
Thus, we have
\begin{align*}
F_{h_\sigma} = (\partial_{h_\sigma} + \partial''_{E_\sigma})^2 = \left( \sigma^* (\partial_h + \partial''_E) \right)^2 = \left( (\partial_h + \partial''_E) \circ \sigma \right)^2.
\end{align*}
Since $\sigma$ is a holomorphic section, we have $\partial''_{E}(\sigma)=0$. Therefore,
\begin{align*}
	F_{h_\sigma} = \sigma^* F_h.
\end{align*}
This finishes the proof of the lemma.
\end{proof}
	
\begin{rem}
It is important for our considerations that we take a holomorphic reduction of structure group $\sigma: X_{\boldsymbol D} \rightarrow E/P$. When $\sigma$ is not holomorphic, the induced operator $\partial''_{E_\sigma}:=\sigma^* \partial''_E$ is not necessarily of type $(0,1)$ and an additional term is expected in the Chern--Weil formula of Lemma \ref{lem Chern-Weil}. In  \cite[Lemma 2.13]{BGM}, the authors obtained such a Chern--Weil formula under certain conditions.
\end{rem}
	
\begin{exmp}\label{exmp char}
We consider a special hermitian metric $h=|z|^{2 \theta}$ of the trivial $G$-bundle on the punctured disc $\mathbb{D}^{*}$, where $\theta \in \mathfrak{t}_{\mathbb{Q}}$ is a weight. Then, we have
\begin{align*}
	\partial h \cdot h^{-1} = \theta \frac{dz}{z}.
\end{align*}
Let $\sigma: \mathbb{D}^{*} \rightarrow G/P \times \mathbb{D}^{*}$ be the trivial reduction of structure group, which is obviously holomorphic. Since $\theta$ is included in the Lie algebra $\mathfrak{p}$ of $P$, we have the same formula for the induced metric $h_\sigma$
\begin{align*}
	\sigma^* (\partial h \cdot h^{-1}) = \partial h_\sigma \cdot h_\sigma^{-1} = \theta \frac{dz}{z}.
\end{align*}
Let $\chi$ be a character of $P$. Then,
\begin{align*}
	\chi_* (\theta \frac{dz}{z})=\langle \theta, \chi \rangle \frac{dz}{z}.
\end{align*}
\end{exmp}
	
\subsection{Reduction of structure group}\label{subsect reduct}
In this subsection, we study the relation of the reductions of structure group between parahoric torsors and metrized principal bundles. Given an arbitrary holomorphic reduction of structure group $\sigma:X_{\boldsymbol D} \rightarrow E/P$, its extension to $X$ is not unique, and furthermore, it may not be extended to a reduction $X \rightarrow \mathcal{E}/\mathcal{P}_{\boldsymbol\theta}$. Therefore, we want to consider a special type of reduction, the \emph{$\boldsymbol\theta$-adapted holomorphic reduction of structure group}. Here is the definition:

\begin{defn}
	If a holomorphic reduction of structure group $\sigma: X_{\boldsymbol D} \rightarrow E/P$ satisfies the condition that for each puncture $x \in \boldsymbol D$, we have
	\begin{align*}
	z^{\theta_x} \cdot \sigma(z) \cdot z^{-\theta_x} \text{ is bounded as $z$ approaches zero,}
	\end{align*}
	where $z$ is the local coordinate of a punctured disc $\mathbb{D}^{*}$ around $x$, then we say that $\sigma$ is a \emph{$\boldsymbol\theta$-adapted holomorphic reduction of structure group}. Given a $\boldsymbol\theta$-adapted $G$-bundle $(E,\partial''_E,h)$, denote by $\Gamma_{\boldsymbol\theta}(X_{\boldsymbol D},E/P)$ the set of $\boldsymbol\theta$-adapted holomorphic reductions of structure group (to $P$).
\end{defn}

\begin{lem}\label{lem red grp struc}
Let $(E,\partial''_E,h)$ be a $\boldsymbol\theta$-adapted $G$-bundle on $X_{\boldsymbol D}$, and denote by $\mathcal{E}$ the corresponding parahoric $\mathcal{G}_{\boldsymbol\theta}$-torsor on $X$. For a fixed parabolic group $P$, there is a one-to-one correspondence
\begin{align*}
	\Gamma_{\boldsymbol\theta}(X_{\boldsymbol D},E/P) = \Gamma(X, \mathcal{E} / \mathcal{P}_{\boldsymbol\theta}).
\end{align*}
\end{lem}
	
\begin{proof}
One direction is clear, and we only have to show that given a $\boldsymbol\theta$-adapted holomorphic reduction of structure group $\sigma$, we can get an element $\varsigma: X \rightarrow \mathcal{E} / \mathcal{P}_{\boldsymbol\theta}$. It is enough to work around a puncture $x \in \boldsymbol{D}$, and then the proof of this lemma is exactly the same as the one for Lemma \ref{lem two defn parah grp}.
\end{proof}
	
\subsection{Equivalence of Analytic Degree and Algebraic (Parahoric) Degree}\label{subsect equiv ana and alg}
Let $(E,\partial''_E,h)$ be a $\boldsymbol\theta$-adapted $G$-bundle on $X_{\boldsymbol D}$. Taking a parabolic subgroup $P \subseteq G$, let $\chi$ be a character of $P$ and let $\sigma$ be a $\boldsymbol\theta$-adapted holomorphic reduction of structure group. Denote by $\mathcal{E}$ the corresponding parahoric $\mathcal{G}_{\boldsymbol\theta}$-torsor of $E$ given by the functor $\Xi$ in \S\ref{subsect functor Xi}, and let $\kappa$ and $\varsigma$ be the corresponding elements to $\chi$ and $\sigma$ respectively by Lemmas \ref{lem char} and \ref{lem red grp struc}.

\begin{prop}\label{prop analytic=algebraic degree}
The following identity holds
\begin{align*}
	\deg^{\rm an} E(h, \sigma, \chi)=parh\deg \mathcal{E}(\varsigma,\kappa).
\end{align*}
\end{prop}
	
\begin{proof}
First, note that since the metric $h$ is $\boldsymbol\theta$-adapted, the analytic degree as defined above is absolutely convergent. Then, we have the following equation
\begin{align*}
\frac{\sqrt{-1}}{2\pi} \int_{X} \chi_* F_{h_\sigma}
=\frac{\sqrt{-1}}{2\pi} \int_{ X_\varepsilon }\chi_* F_{h_\sigma} + \frac{\sqrt{-1}}{2\pi} \sum_{x \in {\boldsymbol D}}\int_{ \mathbb{D}_{x,\varepsilon} } \chi_* F_{h_\sigma},
\end{align*}
where $\mathbb{D}_{x,\varepsilon}$ is the disc around $x$ with radius $\varepsilon$ and $X_\varepsilon$ is the complement of $\bigcup\limits_{x \in \boldsymbol{D}} \mathbb{D}_{x,\varepsilon}$ in $X_{\boldsymbol D}$. By the definition of line bundles $L(\varsigma,\kappa)$ and $L(\sigma,\chi)$ and Example \ref{exmp line bundle}, the left hand side of the equation is
\begin{align*}
\frac{\sqrt{-1}}{2\pi} \int_{X} \chi_* F_{h_\sigma} = \deg L(\sigma,\chi) = \deg L(\varsigma,\kappa).
\end{align*}
Now we consider the right hand side of the equation. Letting $\varepsilon$ go to zero, the first integral
\begin{align*}
\lim_{\varepsilon \rightarrow 0}\frac{\sqrt{-1}}{2\pi} \int_{X_{\varepsilon}}\chi_* F_{h_\sigma} = \frac{\sqrt{-1}}{2\pi} \int_{X_{\boldsymbol{D}}}\chi_* \sigma^* F_{h}  = \deg^{\rm an}E (h, \sigma, \chi)
\end{align*}
is exactly the analytic degree. Therefore, we only have to show that
\begin{align*}
\lim_{\varepsilon \rightarrow 0} \frac{\sqrt{-1}}{2\pi} \sum_{x \in \boldsymbol D} \int_{ \mathbb{D}_{x,\varepsilon} } \chi_*(\sigma^* F_h)= - \langle  \boldsymbol\theta, \kappa \rangle.
\end{align*}
To prove this equality, it is enough to work locally around a puncture $x \in \boldsymbol{D}$. Since $h$ is a $\theta_x$-adapted metric, when $\varepsilon$ goes to zero, we can assume $h = |z|^{2 \theta_x}$. Then,
\begin{align*}
\lim_{\varepsilon \rightarrow 0} \frac{\sqrt{-1}}{2\pi} \int_{ \mathbb{D}_{x,\varepsilon} } \chi_*(\sigma^* F_h) &= \lim_{\varepsilon \rightarrow 0} \frac{\sqrt{-1}}{2\pi} \int_{ \partial \mathbb{D}_{x,\varepsilon} } \chi_* \sigma^* (h^{-1}\partial h)\\
&=\lim_{\varepsilon \rightarrow 0}\frac{\sqrt{-1}}{2\pi}\int_{ \partial \mathbb{D}_{x,\varepsilon} } \chi_* \sigma^* (\theta_x \frac{dz}{z})\\
&=\lim_{\varepsilon \rightarrow 0}\frac{\sqrt{-1}}{2\pi}\int_{ \partial \mathbb{D}_{x,\varepsilon} }  \chi_* \theta_x \frac{dz}{z}\\
&=\lim_{\varepsilon \rightarrow 0}\frac{\sqrt{-1}}{2\pi}\int_{ \partial \mathbb{D}_{x,\varepsilon} } \langle  \theta_x, \chi \rangle \frac{dz}{z}\\
&= - \langle  \theta_x, \chi \rangle = - \langle \theta_x, \kappa \rangle.
\end{align*}
This finishes the proof of this proposition.
\end{proof}

\subsection{Stability Conditions}
Ramanathan introduced a stability condition for $G$-bundles and constructed their moduli spaces \cite{Rama1975,Rama19961,Rama19962}. Since then, this stability condition has been widely used for studying objects related to $G$-Higgs bundles (see for instance \cite{BruOt,Schmitt}). Recently, the authors in \cite{BGM} introduced an analogous stability condition for parabolic $G$-bundles and studied the Kobayashi--Hitchin correspondence in their case. In this subsection, we review the (algebraic) stability condition for parahoric $\mathcal{G}_{\boldsymbol\theta}$-torsors (Definition \ref{defn alg stab cond}) introduced in \cite{KSZparh}, which we called $R$-stability, and define (analytic) $R_h$-stability for $G$-bundles (Definition \ref{defn ana stab cond}), where $h$ is a given metric. We prove that the two are equivalent and this is the main result in this subsection (Proposition \ref{prop stab of parah tor}).

Before we introduce the stability condition, we refer the reader to \cite{BGM} for the definition of \emph{anti-dominant characters}, and to \cite{Rama1975,Rama19961} for \emph{dominant characters}. Furthermore, a character of a group scheme  $\mathcal{P}_{\boldsymbol\theta}$ is \emph{anti-dominant} if the corresponding character of $P$ (by Lemma \ref{lem char}) is anti-dominant.
	
\begin{defn}\label{defn alg stab cond}
A parahoric $\mathcal{G}_{\boldsymbol\theta}$-torsor $\mathcal{E}$ on a curve $X$ is called \textit{$R$-stable} (resp. \textit{$R$-semistable}) if for
\begin{itemize}
\item any proper parabolic subgroup $P$ of $G$,
\item any reduction of structure group $\varsigma:X\to \mathcal{E} /\mathcal{P}_{\boldsymbol\theta}$,
\item any nontrivial anti-dominant character $\kappa: \mathcal{P}_{\boldsymbol\theta} \rightarrow \mathbb{G}_m$, which is trivial on the center of $\mathcal{P}_{\boldsymbol\theta}$,
\end{itemize}
one has
\begin{align*}
	parh\deg \mathcal{E}(\varsigma,\kappa) > 0\quad (\mbox{resp. }\ge 0).
\end{align*}
\end{defn}
	
\begin{defn}\label{defn ana stab cond}
A $\boldsymbol\theta$-adapted $G$-bundle $(E,\partial''_E,h)$ on $X_{\boldsymbol D}$ is \textit{$R_h$-stable} (resp. \textit{$R_h$-semistable}) if for
\begin{itemize}
\item any proper parabolic subgroup $P$ of $G$,
\item any $\boldsymbol\theta$-adapted holomorphic reduction of structure group $\sigma:X_{\boldsymbol D} \to E/P$,
\item any nontrivial anti-dominant character $\chi: P \rightarrow \mathbb{C}^*$, which is trivial on the center of $P$,
\end{itemize}
one has
\begin{align*}
	\deg^{\rm an} E(h, \sigma, \chi) > 0\quad (\mbox{resp. }\ge 0).
\end{align*}
\end{defn}
	
\begin{prop}\label{prop stab of parah tor}
Let $(E,\partial''_E,h)$ be a $\boldsymbol\theta$-adapted $G$-bundle on $X_{\boldsymbol D}$, and let $\mathcal{E}$ be the corresponding parahoric $\mathcal{G}_{\boldsymbol\theta}$-torsor on $X$. Then, $E$ is $R_h$-stable (resp. $R_h$-semistable) if and only if $\mathcal{E}$ is $R$-stable (resp. $R$-semistable).
\end{prop}
	
\begin{proof}
This is a direct result of Lemma \ref{lem char}, Lemma \ref{lem red grp struc} and Proposition \ref{prop analytic=algebraic degree}.
\end{proof}
	
\subsection{$R_\mu$-stability Conditions and Degree Zero Case}\label{subsect degree zero}
Although the $R_\mu$-stability condition is not considered in this paper, it is important to give the definition of a parahoric torsor of \emph{degree zero}. We first restate the definition of $R_\mu$-stability condition from \cite{KSZparh} for convenience:
\begin{defn}\label{Ralpha}
Fixing an element $\mu\in \mathfrak{t}$, a parahoric $\mathcal{G}_{\boldsymbol\theta}$-torsor $\mathcal{E}$ on $X$ is called \textit{$R_\mu$-stable} (resp. \textit{$R_\mu$-semistable}) if for
\begin{itemize}
\item any proper parabolic subgroup $P$ of $G$,
\item any reduction $\varsigma:X\to \mathcal{E}/\mathcal{P}_{\boldsymbol\theta}$,
\item any nontrivial anti-dominant character $\kappa: \mathcal{P}_{\boldsymbol\theta} \rightarrow \mathbb{G}_m$,
\end{itemize}
one has
\[parh\deg \mathcal{E}(\varsigma,\kappa) -\langle \mu,\kappa \rangle> 0\quad (\mbox{resp. }\ge 0).\]
\end{defn}

Definition \ref{defn alg deg} gives the definition of degree of a parahoric torsor $\mathcal{E}$. This definition is not a topological invariant for $\mathcal{E}$ because it not only depends on the torsor but also on the choice of $\varsigma$ and $\kappa$. However, we can find a canonical element $\mu \in \mathfrak{t}$ and take it as a topological invariant for a given parahoric torsor, which is similar to the one considered in \cite[Section 4.2]{BGM} for the study of solutions to the Hermite--Einstein equations.
\begin{prop}[Proposition 5.9 in \cite{KSZparh}]\label{prop mu top inv}
Let $E$ be a parahoric $\mathcal{G}_{\boldsymbol\theta}$-torsor. Let $\mathfrak{z}$ be the center of $\mathfrak{g}$. There exists a canonical choice of $\mu\in \mathfrak{z}$, depending on the topological type of $E$, such that $E$ is $R$-stable (resp. $R$-semistable) if and only if $E$ is $R_\mu$-stable ($R_\mu$-semistable).
\end{prop}

Instead of giving a proof of this proposition, we give the basic idea of finding the element $\mu$. We take the trivial reduction of structure group, and then consider characters of $G$. Clearly, the set of all characters of $G$ is equivalent to the set of all characters of $G^{\rm ab}:=G/[G,G]$. Thus, nontrivial characters of $G$ act nontrivially on the center of $G$. Thus, we can find a unique element $\mu \in \mathfrak{z}$ such that
\begin{align*}
parh\deg \mathcal{E}(\kappa)  = \langle \mu,\kappa \rangle.
\end{align*}

\begin{exmp}
We discuss the case $G={\rm GL}_n(\mathbb{C})$, and we refer to reader to \cite[\S 5.2]{KSZparh} for more details. It is well-known that a parahoric ${\rm GL}_n$-torsor $\mathcal{E}$ corresponds to a unique parabolic bundle $E_{\bullet}$, where $E$ is a vector bundle and the subscript $\bullet$ is for the parabolic structures. Then, the canonical element $\mu$ of $\mathcal{E}$ is exactly the element $\frac{par \deg E_{\bullet}}{{\rm rk} E} \cdot I$, where $I$ is the identity matrix. In the case of curves, the element $\mu$ includes the information of degree and rank, and therefore, gives the information of Hilbert polynomials. This is the reason why $\mu$ is a topological invariant of a given parahoric torsor.
\end{exmp}

\begin{defn}\label{defn degree zero}
With respect to the definition and proposition above, if the corresponding element $\mu$ of a parahoric torsor $\mathcal{E}$ is the trivial element in $\mathfrak{t}$, then we say that \emph{the parahoric torsor $\mathcal{E}$ is of degree zero}. Under the correspondence we defined in \S\ref{subsect functor Xi}, \emph{a $\boldsymbol\theta$-adapted $G$-bundle $E$ on $X_{\boldsymbol D}$ is of degree zero}, if the corresponding parahoric torsor is of degree zero.
\end{defn}

If $\mathcal{E}$ is a parahoric torsor of degree zero, then the $R_\mu$-stability condition of $\mathcal{E}$ is equivalent to the $R$-stability condition. Furthermore, for any character $\kappa$ of $\mathcal{G}_{\boldsymbol\theta}$, we have
\begin{align*}
parh\deg \mathcal{E}(\kappa) = 0.
\end{align*}

Although we do not know how to define a notion of ``degree" (as a numerical topological invariant) for general parahoric torsors, when quoting \emph{a parahoric torsor of degree zero} we shall mean the notion considered above. In this paper, we focus on the nonabelian Hodge correspondence, involving  connections and Higgs bundles of degree zero (on curves) as in \cite{Simp,Simp Higgs Local}. This is the reason why here we only care about the case of degree zero. When $\mu$ is the trivial element, the $R_\mu$-stability condition is exactly the same as the $R$-stability condition (see Definitions \ref{defn ana stab cond} and \ref{defn alg stab cond}). Furthermore, in \cite{BGM} the authors proved a version of the Kobayashi--Hitchin correspondence for parabolic $G$-Higgs bundles for a given topological invariant $\mu \in \mathfrak{t}$ \cite[Theorem 5.1]{BGM}; a polystable parabolic $G$-Higgs bundle in the terminology of \cite{BGM} corresponds to a parabolic local system if and only if the invariant $\mu$ is trivial.
	
\begin{rem}\label{rem_question_BGM}
For the case $G={\rm GL}_n$, in \cite[\S 5.2]{BGM}, the authors raised the question whether there is some relation between the $R$-stability condition and the stability condition given by the parabolic slope for parabolic bundles (see \cite{KSZ20211,KSZ20212} for instance). The question is very important, because this would allow one to generalize the classical case considered by Simpson in the nonabelian Hodge theory for ${\rm GL}_n$-bundles on noncompact curves \cite{Simp}. In the degree zero case, these two stability conditions are equivalent (see \cite[\S 5]{KSZparh}). This means that the nonabelian Hodge correspondence studied in this paper is a full generalization of Simpson's work to $G$-bundles.
\end{rem}
	
\subsection{Polystability}\label{subsect poly stab}
In this subsection, we introduce the definition of polystability for both parahoric torsors and metrized $G$-bundles in the case of degree zero. We refer the reader to \cite{Bis_poly} for the polystability conditions of principal bundles. The notion of polystability can be defined similarly for Higgs bundles and connections, which will be used in the proof of Theorem \ref{thm KH corr Higgs} in \S\ref{subsect Kob-Hitchin}.
	
Given a parabolic subgroup $P \subseteq G$, let $L$ be the Levi subgroup of $G$ with Lie algebra $\mathfrak{l}$. Clearly, $\theta$ is also a weight of $L$. Similar to the construction of $\mathcal{P}_{\boldsymbol\theta}$, we define the group scheme $\mathcal{L}_{\boldsymbol\theta}$ similarly and there is a natural projection $\mathcal{P}_{\boldsymbol\theta} \rightarrow \mathcal{L}_{\boldsymbol\theta}$ induced by $P \rightarrow L$. Let $\mathcal{E}$ be a parahoric $\mathcal{G}_{\boldsymbol\theta}$-torsor. Given a reduction of structure group $\varsigma: X \rightarrow \mathcal{E}/\mathcal{P}_{\boldsymbol\theta}$, denote by $\mathcal{E}_{\varsigma_L}$ the corresponding parahoric $\mathcal{L}_{\boldsymbol\theta}$-torsor, which is induced by the parahoric $\mathcal{P}_{\boldsymbol\theta}$-torsor $\mathcal{E}_{\varsigma}$.

\begin{defn}\label{defn stab poly}
A parahoric $\mathcal{G}_{\boldsymbol\theta}$-torsor $\mathcal{E}$ is \emph{$R$-polystable of degree zero}, if it is $R$-semistable and there exists a proper parabolic subgroup $P \subseteq G$ and a reduction of structure group $\varsigma: X \rightarrow \mathcal{E}/\mathcal{P}_{\boldsymbol\theta}$, satisfying the conditions
\begin{itemize}
\item for any anti-dominant character $\kappa$, the equality
\begin{align*}
	parh \deg \mathcal{E}(\varsigma,\kappa) = 0
\end{align*}
holds;
\item the parahoric $\mathcal{L}_{\boldsymbol\theta}$-torsor $\mathcal{E}_{\varsigma_L}$ is $R$-stable.
\end{itemize}

Let $(E,\partial''_E,h)$ be a $\boldsymbol\theta$-adapted $G$-bundle on $X_{\boldsymbol D}$. It is \emph{$R_h$-polystable of degree zero} if it is $R_h$-semistable and there exists a proper parabolic subgroup $P \subseteq G$ and a $\boldsymbol\theta$-adapted holomorphic reduction of structure group $\sigma: X \rightarrow E/P$, satisfying the conditions
\begin{itemize}
\item for any anti-dominant character $\chi$, the equality
\begin{align*}
	\deg^{\rm an} E(h,\sigma,\chi) = 0
\end{align*}
holds;
\item the $\boldsymbol\theta$-adapted $L$-bundle $E_{\sigma_L}$ is $R_{h_L}$-stable, where $h_L$ is the metric on the $L$-bundle $E_{\sigma_L}$ induced by $h$.
\end{itemize}
\end{defn}

A reduction satisfies the first condition that the degree equals zero is called an \emph{admissible} reduction of structure group \cite[\S 3.3]{Rama19961}.

\begin{lem}
Given a $\boldsymbol\theta$-adapted $G$-bundle, it is $R_h$-polystable if and only if the corresponding parahoric $\mathcal{G}_{\boldsymbol\theta}$-torsor is $R$-polystable.
\end{lem}

\begin{proof}
This lemma can be proved similarly as Proposition \ref{prop stab of parah tor}.
\end{proof}
	
\section{Analytic $G$-Higgs Bundles and analytic $G$-connections}\label{sect ana Higgs and conn}
	
For a metrized $G$-bundle on a noncompact curve $X_{\boldsymbol D}$, we consider in this section an extra layer of structure defined by a Higgs field and study the relationship to $G$-connections on $X_{\boldsymbol D}$. Under an appropriate harmonicity condition, we describe a categorical equivalence between tame harmonic $G$-Higgs bundles and tame harmonic $G$-connections. We also describe a local correspondence between the weights and residues (or monodromies) at the points in the divisor $\boldsymbol D$ of the corresponding elements in the Dolbeault, de Rham and Betti spaces. 
	
\subsection{Basic definitions}\label{subsect basic defn of HIggs and conn}
We begin by extending basic definitions and calculations of Simpson \cite{Simp,Simp Higgs Local} for Higgs fields and flat connections on a metrized $G$-bundle on $X_{\boldsymbol D}$. Note that since $X_{\boldsymbol D}$ is a curve, the \emph{integrability condition} for both connections and Higgs fields is satisfied automatically.
	
Let $(E,\partial''_E,h)$ be a metrized $G$-bundle on $X_{\boldsymbol D}$. Denote by $d=\partial + \bar{\partial}$ the usual exterior differential operator with $(1,0)$-part $\partial$ and $(0,1)$-part $\bar{\partial}$. A \emph{Higgs field} $\phi$ associated to $E$ is defined as a holomorphic section of $E(\mathfrak{g}) \otimes K_{X_{\boldsymbol D}}$, where $K_{X_{\boldsymbol D}}$ is the canonical line bundle. With respect to the given metric $h$, we can get an ``adjoint" section $\phi^{*}$ of type $(0,1)$
\begin{align*}
	(\phi^{*})^{\rm T}={\rm Ad}(h) \bar{\phi},
\end{align*}
where $\bar{\phi}$ is the (complex) conjugation and $\phi^{\rm T}$ is the transpose (see \S\ref{subsect functor Xi}). Since $h$ is hermitian, the relation above is equivalent to
\begin{align*}
	{\rm Ad}(h) \bar{\phi}^* = \phi^{\rm T},
\end{align*}
Now let $\partial', \partial'': \mathscr{A}^0(E) \rightarrow \mathscr{A}^1(E)$ be operators of type $(1,0)$ and $(0,1)$ respectively.
\begin{notn}
We would like to use the notation $\partial'$ (resp. $\partial''$) for operators of type $(1,0)$ (resp. type $(0,1)$), and the notation $\bar{\partial}'$ actually means the (complex) conjugation of $\partial'$ with respect to what we defined in \S\ref{subsect functor Xi}. The operators $\partial$ and $\bar{\partial}$ correspond to the differential operator $d$ as we defined above. Later on, we shall also use the notation $\partial'_0:=\partial$ and $\partial''_0:=\bar{\partial}$.	
\end{notn}
	
We say that the operator $\partial' + \partial''$ \emph{preserves the metric} $h$, if we have
\begin{align*}
	\partial h \cdot h^{-1} = \partial'^{\rm T} + {\rm Ad}(h)\bar{\partial}'',
\end{align*}
which is equivalent to the equation
\begin{align*}
	h^{-1} \cdot \partial h = {\rm Ad}(h^{-1}) \partial'^{\rm T} + \bar{\partial}''.
\end{align*}
	
\begin{rem}\label{rem hermitian metric}
In the case of $G={\rm GL}_n(\mathbb{C})$, a metric $h$ on $E$ (as vector bundles) induces a natural bilinear form $h(\cdot,\cdot)$ on fibers of $E$. Then, for the Higgs field $\phi$, we have
\begin{align*}
	h(\phi^{*}u,v)=h(u,\phi v).
\end{align*}
or equivalently,
\begin{align*}
	(\phi^{*})^{\rm T} h = h \bar{\phi}.
\end{align*}
For the operators $\partial'$ and $\partial''$, we have
\begin{align*}
	\partial h(u,v)= h(\partial' u,v) + h(u,\partial'' v),
\end{align*}
which means
\begin{align*}
	\partial h= \partial'^{\rm T} h + h \bar{\partial}''.
\end{align*}
For general reductive groups, the above relations are exactly
\begin{align*}
	(\phi^{*})^{\rm T}={\rm Ad}(h) \bar{\phi}, \quad \partial h \cdot h^{-1} = \partial'^{\rm T} + {\rm Ad}(h)\bar{\partial}''.
\end{align*}
Furthermore, as we discussed above, the hermitian metric $h$ induces a bilinear form $h(\cdot,\cdot)$ on the fibers of a vector bundle (the case of ${\rm GL}_n(\mathbb{C})$). We extend this property to $G$-bundles $E$, more precisely, the adjoint bundle $E(\mathfrak{g})$, and we still work on ``fibers". The metric $h$ induces a bilinear operation $h(\cdot,\cdot): \mathfrak{g} \times \mathfrak{g} \rightarrow \mathfrak{g}$ as $h(u,v)=u^T h \bar{v}$. Here is a brief explanation of the product $u^T h \bar{v}$. We regard $h,u,v$ as the corresponding adjoint operators ${\rm Ad}(h),{\rm ad}(u),{\rm ad}(v)$ on $\mathfrak{g}$, and then, the product $u^T h \bar{v}$ is defined as the product of the corresponding adjoint operators.
\end{rem}
	
Let $(E,D,h)$ be a metrized $G$-bundle equipped with an integrable connection $D: \mathscr{A}^0(E) \rightarrow \mathscr{A}^1(E)$. Equivalently, an integrable connection on $E$ is equivalent to a homomorphism $T_{X_{\boldsymbol D}} \rightarrow {\rm At}(E)$ of Lie algebroids, where $T_{X_{\boldsymbol D}}$ is the tangent bundle and ${\rm At}(E)$ is the Atiyah algebroid (see Appendix \ref{sect appendix} for the definition). Furthermore, this equivalent definition of connections on $G$-bundles implies that a connection corresponds to an element in $\mathscr{A}^1(E(\mathfrak{g}))$ locally, which is regarded as the \emph{connection form} (see Example \ref{exmp connection analytic}). We write $D=d'+d''$ as a sum of operators of type $(1,0)$ and $(0,1)$. Then, for the given metric $h$, we can find a $(1,0)$-operator $\delta'$ and a $(0,1)$-operator $\delta''$ such that both $d'+\delta''$ and $\delta'+d''$ preserve the metric $h$, in other words,
\begin{align*}
	& h^{-1} \cdot \partial h = d'^{\rm T} + {\rm Ad}(h) \bar{\delta}'' ,\\
	& h^{-1} \cdot \partial h = \delta'^{\rm T}+ {\rm Ad}(h) \bar{d}''.
\end{align*}
We then define the following operators
\begin{align*}
	& \partial' :=\frac{1}{2}(d'+\delta'), \quad \phi:=\frac{1}{2}(d'-\delta')\\
	& \partial'':=\frac{1}{2}(d''+\delta''), \quad \phi^*:=\frac{1}{2}(d''-\delta''),
\end{align*}
and compute
\begin{align*}
	\partial'^{\rm T} + {\rm Ad}(h) \bar{\partial}'' +  &= \frac{1}{2}\left( (d'^{\rm T}+\delta'^{\rm T}) + {\rm Ad}(h)(\bar{d}''+\bar{\delta}'') \right)\\
	& =\frac{1}{2} \left( ( \delta'^{\rm T} + {\rm Ad}(h)\bar{d}'') +  (d'^{\rm T} + {\rm Ad}(h)\bar{\delta}'' ) \right)\\
	& = h^{-1} \cdot \partial h.
\end{align*}
This actually means that $\partial'+\partial''$ preserves the metric $h$. At the same time, we have
\begin{align*}
	d'^{\rm T} + {\rm Ad}(h) \bar{\delta}'' =  \delta'^{\rm T} + {\rm Ad}(h) \bar{d}'',
\end{align*}
which implies
\begin{align*}
	{\rm Ad}(h)(\bar{d}'' - \bar{\delta}'')=d'^{\rm T} - \delta'^{\rm T},
\end{align*}
and then
\begin{align*}
	{\rm Ad}(h) \bar{\phi}^* = \phi^T.
\end{align*}
This calculation shows that $\phi$ as introduced above is adjoint to $\phi^*$ with respect to the metric $h$.
	
\subsection{Analytic $G$-Higgs Bundles and Analytic $G$-connections}\label{subsect ana Hig ana D-mod}
	
In this subsection, we define $G$-Higgs bundles and $G$-connections over $X_{\boldsymbol D}$ analytically and introduce a notion of harmonicity for these objects. The harmonicity condition commonly refers to a hermitian metric solving the Hermite--Einstein equation and it is equivalent to a notion of stability for the points in the Dolbeault moduli space. This is the content of the Kobayashi--Hitchin correspondence, which we will prove later on in \S\ref{subsect Kob-Hitchin}. We will see then that the existence of a hermitian metric solving the Hermite--Einstein equation, in fact, implies the property we call \emph{harmonicity} in this subsection, thus explaining this slight abuse of terminology.

\begin{defn}\label{defn ana Higgs}
An \emph{analytic $G$-Higgs bundle} on $X_{\boldsymbol D}$ is a triple $(E,\partial''_E, \phi)$, where
\begin{itemize}
	\item $(E,\partial''_E)$ is a holomorphic $G$-bundle on $X_{\boldsymbol D}$,
	\item $\phi$ is a Higgs field that is a holomorphic section of $E(\mathfrak{g}) \otimes K_{X_{\boldsymbol D}}$.
\end{itemize}
It is called a \emph{metrized $G$-Higgs bundle}, if it comes equipped with a metric $h$.
\end{defn}

Given a metrized $G$-Higgs bundle $(E,\partial''_E,\phi,h)$, we obtain an operator
\begin{align*}
	D'':=\partial''_E + \phi.
\end{align*}
Denote by
\begin{align*}
	F_{D''}:=(D'')^2
\end{align*}
the \emph{pseudo-curvature} of $D''$ and clearly, $F_{D''}=0$. We also induce a connection
\begin{equation}\label{induced_flat_connect}
	D=\partial'_h +\partial''_E + \phi + \phi^*,
\end{equation}
where $\partial'_h+\partial''_E$ preserves the metric $h$ and $\phi^*$ is adjoint to $\phi$ with respect to the metric $h$. Denote by $F_D$ the curvature of $D$. In conclusion, with respect to the metric $h$, we obtain a tuple $(E,D,D'',h)$. The discussion above introduces the concept of $G$-connections as follows:
	
\begin{defn}
An \emph{analytic $G$-connection} on $X_{\boldsymbol D}$ is a pair $(E,D)$, where
\begin{itemize}
	\item $E$ is a $G$-bundle on $X_{\boldsymbol D}$,
	\item $D$ is an integrable connection on $E$.
\end{itemize}
If we equip $E$ with a metric $h$, then we say that $(E,D,h)$ is a \emph{metrized $G$-connection}.
\end{defn}
	
Given a metrized $G$-connection $(E,D,h)$, then $F_D=0$ by definition. Also, given the above data, we can define the operator $D'':=\partial''+\phi$. If we want to obtain an analytic $G$-Higgs bundle, we need to add the extra condition that the pseudo-curvature $F_{D''}$ is trivial. Under this condition, one can follow the calculation in \S\ref{subsect basic defn of HIggs and conn} and induce a metrized $G$-Higgs bundle. Before we state the correspondence, we first introduce the following definition.
	
\begin{defn}\label{defn har bun Dol+DR}
A \emph{harmonic $G$-Higgs bundle} is a metrized $G$-Higgs bundle such that its induced connection $D$ as in \eqref{induced_flat_connect} is flat. On the other hand, a \emph{harmonic $G$-connection} is a metrized $G$-connection such that its induced pseudo-curvature $F_{D''}$ is trivial. The corresponding metric inducing these conditions in either case will be called a \emph{harmonic metric}.
\end{defn}
	
Given the definitions above and the calculations from \S \ref{subsect basic defn of HIggs and conn} we have the following:
\begin{lem}\label{lem higgs and connect 1}
The category of harmonic $G$-Higgs bundles is equivalent to the category of harmonic $G$-connections.
\end{lem}
	
Given a harmonic $G$-connection $(E,D,h)$, the $D_X$-module structure is induced from the flat connection $D$, which decomposes to its $(1,0)$-part $d'$ and $(0,1)$-part $d''$, as
\begin{align*}
	D=d'+d''.
\end{align*}
Therefore, the $(0,1)$-type operator $d''$ equips the $G$-bundle $E$ with a holomorphic structure and the induced holomorphic $G$-bundle $(E,d'')$ will be denoted by $V$. Furthermore, the $(1,0)$-operator is holomorphic with respect to the holomorphic structure $d''$. Therefore, we get a holomorphic $G$-connection $(V,d')$, which is a more algebraic description of the $G$-connection.
	
\subsection{Degree}\label{subsect deg Higgs and mod}
Let $(E,\partial''_E,\phi,h)$ be a metrized $G$-Higgs bundle on $X_{\boldsymbol D}$. Let $\sigma: X_{\boldsymbol D} \rightarrow E/P$ be a holomorphic reduction of structure group and let $\chi: P \rightarrow \mathbb{C}^*$ be a character. Similar to \S\ref{subsect ana deg on bundle}, we first define the induced metrized $P$-Higgs bundle $E_\sigma$. Since the induced metric $h_\sigma$ and holomorphic structure $\partial''_{E_\sigma}$ is already given in \S\ref{subsect ana deg on bundle}, we only have to consider the Higgs field.

Given an arbitrary reduction of structure group $\sigma$, the Higgs field $\phi$ may not be lifted to a Higgs field on $E_\sigma$. We introduce the following:
\begin{defn}
Let $(E,\partial''_E,\phi)$ be an analytic $G$-Higgs bundle on $X_{\boldsymbol D}$. A holomorphic reduction of structure group $\sigma: X_{\boldsymbol D} \rightarrow E / P$ is said to be \emph{$\phi$-compatible}, if there exists a lifting $\phi_\sigma : X_{\boldsymbol D} \rightarrow E_\sigma(\mathfrak{p}) \otimes K_{X_{\boldsymbol D}}$, such that the following diagram commutes
\begin{center}
\begin{tikzcd}
	& E_\sigma(\mathfrak{p}) \otimes K_{X_{\boldsymbol D}} \arrow[d, hook] \\
	X_{\boldsymbol D} \arrow[r, "\phi"] \arrow[ur, "\phi_\sigma", dotted] & E(\mathfrak{g}) \otimes K_{X_{\boldsymbol D}}.
\end{tikzcd}
\end{center}
If there is no ambiguity, we shall simply say \emph{compatible holomorphic reductions of structure group}.
\end{defn}

Here is an alternative way to understand the condition of compatibility. By taking the adjoint representation, we have the following Cartesian diagram for the adjoint bundles
\begin{center}
\begin{tikzcd}
	E_\sigma(\mathfrak{p}) \arrow[r] \arrow[d] & E(\mathfrak{g}) \arrow[d]\\
	X_{\boldsymbol D} \arrow[r, "\sigma"]  & E(\mathfrak{u}^-),
\end{tikzcd}
\end{center}
where $\mathfrak{u}^-$ is the Lie algebra of the unipotent group $U^- \subseteq G$, which is defined by the set of roots $\mathcal{R} \backslash \mathcal{R}_P$ (see the notation in \S\ref{subsect alg parah deg}); here we use the same notation $\sigma$ for the induced section of $E(\mathfrak{u}^-)$. Now, let us consider the following diagram
\begin{center}
\begin{tikzcd}
	X_{\boldsymbol D} \arrow[rd, dotted, "\phi_\sigma"] \arrow[rrd, bend left, "\phi"] \arrow[rdd, bend right, "{\rm id}"] & & \\
	& E_\sigma(\mathfrak{p}) \otimes K_{X_{\boldsymbol D}} \arrow[r] \arrow[d] & E(\mathfrak{g}) \otimes K_{X_{\boldsymbol D}} \arrow[d]\\
	& X_{\boldsymbol D} \arrow[r, "\sigma"]  & E(\mathfrak{u}^-).
\end{tikzcd}
\end{center}
Then, $\phi_\sigma$ exists if the diagram commutes, that is, $\phi|_{\mathfrak{u}^-} = \sigma$. From the point of view of Higgs bundles, this condition pertains to the existence of subbundles preserved by the Higgs field.

For analytic $G$-connections, we have a similar definition:
\begin{defn}
Let $(E,D)$ be an analytic $G$-connection on $X_{\boldsymbol D}$. A holomorphic reduction of structure group $\sigma: X_{\boldsymbol D} \rightarrow E / P$ is said to be \emph{$D$-compatible}, if there exists a lifting $D_\sigma : \mathscr{A}^0(E_\sigma(\mathfrak{p})) \rightarrow \mathscr{A}^1(E_\sigma(\mathfrak{p}))$, such that the following diagram commutes
\begin{center}
\begin{tikzcd}
	\mathscr{A}^0(E_\sigma (\mathfrak{p})) \arrow[r,"D_\sigma", dotted] \arrow[d] & \mathscr{A}^1(E_\sigma(\mathfrak{p})) \arrow[d] \\
	\mathscr{A}^0(E_(\mathfrak{g})) \arrow[r, "D"] & \mathscr{A}^1(E (\mathfrak{g})).
\end{tikzcd}
\end{center}
If there is no ambiguity, we shall simply say \emph{compatible holomorphic reductions of structure group}.
\end{defn}

We remind the reader that the holomorphic reduction of structure group depends on the holomorphic structure. More precisely, given an analytic $G$-Higgs bundle $(E,\partial''_E,\phi)$, a reduction of structure group $\sigma$ is holomorphic with respect to $\partial''_E$. Similarly, given an analytic $G$-connection $(V,d')$, $\sigma$ is holomorphic with respect to $d''$.

\begin{rem}
Let $E$ be the trivial $G$-bundle. Then, a connection $D$ is equivalent to an element in $\mathscr{A}^1(E(\mathfrak{g}))$ (see Appendix \ref{sect appendix}). Thus, the compatibility condition for connections can be stated in the same way as for Higgs fields:
\begin{center}
\begin{tikzcd}
	& E_\sigma(\mathfrak{p}) \otimes K_{X_{\boldsymbol D}} \arrow[d, hook] \\
	X_{\boldsymbol D} \arrow[r, "D"] \arrow[ur, "D_\sigma", dotted] & E(\mathfrak{g}) \otimes K_{X_{\boldsymbol D}}.
\end{tikzcd}
\end{center}
For a general analytic $G$-connection, this statement holds locally.
\end{rem}

\begin{defn}\label{defn ana deg Higgs and loc}
Let $(E,\partial''_E,\phi,h)$ be a metrized $G$-Higgs bundle, let $\sigma : X_{\boldsymbol D} \rightarrow E / P$ be a compatible reduction of structure group, and let $\chi : P \to \mathbb{C}^{*}$ be a character. We define the \emph{analytic degree} of $(E,\partial''_E,\phi,h)$ as
\begin{align*}
	\deg^{\rm an} E (\bullet,\sigma,\chi) = \frac{\sqrt{-1}}{2\pi} \int_{X_{\boldsymbol{D}}} \chi_* F_{\bullet_\sigma},
\end{align*}
where $\bullet=h, D, D''$. If $\sigma$ is the trivial reduction of structure group, we define
\begin{align*}
	\deg^{\rm an} E (\bullet,\chi) = \frac{\sqrt{-1}}{2\pi} \int_{X_{\boldsymbol{D}}} \chi_* F_{\bullet},
\end{align*}
where $\bullet=h, D, D''$. Given an analytic metrized $G$-connection $(E,D,h)$, one similarly defines a notion of degree.
\end{defn}

For a metrized $G$-connection $(E,D,h)$, the operator $\partial' + \partial''$ preserves the metric $h$ and this implies that $F_h = \partial' \circ \partial'' + \partial'' \circ \partial'$, thus
\begin{align*}
	F_D = F_h + 2 F_{D''} + [\phi, \phi^*].
\end{align*}
Since $D$ is an integrable connection, for an arbitrary character $\chi$, we have
\begin{align*}
	\deg^{\rm an} E (D,\chi) = \frac{\sqrt{-1}}{2\pi} \int_{X_{\boldsymbol{D}}} \chi_* F_{D}=0.
\end{align*}
Thus,
\begin{align*}
	0=\int_{X_{\boldsymbol{D}}} \chi_* F_{D} = \int_{X_{\boldsymbol{D}}} \chi_* (F_{h} + 2 F_{D''} + [\phi, \phi^*]).
\end{align*}
Note that
\begin{align*}
	\int_{X_{\boldsymbol{D}}} \chi_* [\phi, \phi^*]=0,
\end{align*}
therefore,
\begin{align*}
	\deg^{\rm an} E (h,\chi) = \frac{\sqrt{-1}}{2\pi} \int_{X_{\boldsymbol{D}}} \chi_* F_{h} = -2 \times \frac{\sqrt{-1}}{2\pi} \int_{X_{\boldsymbol{D}}} \chi_* F_{D''} =-2 \deg^{\rm an} E (D'',\chi).
\end{align*}
On the other hand, given a metrized $G$-Higgs bundle $(E,\partial''_E,\phi,h)$, we have
\begin{align*}
	2 F_{D''} = F_D - F_h - [\phi, \phi^*].
\end{align*}
Since $F_{D''}$ is trivial, we get
\begin{align*}
\deg^{\rm an} E (D,\chi)= \frac{\sqrt{-1}}{2\pi} \int_{X_{\boldsymbol{D}}} \chi_* F_D = \frac{\sqrt{-1}}{2\pi} \int_{X_{\boldsymbol{D}}} \chi_* F_h= \deg^{\rm an} E (h,\chi).
\end{align*}	
The above discussion gives us the following lemma.
\begin{lem}\label{lem three ana deg}
Given a harmonic $G$-Higgs bundle $(E,\partial''_E, \phi,h)$, we have
\begin{align*}
	\deg^{\rm an} E (h,\chi)=\deg^{\rm an} E (D,\chi)=\deg^{\rm an} E (D'',\chi)=0,
\end{align*}
for any character $\chi: G \rightarrow \mathbb{C}^*$.
		
Similarly, given a harmonic $G$-connection $(E,D,h)$, we have
\begin{align*}
	\deg^{\rm an} E (h,\chi)=\deg^{\rm an} E (D,\chi)=\deg^{\rm an} E (D'',\chi)=0
\end{align*}
for any character $\chi: G \rightarrow \mathbb{C}^*$.

Furthermore, harmonic $G$-Higgs bundles and harmonic $G$-connections are of degree zero in the sense of \S\ref{subsect degree zero}.
\end{lem}

\subsection{Tameness}
In this subsection, we introduce the tameness condition for analytic $G$-Higgs bundles and analytic $G$-connections.
\begin{defn}\label{defn tame ana Higgs}
Let $(E,\partial''_E,\phi)$ be an analytic $G$-Higgs bundle. For each puncture $x \in \boldsymbol D$, let $e$ be a local trivialization such that $\partial''_E = \partial''_0$ on a neighborhood of $x$ with local coordinate $z$, where $\partial''_0$ is the $(0,1)$-part of the standard connection on $\mathcal{O}_X$. The tuple $(E,\partial''_E,\phi)$ is a \emph{tame $\boldsymbol\theta$-adapted $G$-Higgs bundle}, if for each puncture $x \in \boldsymbol D$, we have
\begin{align*}
	z^{\theta_x} \cdot (z\phi(z)) \cdot z^{-\theta_x} \text{ is bounded as $z$ approaches zero}.
\end{align*}
A \emph{tame metrized $\boldsymbol\theta$-adapted $G$-Higgs bundle} is a tame $\boldsymbol\theta$-adapted $G$-Higgs bundle equipped with a metric $h$. Furthermore, a \emph{tame harmonic $\boldsymbol\theta$-adapted $G$-Higgs bundle} is a tame metrized $\boldsymbol\theta$-adapted $G$-Higgs bundle such that the induced connection $D$ is flat.
\end{defn}

\begin{rem}
	In the case of ${\rm GL}_n(\mathbb{C})$ (that is, for vector bundles), the tameness condition is exactly the same as Simpson considered in \cite[\S 2]{Simp}. Let $\theta$ be a weight of ${\rm GL}_n(\mathbb{C})$, and denote by ${\rm GL}_\theta(K)$ the corresponding parahoric group. Since any weight of ${\rm GL}_n(\mathbb{C})$ is equivalent to a small weight, the parahoric group ${\rm GL}_\theta(K)$ can be regarded as a subgroup of ${\rm GL}_n(R)$ under conjugation. Therefore, in local charts, a Higgs field $\varphi$ is an element in $\frac{\mathfrak{gl}_{\theta}(K)}{z} dz \subseteq \frac{\mathfrak{gl}_n(R)}{z}dz$, and then, the eigenvalues of $\varphi$ have poles at most one, which matches the definition of a \emph{tame Higgs field} given by Simpson.
\end{rem}

We next move on to analytic $G$-connections. Let $(E,D)$ be an analytic $G$-connection, and let $D=d'+d''$. In \S\ref{subsect ana Hig ana D-mod}, we gave an alternative definition, and equipped $E$ with a new holomorphic structure $d''$ together with a holomorphic connection $d'$. Then, we can use a similar approach as we did in the case of Higgs bundles. For each puncture $x$, we can take a local coordinate and a good trivialization of $E$ such that $d''=\partial''_0$ around the puncture $x$. Since we work on local charts, we can suppose that the induced connection $d'$ is an $E(\mathfrak{g})$-valued one form. With respect to this trivialization, we give the definition of tameness for $G$-connections.
	
\begin{defn}\label{defn tame ana conn}
Let $(E,D)$ be an analytic $G$-connection. For each puncture $x$, let $e$ be a local trivialization such that $\partial''_E = \partial''_0$ on a neighborhood of $x$. Then, the tuple $(E,D)$ is a \emph{tame $\boldsymbol\theta$-adapted $G$-connection}, if for each puncture $x \in \boldsymbol D$, we have
\begin{align*}
z^{\theta_x} \cdot (zd') \cdot z^{-\theta_x} \text{ is bounded as $z$ approaches zero}.
\end{align*}
A \emph{tame metrized $\boldsymbol\theta$-adapted $G$-connection} is a tame $\boldsymbol\theta$-adapted $G$-connection equipped with a metric $h$. Moreover, a \emph{tame harmonic $\boldsymbol\theta$-adapted $G$-connection} is a tame metrized $\boldsymbol\theta$-adapted $G$-connection such that the induced pseudo-curvature $F_{D''}$ is trivial.
\end{defn}
	
\begin{rem}
In the above definition, since both the Higgs field $\phi$ and the operator $d'$ of type $(1,0)$ are holomorphic, we use the coordinate $z$ to define boundedness. More generally, we can rewrite this condition by applying the polar coordinate $r=|z|$, and for example, the boundedness condition for $d'$ will be
\begin{align*}
	r^{\theta_x} \cdot (rd') \cdot r^{-\theta_x} \text{ is bounded as $r$ approaches zero}.
\end{align*}
This is a more workable definition when dealing with local calculations.
\end{rem}

\begin{rem}	
Note that the terminology \emph{$\boldsymbol\theta$-adapted} in Definitions \ref{defn tame ana Higgs} and \ref{defn tame ana conn} is not for the metric, but for Higgs fields and connections. In this paper, the metric $h$ of a tame harmonic $\boldsymbol\theta$-adapted Higgs bundle $(E,\partial''_E,\phi,h)$ is also $\boldsymbol\theta$-adapted, but the metric of a tame harmonic $\boldsymbol\theta$-adapted $G$-connection $(V,d')$ is usually not $\boldsymbol\theta$-adapted. We will study this relation locally in the next subsection.
\end{rem}

\subsection{Local Study}\label{subsect local study}
Analogously to Simpson's description of model metrics in the case of filtered ($\text{GL}_n(\mathbb{C})$-)Higgs bundles (\cite[\S 5]{Simp}), we construct in this section a class of model metrics $h_0$ locally around a puncture in $\boldsymbol D$ and consider the relation between Higgs bundles (Dolbeault side), connections (de Rham side) and representations (Betti side). More precisely, we complete the following table
\begin{table}[H]
\begin{tabular}{|c|c|c|c|}
	\hline
	& Dolbeault & de Rham & Betti \\
	\hline
	weights & $\alpha$ & $\beta$ & $\gamma$ \\
	\hline
	residues $\backslash$ monodromies & $\phi_\alpha$ & $d'_\beta$ & $M_\gamma$\\
	\hline
	\end{tabular}
\end{table}
The \emph{weights} $\alpha, \beta, \gamma$ are considered as the jump values of the metric $h_0$ with respect to distinct choices of trivializations. For the ``residues $\backslash$ monodromies", we mean the Levi factor $\phi_\alpha$ of the residue of the Higgs field $\phi$, the Levi factor $d'_\beta$ of the residue of the connection $d'$ at a fixed puncture, and the Levi factor $M_{\gamma}$ of the monodromy of the representation around the puncture. The choice of the model metric needed for this complete description is inspired by the work of Biquard--Garc\'{i}a-Prada--Mundet i Riera \cite{BGM}, where parabolic $G$-Higgs bundles for real reductive groups were studied, which also includes the case when the group $G$ is complex.
	
This detailed description of the local data will help us later on to establish the correspondence between tame harmonic $\boldsymbol\alpha$-adapted $G$-Higgs bundles and tame harmonic $\boldsymbol\beta$-adapted $G$-connections. Moreover, we will construct harmonic metrics with the desired local behavior and the local models here allow us to understand the behavior of the non-abelian Hodge correspondence in terms of the parahoric weights.
	
Let $x \in \boldsymbol D$ be a point in the divisor, and denote by $\mathbb{D}^*$ a punctured disc around $x$ with local coordinate $z$. Let $\alpha$ be a weight. By Lemma \ref{lem two defn parah Lie alg}, a tame $\alpha$-adapted Higgs field $\phi$ is regarded as an element in $\mathfrak{g}_\alpha(K) \frac{dz}{z}$. In this subsection, we suppose that
\begin{align*}
	\phi=\phi_\alpha \frac{dz}{z},
\end{align*}
where $\phi_\alpha \in \mathfrak{p}_\alpha \subseteq \mathfrak{g}$ is the residue of the Higgs field $\phi$ at $x$. Applying Jordan decomposition, we have
\begin{align*}
	\phi_\alpha=s_\alpha+Y_\alpha,
\end{align*}
where $s_\alpha$ is the semisimple part and $Y_\alpha$ is the nilpotent part. Since $G$ is a complex reductive group, we can suppose that the semisimple part $s_\alpha$ is an element in $\mathfrak{t}$. The element $Y_{\alpha}$ can be completed into a natural $\mathfrak{sl}_2$-triple $(X_\alpha, H_\alpha, Y_\alpha)$ called a \textit{normal triple} or a \textit{Kostant--Rallis triple} (see \cite[Proposition 4]{KR}) such that $X_\alpha$ is the ``hermitian dual" of $Y_\alpha$ and $H_\alpha=[X_\alpha, Y_\alpha]$ lies in $\mathfrak{t}$.
	
We consider similarly to \cite{BGM} the model metric
\begin{equation}\label{model_metric_description_nilpotent}
	h_0=|z|^{\alpha} (- \ln |z|^2)^{ H_\alpha} |z|^{\alpha}.
\end{equation}
	
\begin{rem}\label{comparison_model_real}
Note that the model metric considered in \cite[Formula (5.8)]{BGM} is given for the case of a real group $G$ by
\begin{align*}
	h_0=|z|^{\alpha} (- \ln |z|^2)^{ {\rm Ad}(e^{i \theta \alpha}) H_\alpha} |z|^{\alpha}.
\end{align*}
Compared to the metric we are considering above, the reason why the term ${\rm Ad}(e^{i \theta \alpha})$ is not apparent is that in the complex group case we always have ${\rm Ad}(e^{i \theta \alpha}) H_\alpha=H_\alpha$.
\end{rem}
	
For the model metric $h_0$ as in (\ref{model_metric_description_nilpotent}), we have
\begin{align*}
	\partial'_{h_0}=\partial'_0 + \left( \alpha+\frac{  H_\alpha }{ \ln |z|^2 } \right) \frac{dz}{z}, \quad D_0=d_0 + \left( \alpha+\frac{ H_\alpha }{ \ln |z|^2 } \right) \frac{dz}{z},
\end{align*}
where $d_0=\partial'_0 + \partial''_0$. Take
\begin{align*}
	g_0=|z|^{-\alpha} (- \ln |z|^2)^{  -\frac{ H_\alpha}{2} },
\end{align*}
and define the trivialization $e_0=e g_0$. Then, $h_0(e_0,e_0)=1$. Under the gauge action of $g_0$, we have
\begin{align*}
	\partial'_{h_0}=\partial'_0 + \frac{1}{2}\left( \alpha+ \frac{ H_\alpha }{ \ln |z|^2 } \right) \frac{dz}{z}, \quad \partial''_{h_0}=\partial''_0 - \frac{1}{2} \left(\alpha + \frac{ H_\alpha }{ \ln |z|^2 } \right) \frac{d\bar{z}}{\bar{z}},
\end{align*}
where we use the same notation for operators $\partial'_{h_0},\partial''_{h_0}$.

For the Higgs field $\phi$, we calculate ${\rm Ad}(g_0^{-1}) \phi$ and ${\rm Ad}(g_0^{-1}) \phi^*$. Note that
\begin{align*}
	{\rm Ad}(g_0) Y_\alpha &= {\rm Ad}( e^{\frac{H_\alpha}{2} \ln (-\ln |z|^2) } ) Y_\alpha \\
	&= Y_\alpha e^{ - \ln (-\ln |z|^2) }\\
	&= - \frac{ Y_\alpha}{ \ln |z|^2 },
\end{align*}
where we used the formula ${\rm Ad } (e^X) Y=e^{{\rm ad}(X)} Y$ in the second equality. Therefore, under the trivialization $e_0$, the Higgs field and its adjoint are given by
\begin{align*}
	\phi=\left( s_\alpha - \frac{ Y_\alpha }{ \ln |z|^2 } \right) \frac{dz}{z}, \quad \phi^{*}=\left( \bar{s}_\alpha - \frac{ X_\alpha }{ \ln |z|^2 } \right) \frac{d \bar{z}}{\bar{z}},
\end{align*}
where we also use the same notation for convenience.
	
With respect to the above description, the connection $D$ (see (\ref{induced_flat_connect})) is given as
\begin{align*}
	D & = \partial'_{h_0} + \partial''_{h_0}+\phi +\phi^{*} \\
	&=d_0  + \left( \frac{1}{2}\alpha + s_\alpha +  \frac{\frac{1}{2}H_\alpha- Y_\alpha }{ \ln |z|^2 } \right)\frac{dz}{z}\\
	&\ \ \ \ \ \ + \left( - \frac{1}{2}\alpha + \bar{s}_\alpha +  \frac{ - \frac{1}{2}H_\alpha- X_\alpha }{ \ln |z|^2 } \right)\frac{d \bar{z}}{\bar{z}}.
\end{align*}
Next, we take another trivialization $e_1=e_0 g_1$, where
\begin{align*}
	g_1& =|z|^{-s_\alpha - \bar{s}_\alpha} ( - \ln |z|^2 )^{ \frac{1}{2}(X_\alpha+Y_\alpha) }.
\end{align*}
Note that
\begin{align*}
	h_0(e_1,e_1)=|z|^{-2(s_\alpha+\bar{s}_\alpha)}(-\ln |z|^2)^{(X_\alpha+Y_\alpha)}.
\end{align*}
Under the trivialization $e_1$, the connection $D$ becomes
\begin{align*}
	d_0 &+ \left( \frac{1}{2}\alpha + \frac{1}{2} s_\alpha - \frac{1}{2} \bar{s}_\alpha -\frac{1}{2}(H_\alpha + X_\alpha -Y_\alpha) \right) \frac{dz}{z}\\
	&- \left( \frac{1}{2}\alpha + \frac{1}{2} s_\alpha - \frac{1}{2} \bar{s}_\alpha -\frac{1}{2}(H_\alpha + X_\alpha -Y_\alpha) \right) \frac{d \bar{z}}{\bar{z}},
\end{align*}
where we applied the formulas
\begin{align*}
	\mathrm{Ad}\Big(( - \ln |z|^2 )^{- \frac{1}{2}(X_\alpha+Y_\alpha) }\Big)&\left(\frac{1}{2}H_\alpha- Y_\alpha\right)\\
	&=\frac{1}{2}(H_\alpha + X_\alpha -Y_\alpha)(- \ln |z|^2)-\frac{1}{2}(X_\alpha+Y_\alpha);\\
	\mathrm{Ad}\Big(( - \ln |z|^2 )^{- \frac{1}{2}(X_\alpha+Y_\alpha) }\Big)&\left(-\frac{1}{2}H_\alpha- X_\alpha\right)\\
	&=-\frac{1}{2}(H_\alpha + X_\alpha -Y_\alpha)(- \ln |z|^2)-\frac{1}{2}(X_\alpha+Y_\alpha).
\end{align*}
Now we consider the connection $D$ in the following two ways:
\begin{enumerate}
\item In the Betti side, recall the formula
\begin{align*}
	i d \theta = \frac{1}{2} ( \frac{dz}{z} - \frac{d \bar{z}}{\bar{z}} ).
\end{align*}
We write the formula about the connection $D$ in polar coordinates:
\begin{align*}
	d_0 + i ( \alpha + s_\alpha - \bar{s}_\alpha - ( H_\alpha + X_\alpha - Y_\alpha ) ) d \theta.
\end{align*}
Therefore, the \emph{weight} is given by
\begin{align*}
	\gamma = -(s_\alpha + \bar{s}_\alpha),
\end{align*}
and the \emph{monodromy} is given by
\begin{align*}
	\exp\left( - 2 \pi i ( \alpha + s_\alpha - \bar{s}_\alpha) \right)\exp \left( 2 \pi i( H_\alpha + X_\alpha - Y_\alpha ) \right).
\end{align*}
\item On the de Rham side, let $e_2=e_1 g_2$, where
\begin{align*}
	g_2=|z|^{ \alpha + s_\alpha - \bar{s}_\alpha -( H_\alpha +X_\alpha -Y_\alpha) }.
\end{align*}
Under the trivialization $e_2$, the connection $D$ becomes
\begin{align*}
		d_0 + ( \alpha + s_\alpha - \bar{s}_\alpha - ( H_\alpha + X_\alpha - Y_\alpha ))\frac{dz}{z}.
\end{align*}
Furthermore,
\begin{align*}
	h_{0}(e_2, e_2)& = |z|^{2\alpha- 2(s_\alpha+\bar{s}_\alpha) - 2 ( H_\alpha +X_\alpha -Y_\alpha) } \cdot(-\ln |z|^2)^{(X_\alpha+Y_\alpha)}\\
	&= |z|^{2\alpha- 2(s_\alpha+\bar{s}_\alpha)} |z|^{- 2 ( H_\alpha +X_\alpha -Y_\alpha)} \cdot(-\ln |z|^2)^{(X_\alpha+Y_\alpha)}.
\end{align*}
Note that $H_\alpha+X_\alpha - Y_\alpha$ is a nilpotent element, and under some good choice of basis, it can be regarded as the matrix
\begin{align*}
\begin{pmatrix}
		1 & 1 \\
		-1 & -1
\end{pmatrix}.
\end{align*}
Thus,
\begin{align*}
	|z|^{- 2 ( H_\alpha +X_\alpha -Y_\alpha)} & = \exp \left( \begin{pmatrix}
	1 & 1 \\
	-1 & -1
\end{pmatrix} \cdot \ln |z|^{-2} \right) \\
& =  \begin{pmatrix}
1 & 0 \\
0 & 1
\end{pmatrix} +  \begin{pmatrix}
	1 & 1 \\
	-1 & -1
\end{pmatrix} \cdot \ln |z|^{-2}.
\end{align*}
Therefore, this term does not contribute to the jump value. Summing up, the \emph{weight} is
\begin{align*}
	\beta = \alpha-(s_\alpha+\bar{s}_\alpha),
\end{align*}
and the \emph{residue} is
\begin{align*}
	d'_\beta = \alpha+(s_\alpha-\bar{s}_\alpha)-(H_\alpha + X_\alpha - Y_\alpha).
\end{align*}
\end{enumerate}
	
In conclusion, we have the following table:
\begin{table}[H]
	\begin{tabular}{|c|c|c|c|}
		\hline
		& Dolbeault & de Rham & Betti \\
		\hline
		weights & $\alpha$ & $\alpha-(s_\alpha+\bar{s}_\alpha)$ & $-(s_\alpha+\bar{s}_\alpha)$ \\
		\hline
		residues $\backslash$ monodromies & $s_\alpha+Y_\alpha$ & $d'_\beta$ & $M_\gamma$\\
		\hline
	\end{tabular}
\end{table}
\noindent where
\begin{align*}
	d'_\beta =\alpha+(s_\alpha-\bar{s}_\alpha)- ( H_\alpha + X_\alpha - Y_\alpha)
\end{align*}
and
\begin{align*}
	M_\gamma=\exp\left( - 2 \pi i ( \alpha + s_\alpha - \bar{s}_\alpha) \right)\exp \left( 2 \pi i( H_\alpha + X_\alpha - Y_\alpha ) \right).
\end{align*}
	
\begin{rem}
If the Higgs field $\phi$ does not have the nilpotent part, i.e.
\begin{align*}
	\phi=\phi_\alpha\frac{dz}{z}=s_\alpha \frac{dz}{z},
\end{align*}
the model metric we choose is $h_0=|z|^{2\alpha}$. With the same calculation as above, we have the following table for the semisimple case:
\begin{table}[H]
	\begin{tabular}{|c|c|c|c|}
		\hline
		& Dolbeault & de Rham & Betti \\
		\hline
		weights & $\alpha$ & $\alpha-(s_\alpha+\bar{s}_\alpha)$ & $-(s_\alpha+\bar{s}_\alpha)$ \\
		\hline
		residues $\backslash$ monodromies & $s_\alpha$ & $\alpha+(s_\alpha-\bar{s}_\alpha)$ & $\exp(-2\pi i (\alpha+(s_\alpha-\bar{s}_\alpha)) )$\\
		\hline
	\end{tabular}.
\end{table}
\noindent The data in this table was predicted by Boalch in \cite[\S 6]{Bo}.
\end{rem}
	
\subsection{The Correspondence}
	
From Lemma \ref{lem higgs and connect 1}, a metrized $G$-Higgs bundle corresponds to a metrized $G$-connection under harmonicity. Now we are going to give a more precise description of this correspondence under the condition of tameness. Let $\boldsymbol\alpha=\{\alpha_x, \text{ } x \in \boldsymbol D\}$ and $\boldsymbol\beta=\{\beta_x, \text{ } x \in \boldsymbol D\}$ be two collections of weights.
\begin{prop}\label{prop tame ana Higgs and conn}
Let $h$ be an $\boldsymbol\alpha$-adapted metric. A tame harmonic $\boldsymbol\alpha$-adapted $G$-Higgs bundle $(E,\partial''_E, \phi,h)$, with
\begin{align*}
	\phi_{\alpha_x} = s_{\alpha_x} + Y_{\alpha_x}
\end{align*}
the Levi factor of the residue of the Higgs field $\phi$ around each puncture $x$, corresponds to a tame harmonic $\boldsymbol\beta$-adapted $G$-connection $(V,d')$, where
\begin{align*}
	\beta_x=\alpha_x - (s_{\alpha_x} + \bar{s}_{\alpha_x}),
\end{align*}
with
\begin{align*}
	d'_{\beta_x} = \alpha_x +(s_{\alpha_x} - \bar{s}_{\alpha_x}) - (H_{\alpha_x} + X_{\alpha_x} - Y_{\alpha_x})
\end{align*}
the Levi factor of the residue of the connection $d'$ around each puncture $x$. The correspondence also holds in the other direction.
\end{prop}
	
\begin{proof}
We only give the proof for the case that the residue of the Higgs field is semisimple, and the general case is treated similarly to the local study in \S\ref{subsect local study}. Also, it is enough to work around a puncture $x \in D$ with the help of Lemma \ref{lem higgs and connect 1}. Given a tame harmonic $\boldsymbol\alpha$-adapted $G$-Higgs bundle $(E,\partial''_E, \phi,h)$, suppose that we have $\partial''_E=\partial''_0$ around the puncture $x$ with a good choice of trivialization $e_0$. Then, we have
\begin{align*}
	d'=\partial'_h + \phi, \quad d''= \partial''_0 + \phi^*,
\end{align*}
where $\partial'_h=\partial'_0 + \alpha \frac{dz}{z}$. To simplify the notation in the proof, we omit the subscript $x$ in the notation for convenience. The jump value of $h$ with respect to this trivialization $e$ is $\alpha$ and, without loss of generality, suppose that $h(e,e)=|z|^{2 \alpha}$. Moreover, when we say ``residue", it actually means the Levi factor of the residue.
		
\emph{Step 0}: Take $e_0= e g_0$, where $g_0= | z|^{-\alpha}$. Under the trivialization $e_0$, the operators become
\begin{align*}
	\partial'_0 + \frac{1}{2}\alpha \frac{dz}{z}+ \phi_0, \quad \partial''_0 - \frac{1}{2} \alpha \frac{d\bar{z}}{\bar{z}} + \phi_0^*,
\end{align*}
where
\begin{align*}
	\phi_0 =  {\rm Ad}(g_0) \phi, \quad \phi_0^* = {\rm Ad}(g_0) \phi^*.
\end{align*}
Then, $h(e_0,e_0) \sim 1$ and the limit $\lim\limits_{r \rightarrow 0} r\phi(r)$ is bounded.
		
\emph{Step 1}: We can find a complex gauge transformation $g_1$ such that
\begin{align*}
	\partial'_0 g_1 \cdot g^{-1}_1 & = - \frac{1}{2} {\rm Ad}(g_1) \phi_0 - \frac{1}{2} {\rm Ad}(g_1) \bar{\phi}^*_0; \\
	\partial''_0 g_1 \cdot g^{-1}_1 & =  - \frac{1}{2} {\rm Ad}(g_1) \bar{\phi}_0 - \frac{1}{2} {\rm Ad}(g_1) \phi^*_0.
\end{align*}
In fact, these two equations are equivalent to
\begin{align*}
	g_1^{-1} \cdot \partial'_0 g_1 & = -\frac{1}{2}(\phi_0 + \bar{\phi}^*_0);\\
	g_1^{-1} \cdot \partial''_0 g_1 & = -\frac{1}{2}(\bar{\phi}_0 + \phi^*_0),
\end{align*}
and it is clear that the element $g_1$ exists. Under the trivialization $e_1=e_0 g_1$, we have the operators
\begin{align*}
\partial'_0 + \phi_1, \quad \partial''_0 - \phi_1^*,
\end{align*}
where
\begin{align*}
	\phi_1 = \frac{1}{2} {\rm Ad}(g_1)  \left( \alpha  \frac{dz}{z} + \phi_0 - \bar{\phi}^*_0  \right), \quad \phi_1^* =  \frac{1}{2} {\rm Ad}(g_1) \left( \alpha \frac{d\bar{z}}{\bar{z}}  + \bar{\phi}_0 - \phi^*_0  \right).
\end{align*}
As we studied in \S\ref{subsect local study}, we have
\begin{align*}
	h(e_1,e_1) \sim r^{- 2(s_{\alpha} + \bar{s}_{\alpha})},
\end{align*}
the limit $\lim\limits_{r \rightarrow 0} {\rm Ad}(r^{s_{\alpha} + \bar{s}_{s_{\alpha}}}) (r \phi_1(r))$ is bounded and the residue of $\phi_1$ is $\frac{1}{2}(\alpha + s_\alpha - \bar{s}_\alpha)$.
		
\emph{Step 2}: We can find a second complex gauge transformation $g_2$ such that
\begin{align*}
	\partial'_0 g_2 \cdot g_2^{-1} & =  -{\rm Ad}(g_2) \phi^*_1,\\
	\partial''_0 g_2 \cdot g_2^{-1} & = -{\rm Ad}(g_2) \bar{\phi}^*_1.
\end{align*}
Under the trivialization $e_2=e_1 g_2$, the $(0,1)$-type operator becomes trivial and the $(1,0)$-type operator is
\begin{align*}
	\partial'_0 + \phi_2 - \bar{\phi}_2^*,
\end{align*}
where
\begin{align*}
	\phi_2 = {\rm Ad}(g_2) \phi_1, \quad \phi_2^* = {\rm Ad}(g_2) \phi^*_1.
\end{align*}
Furthermore,
\begin{align*}
	h(e_2,e_2) \sim r^{2 \alpha - 2(s_{\alpha} + \bar{s}_{\alpha})},
\end{align*}
the limit
\begin{align*}
	\lim\limits_{r \rightarrow 0} {\rm Ad}(r^{\alpha - (s_{\alpha} + \bar{s}_{s_{\alpha}})}) \left( r(\phi_2(r) - \bar{\phi}_2^*(r)) \right)
\end{align*}
is bounded and the residue of $\phi_2$ is $\alpha + s_\alpha - \bar{s}_\alpha$.
		
In conclusion, we obtain a holomorphic connection $\partial'_0+\phi_2 - \bar{\phi}_2^*$ with weight $\alpha-(s_\alpha + \bar{s}_\alpha)$ and residue $\alpha+s_\alpha-\bar{s}_\alpha$, which is exactly the data stated in the proposition. Thus, given a tame harmonic $\boldsymbol\alpha$-adapted $G$-Higgs bundle, we obtain a tame harmonic $\boldsymbol\beta$-adapted $G$-connection with the desired data.
		
The other direction can be proved similarly. This finishes the proof of this proposition.
\end{proof}
	
\subsection{Stability Condition}\label{subsect ana stab Higgs and Dmod}
	
\begin{defn}
A tame metrized $\boldsymbol\theta$-adapted $G$-Higgs bundle $(E,\partial''_E, \phi,h)$ is called \emph{$R_h$-stable} (resp. \emph{$R_h$-semistable}), if for
\begin{itemize}
	\item any proper parabolic group $P \subseteq G$,
	\item any compatible $\boldsymbol\theta$-adapted holomorphic reduction of structure group $\sigma: X_{\boldsymbol D} \rightarrow E/P$,
	\item any nontrivial anti-dominant character $\chi: P \rightarrow \mathbb{G}_m$, which is trivial on the center of $P$,
\end{itemize}
one has
\begin{align*}
	\deg^{\rm an} E (h, \sigma, \chi) > 0, \quad (\text{resp. } \geq 0).
\end{align*}
\end{defn}

\begin{defn}
A tame metrized $\boldsymbol\theta$-adapted $G$-connection $(V,d',h)$ is called \emph{$R_h$-stable} (resp. \emph{$R_h$-semistable}), if for
\begin{itemize}
	\item any proper parabolic group $P \subseteq G$,
	\item any compatible $\boldsymbol\theta$-adapted holomorphic reduction of structure group $\sigma: X_{\boldsymbol D} \rightarrow V/P$,
	\item any nontrivial anti-dominant character $\chi: P \rightarrow \mathbb{G}_m$, which is trivial on the center of $P$,
\end{itemize}
one has
\begin{align*}
	\deg^{\rm an} V (h, \sigma,\chi) > 0, \quad (\text{resp. } \geq 0).
\end{align*}
\end{defn}
	
The definitions of polystability for tame metrized $G$-Higgs bundles and $G$-connections are analogous to Definition \ref{defn stab poly} by adding the compatibility condition of reductions of structure group, thus we skip the definition here.
	
	
\begin{lem}\label{lem ana sub Higgs and sub mod}
Given an $\boldsymbol\alpha$-adapted metric $h$, there is a one-to-one correspondence between compatible $\boldsymbol\alpha$-adapted reductions of structure group of $(E,\partial''_E,\phi,h)$ and compatible $\boldsymbol\beta$-adapted reductions of structure group of $(V,d',h)$.
\end{lem}
	
\begin{proof}
The proof of this lemma is exactly the same as that of Proposition \ref{prop tame ana Higgs and conn}. The idea is that we start with a compatible $\boldsymbol\alpha$-adapted holomorphic reduction of structure group $\sigma$, and the corresponding $P$-Higgs bundle is $(E_\sigma, \phi_\sigma)$. Applying the gauge actions given in Proposition \ref{prop tame ana Higgs and conn} to $\sigma$, we obtain a $\boldsymbol\beta$-adapted holomorphic reduction of structure group with respect to the holomorphic structure $d''$, and the compatibility follows directly from the construction of $d'$.
\end{proof}

We introduce the following notions:
\begin{itemize}
\item $\mathcal{C}_{\rm Dol}(X_{\boldsymbol D},G,\boldsymbol\alpha, \phi_{\boldsymbol\alpha})$: the category of $R_h$-stable tame harmonic $\boldsymbol\alpha$-adapted $G$-Higgs bundles of degree zero on $X_{\boldsymbol D}$, and the Levi factors of residues of the Higgs field are $\phi_{\boldsymbol\alpha}$ at punctures;
\item $\mathcal{C}_{\rm dR}(X_{\boldsymbol D},G,\boldsymbol\beta,d'_{\boldsymbol\beta})$: the category of $R_h$-stable tame harmonic $\boldsymbol\beta$-adapted $G$-connections of degree zero on $X_{\boldsymbol D}$, and the Levi factors of residues of the connection are $d'_{\boldsymbol\beta}$ at punctures,
\end{itemize}
where
\begin{align*}
	\phi_{\boldsymbol\alpha}:=\{\phi_{\alpha_x}, x \in \boldsymbol D\} \quad \text{and} \quad d'_{\boldsymbol\beta}:=\{d'_{\beta_x}, x \in \boldsymbol D\}
\end{align*}
are collections of elements in $\mathfrak{g}$. We would like to explain the terminology here. In the classical case, the Kobayashi--Hitchin correspondence shows that harmonicity is equivalent to the polystability condition. We denote the category by including both harmonicity and stability at this point because a version of Kobayashi--Hitchin correspondence has not been proved yet; this will be proven later on in \S\ref{subsect Kob-Hitchin}.
	
\begin{thm}\label{thm ana Dol and DR}
Fix an $\boldsymbol\alpha$-adapted metric $h$. Let $(E,\partial''_E, \phi,h)$ be a tame harmonic $\boldsymbol\alpha$-adapted $G$-Higgs bundle, and let $(V,d',h)$ be the corresponding tame harmonic $\boldsymbol\beta$-adapted $G$-connection. The tame harmonic $G$-Higgs bundle $(E,\partial''_E,\phi,h)$ is $R_h$-stable (resp. $R_h$-polystable) if and only if the corresponding tame harmonic $G$-connection $(V,d',h)$ is $R_h$-stable (resp. $R_h$-polystable). Furthermore, the categories $\mathcal{C}_{\rm Dol}(X_{\boldsymbol D},G,\boldsymbol\alpha, \phi_{\boldsymbol\alpha})$ and $\mathcal{C}_{\rm dR}(X_{\boldsymbol D},G,\boldsymbol\beta,d'_{\boldsymbol\beta})$ are equivalent.
\end{thm}
	
\begin{proof}
By Lemma \ref{lem ana sub Higgs and sub mod}, if $\sigma$ is a compatible $\boldsymbol\alpha$-adapted reduction of structure group of $(E,\partial''_E,\phi,h)$, then and only then it is a compatible $\boldsymbol\beta$-adapted reduction of structure group of $(V,d',h)$. Based on Proposition \ref{prop tame ana Higgs and conn}, the correspondence follows directly.
\end{proof}

\section{Logahoric Higgs Torsors and Logahoric Connections}\label{sect alg Higgs and conn}
	
In this section, we extend the functor $\Xi$ studied in \S\ref{subsect functor Xi} to categories of Higgs bundles and connections. Let $X$ be a smooth algebraic curve, and denote by $K_X$ the canonical bundle on $X$. A parahoric $\mathcal{G}_{\boldsymbol\theta}$-torsor $\mathcal{E}$ on $X$ can be understood from its local charts $(E, E_x)$, where $E$ is a $G$-bundle on $X_{\boldsymbol D}$ and $E_x$ is a $\mathcal{G}_{\theta_x}$-torsor on $\mathbb{D}_x$. By taking adjoint representations, the adjoint bundles $E(\mathfrak{g})$ on $X_{\boldsymbol D}$ and $E_x(\mathfrak{g})$ on $\mathbb{D}_x$ are defined in a natural way. Under a trivialization around $x$, sections of $E_x(\mathfrak{g})$ are regarded as $\mathfrak{g}_{\theta_x}(K)$ (see \S\ref{sect parah grp} for the notation). Patching them together via the same transition functions as $\mathcal{E}$, we get an \emph{adjoint bundle} on $X$ which is denoted by $\mathcal{E}(\mathfrak{g})$.
	
\begin{defn}\label{defn alg parah Higgs}
A \emph{logahoric $\mathcal{G}_{\boldsymbol\theta}$-Higgs torsor} on a smooth algebraic curve $X$ is a pair $(\mathcal{E},\varphi)$, where
\begin{itemize}
	\item $\mathcal{E}$ is a parahoric $\mathcal{G}_{\boldsymbol\theta}$-torsor on $X$;
	\item $\varphi \in H^0(X, \mathcal{E}(\mathfrak{g}) \otimes K_X(\boldsymbol D))$ is a section.
\end{itemize}
The section $\varphi$ is called a \emph{logarithmic Higgs field}.
\end{defn}
We would like to repeat at this point Remark \ref{rem conv holo and alg}; namely, the definition above is a purely algebraic one, therefore we do not use the term \emph{holomorphic section}.  Similarly to parahoric torsors (cf. \S \ref{sect parah tor}), one naturally defines equivalence classes of logahoric Higgs torsors. Next, we define \emph{logahoric connections}. We only give the key definition and refer the reader to Appendix \ref{sect appendix} for more details.
\begin{defn}\label{defn alg parah conn}
Let $\mathcal{E}$ be a parahoric $\mathcal{G}_{\boldsymbol\theta}$-torsor. A \emph{logarithmic connection} on $\mathcal{E}$ is a $\mathbb{C}$-linear map $\nabla:\mathcal{O}_{\mathcal{E}} \rightarrow \mathcal{O}_{\mathcal{E}} \otimes K_X(\boldsymbol D)$ such that
\begin{enumerate}
	\item $\nabla$ satisfies the Leibniz rule;
	\item the diagram commutes, 
	\begin{center}
		\begin{tikzcd}
		\mathcal{O}_{\mathcal{E}} \arrow[rr,"\nabla"] \arrow[d,"a"] & & \mathcal{O}_{\mathcal{E}} \otimes K_X(\boldsymbol D) \arrow[d,"a \otimes 1"] \\
		\mathcal{O}_{\mathcal{E}} \otimes \mathcal{O}_{\mathcal{G}_{\boldsymbol\theta}} \arrow[rr,"\nabla \otimes 1 + 1 \otimes \nabla_{\mathcal{G}_{\boldsymbol\theta}}"] & & (\mathcal{O}_{\mathcal{E}} \otimes \mathcal{O}_{\mathcal{G}_{\boldsymbol\theta}}) \otimes K_X(\boldsymbol D)
		\end{tikzcd}
	\end{center}
	where $a: \mathcal{O}_{\mathcal{E}} \rightarrow \mathcal{O}_{\mathcal{E}} \otimes \mathcal{O}_{\mathcal{G}_{\boldsymbol\theta}}$ is the co-action map and $\nabla_{\mathcal{G}_{\boldsymbol\theta}}$ is the canonical integrable connection on $\mathcal{G}_{\boldsymbol\theta}$.
\end{enumerate}
\end{defn}

A logarithmic connection $\nabla$ corresponds to a section of $\mathcal{E}(\mathfrak{g}) \otimes K_X(\boldsymbol D)$ locally, which is regarded as the \emph{connection form} of $\nabla$.
	
\begin{defn}\label{defn tame loga mod}
A \emph{logahoric $\mathcal{G}_{\boldsymbol\theta}$-connection} on $X$ is a pair $(\mathcal{E},\nabla)$, where $\mathcal{E}$ is a parahoric $\mathcal{G}_{\boldsymbol\theta}$-torsor and $\nabla$ is a logarithmic connection on $\mathcal{E}$.
\end{defn}
	
In \S\ref{sect parah tor}, we have seen that there is a correspondence between $\boldsymbol\theta$-adapted $G$-bundles on $X_{\boldsymbol D}$ and parahoric $\mathcal{G}_{\boldsymbol\theta}$-torsors on $X$. In the sequel, we include the additional information given by Higgs fields (resp. connections) to get a correspondence between \emph{tame metrized $\boldsymbol\theta$-adapted $G$-Higgs bundles} (resp. \emph{$G$-connections}) on $X_{\boldsymbol D}$ and \emph{logahoric $\mathcal{G}_{\boldsymbol\theta}$-Higgs torsors} (resp. \emph{logahoric $\mathcal{G}_{\boldsymbol\theta}$-connections}) on $X$.
	
A tame $\boldsymbol\theta$-adapted metrized $G$-Higgs bundle $(E,\partial''_E,\phi,h)$ corresponds to a parahoric $\mathcal{G}_{\boldsymbol\theta}$-torsor $\mathcal{E}$. Note that around each puncture $x \in \boldsymbol D$, the Higgs field $\phi$ satisfies the tameness condition that $z^{\theta_x} \cdot (z\phi(z)) \cdot z^{-\theta_x}$ is bounded as $z$ approaches zero (see Definition \ref{defn tame ana Higgs}). By Lemma \ref{lem two defn parah Lie alg}, the Higgs field $\phi$ naturally extends to a logarithmic Higgs field $\varphi: X \rightarrow \mathcal{E}(\mathfrak{g}) \otimes K_X(\boldsymbol D)$ . Therefore, we associate $(E,\partial''_E,\phi,h)$ with a logahoric $\mathcal{G}_{\boldsymbol\theta}$-Higgs torsor $(\mathcal{E},\varphi)$, and thus establish the functor
\begin{align*}
	\Xi_{\rm Higgs}:  \mathcal{C}_{\rm Higgs}(X_{\boldsymbol D},G,\boldsymbol\theta) \rightarrow \mathcal{C}_{\rm Higgs}(X,\mathcal{G}_{\boldsymbol\theta}),
\end{align*}
where $\mathcal{C}_{\rm Higgs}(X_{\boldsymbol D},G,\boldsymbol\theta)$ is the category of tame metrized $\boldsymbol\theta$-adapted $G$-Higgs bundles on $X_{\boldsymbol D}$ and $\mathcal{C}_{\rm Higgs}(X,\mathcal{G}_{\boldsymbol\theta})$ is the category of logahoric $\mathcal{G}_{\boldsymbol\theta}$-Higgs torsors $(\mathcal{E},\varphi)$ on $X$.
	
The argument for connections is similar. Given a tame metrized $\boldsymbol\theta$-adapted $G$-connection $(V,d',h)$, the holomorphic bundle $V$ can be extended to a parahoric $\mathcal{G}_{\boldsymbol\theta}$-torsor $\mathcal{V}$ under the functor $\Xi$. Under the tameness condition for $d'$, the connection $d'$ corresponds to a logarithmic connection on $\mathcal{V}$ and denote by $\nabla$ the corresponding connection. Therefore, we have the functor
\begin{align*}
	\Xi_{\rm Conn}:  \mathcal{C}_{\rm Conn}(X_{\boldsymbol D},G,\boldsymbol\theta) \rightarrow \mathcal{C}_{\rm Conn}(X,\mathcal{G}_{\boldsymbol\theta}),
\end{align*}
where $\mathcal{C}_{\rm Conn}(X_{\boldsymbol D},G,\boldsymbol\theta)$ is the category of tame metrized $\boldsymbol\theta$-adapted $G$-connections on $X_{\boldsymbol D}$, and similarly, $\mathcal{C}_{\rm Conn}(X,\mathcal{G}_{\boldsymbol\theta})$ is the category of logahoric $\mathcal{G}_{\boldsymbol\theta}$-connections on $X$.
	
After establishing the correspondence for Higgs bundles and connections, we move to stability conditions. We first consider the stability condition for logahoric $\mathcal{G}_{\boldsymbol\theta}$-Higgs torsors on $X$. Let $(\mathcal{E},\varphi)$ be a parahoric $\mathcal{G}_{\boldsymbol\theta}$-Higgs torsor on $X$. A reduction of structure group $\varsigma: X \rightarrow \mathcal{E}/\mathcal{P}_{\boldsymbol \theta}$ is said to be \emph{$\varphi$-compatible} (or \emph{compatible} in short), if there is a lifting $\varphi_{\varsigma}:X \rightarrow \mathcal{E}_{\varsigma}(\mathfrak{p}) \otimes K_X(\boldsymbol D)$ such that the following diagram commutes
	\begin{center}
		\begin{tikzcd}
		& \mathcal{E}_{\varsigma}(\mathfrak{p}) \otimes K_X(\boldsymbol D) \arrow[d, hook] \\
		X \arrow[r, "\varphi"] \arrow[ur, "\varphi_{\varsigma}", dotted] & \mathcal{E}(\mathfrak{g}) \otimes K_X(\boldsymbol D).
		\end{tikzcd}
	\end{center}
	
	\begin{defn}
		A logahoric $\mathcal{G}_{\boldsymbol\theta}$-Higgs torsor $(\mathcal{E},\varphi)$ is called \emph{$R$-stable} (resp. \emph{$R$-semistable}), if for
		\begin{itemize}
			\item any proper parabolic group $P \subseteq G$,
			\item any compatible reduction of structure group $\varsigma: X \rightarrow \mathcal{E}/\mathcal{P}_{\boldsymbol\theta}$,
			\item any nontrivial anti-dominant character $\kappa: \mathcal{P}_{\boldsymbol\theta} \rightarrow \mathbb{G}_m$, which is trivial on the center of $\mathcal{P}_{\boldsymbol\theta}$,
		\end{itemize}
		one has
		\begin{align*}
		parh\deg \mathcal{E}(\varsigma,\kappa) > 0, \quad (\text{resp. } \geq 0).
		\end{align*}
	\end{defn}
	
For a logahoric $\mathcal{G}_{\boldsymbol\theta}$-connection $(\mathcal{V},\nabla)$, a reduction of structure group $\varsigma: X \rightarrow \mathcal{V}/\mathcal{P}_{\boldsymbol \theta}$ is \emph{$\nabla$-compatible} (or \emph{compatible} in short), if there is a lifting $\nabla_{\varsigma}:\mathcal{O}_{\mathcal{V}_{\varsigma}} \rightarrow\mathcal{O}_{\mathcal{V}_{\varsigma}} \otimes K_X(\boldsymbol D)$ such that the following diagram commutes
\begin{center}
\begin{tikzcd}
	\mathcal{O}_{\mathcal{V}} \arrow[r,"\nabla"] \arrow[d] & \mathcal{O}_{\mathcal{V}} \otimes K_X(\boldsymbol D) \arrow[d]  \\
	\mathcal{O}_{\mathcal{V}_{\varsigma}} \arrow[r,"\nabla_{\varsigma}",dotted] & \mathcal{O}_{\mathcal{V}_{\varsigma}} \otimes K_X(\boldsymbol D) .
\end{tikzcd}
\end{center}
	
\begin{defn}
A logahoric $\mathcal{G}_{\boldsymbol\theta}$-connection $(\mathcal{V},\nabla)$ is called \emph{$R$-stable} (resp. \emph{$R$-semistable}), if for
\begin{itemize}
	\item any proper parabolic group $P \subseteq G$,
	\item any compatible reduction of structure group $\varsigma: X \rightarrow \mathcal{V}/\mathcal{P}_{\boldsymbol\theta}$,
	\item any nontrivial anti-dominant character $\kappa: \mathcal{P}_{\boldsymbol\theta} \rightarrow \mathbb{G}_m$, which is trivial on the center of $\mathcal{P}_{\boldsymbol\theta}$,
\end{itemize}
one has
\begin{align*}
	parh\deg \mathcal{V}(\varsigma,\kappa) > 0, \quad (\text{resp. } \geq 0).
\end{align*}
\end{defn}
	
	
Lemma \ref{lem red grp struc} gives a one-to-one correspondence between $\boldsymbol\theta$-adapted holomorphic reductions of group structure of $E$ and reductions of group structure of $\mathcal{E}$. This result can be extended to include Higgs fields and connections.
	
\begin{lem}\label{lem red grp struc Higgs and conn}
Let $(E,\partial''_E,\phi,h)$ (resp. $(V,d',h)$) be a tame metrized $\boldsymbol\theta$-adapted $G$-Higgs bundle (resp. $G$-connection) on $X_{\boldsymbol D}$. Denote by $(\mathcal{E}, \varphi)$ (resp. $(\mathcal{V},\nabla)$) the corresponding logahoric $\mathcal{G}_{\boldsymbol\theta}$-Higgs torsor (resp. logahoric $\mathcal{G}_{\boldsymbol\theta}$-connection). Given a $\boldsymbol\theta$-adapted holomorphic reduction of structure group $\sigma: X_{\boldsymbol D} \rightarrow E/P$, it is compatible with $\phi$ (resp. $d'$) if and only if the corresponding reduction of structure group $\varsigma$ is compatible with $\varphi$ (resp. $\nabla$).
\end{lem}
	
\begin{proof}
The proof of this lemma is similar to the one for Lemma \ref{lem red grp struc}.
\end{proof}
	
\begin{prop}\label{prop ana=alg Higgs and conn}
A tame metrized $\boldsymbol\theta$-adapted $G$-Higgs bundle is $R_h$-stable (resp. $R_h$-semistable) if and only if the corresponding logahoric $\mathcal{G}_{\boldsymbol\theta}$-Higgs torsor is $R$-stable (resp. $R$-semistable). Similarly, a tame metrized $\boldsymbol\theta$-adapted $G$-connection is $R_h$-stable (resp. $R_h$-semistable) if and only if the corresponding logahoric $\mathcal{G}_{\boldsymbol\theta}$-connection is $R$-stable (resp. $R$-semistable).
\end{prop}
	
\begin{proof}
This proposition is a direct result of Lemma \ref{lem red grp struc Higgs and conn} and Propositions \ref{prop analytic=algebraic degree}, \ref{prop stab of parah tor}.
\end{proof}
	
	Let
	\begin{align*}
	\varphi_{\boldsymbol\theta}=\{\varphi_{\theta_x}, \text{ } x \in \boldsymbol{D}\}, \quad \nabla_{\boldsymbol\theta}=\{\nabla_{\theta_x}, \text{ } x \in \boldsymbol{D}\}
	\end{align*}
	be two collections of elements (as residues) in $\mathfrak{g}$ indexed by punctures. We introduce the following notions:
\begin{itemize}
\item $\mathcal{C}_{\rm Dol}(X,\mathcal{G}_{\boldsymbol\theta},\varphi_{\boldsymbol\theta})$: the category of $R$-stable logahoric $\mathcal{G}_{\boldsymbol\theta}$-Higgs torsors of degree zero on $X$, and the Levi factors of residues of the Higgs field are $\varphi_{\boldsymbol\theta}$ at punctures;
\item $\mathcal{C}_{\rm dR}(X,\mathcal{G}_{\boldsymbol\theta},\nabla_{\boldsymbol\theta})$: the category of $R$-stable logahoric $\mathcal{G}_{\boldsymbol\theta}$-connections of degree zero on $X$ and the Levi factors of residues of the connection are $\nabla_{\boldsymbol\theta}$ at punctures.
\end{itemize}
Proposition \ref{prop ana=alg Higgs and conn} shows that the functor $\Xi_{\rm Higgs}$ (resp. $\Xi_{\rm Conn}$) induces a well-defined one
	\begin{align*}
	\Xi_{\rm Dol}: \mathcal{C}_{\rm Dol}(X_{\boldsymbol D},G,\boldsymbol\theta, \phi_{\boldsymbol\theta}) \rightarrow \mathcal{C}_{\rm Dol}(X_{\boldsymbol D},\mathcal{G}_{\boldsymbol\theta},\varphi_{\boldsymbol\theta}) \quad (\text{resp. } \Xi_{\rm dR}: \mathcal{C}_{\rm dR}(X_{\boldsymbol D},G,\boldsymbol\theta,d'_{\boldsymbol\theta}) \rightarrow \mathcal{C}_{\rm dR}(X_{\boldsymbol D},\mathcal{G}_{\boldsymbol\theta},\nabla_{\boldsymbol\theta}))
	\end{align*}
	between the Dolbeault categories (resp. de Rham categories), where $\phi_{\boldsymbol\theta}=\varphi_{\boldsymbol\theta}$ (resp. $d'_{\boldsymbol\theta}=\nabla_{\boldsymbol\theta}$).

\section{Tame Parahoric Nonabelian Hodge Correspondence: Category}\label{sec_categorical naHc}

We are now ready to establish the tame parahoric nonabelian Hodge correspondence at the level of categories. The first step is a Kobayashi--Hitchin correspondence relating $R_h$-polystability with the existence of a $\boldsymbol \theta$-adapted harmonic metric $h$ on a tame metrized $\boldsymbol\theta$-adapted $G$-Higgs bundle $(E,\partial''_E,\phi)$ of degree zero on $X_{\boldsymbol D}$ (Theorem \ref{thm KH corr Higgs}). Here, by a \emph{harmonic metric} $h$, we mean one that solves the Hermite--Einstein equation and such a solution then induces a flat connection $D$ as in \eqref{induced_flat_connect}; this explains the slight abuse of notation for harmonicity that appeared in Definition \ref{defn har bun Dol+DR}. Given the Kobayashi--Hitchin correspondence, we obtain a bijection between objects in the Dolbeault and the de Rham categories (Theorem \ref{thm alg Dol and DR}). The next step involves a Riemann--Hilbert correspondence describing a bijection between objects in the de Rham and Betti categories. In \S\ref{subsect Riemann-Hilbert cat} we profit from the local behavior of the correspondence studied by Boalch in \cite{Bo} and show how the tame harmonic $G$-connections introduced here align with that work. The outcome will be a complete diagram describing the correspondences at the level of categories as follows:
\begin{center}
	\begin{tikzcd}
	\mathcal{C}_{\rm Dol}(X_{\boldsymbol D},G,\boldsymbol\alpha,\phi_{\boldsymbol\alpha}) \arrow[rr, equal] \arrow[d, equal] & &
	\mathcal{C}_{\rm dR}(X_{\boldsymbol D},G,\boldsymbol\beta,d'_{\boldsymbol\beta}) \arrow[d, equal] & &\\
	\mathcal{C}_{\rm Dol}(X,\mathcal{G}_{\boldsymbol\alpha},\varphi_{\boldsymbol\alpha}) \arrow[rr, equal]& & \mathcal{C}_{\rm dR}(X,\mathcal{G}_{\boldsymbol\beta},\nabla_{\boldsymbol\beta}) \arrow[rr, equal] & & \mathcal{C}_{\rm B}(X_{\boldsymbol D}, G,\boldsymbol\gamma, M_{\boldsymbol\gamma}),
	\end{tikzcd}
\end{center}
for the terms introduced in \S\ref{subsect local study}, and $\phi_{\boldsymbol\alpha}=\varphi_{\boldsymbol\alpha}, d'_{\boldsymbol\beta}=\nabla_{\boldsymbol\beta}$.

\subsection{Kobayashi--Hitchin Correspondence}\label{subsect Kob-Hitchin}

Let $(E,\partial''_E,\phi)$ be a tame $\boldsymbol\theta$-adapted $G$-Higgs bundle on $X_{\boldsymbol D}$. The purpose of this section is to prove a Kobayashi--Hitchin correspondence relating the polystability condition for the triple $(E,\partial''_E, \phi)$ with the existence of a hermitian metric $h$ on $E$ satisfying the Hermite--Einstein equation
\begin{equation}\label{Hermite-Einstein eq}
F_h+\left[ \phi ,{{\phi }^{*}} \right]=0,
\end{equation}
where $F_h$ is the curvature of the Chern connection associated to the metric $h$ and the holomorphic structure $\partial''_E$, while
${{\phi }^{*}}$ is the adjoint section of type $(0,1)$ defined by
\[{{\left( {{\phi }^{*}} \right)}^{T}}=\text{Ad}(h)\bar{\phi }.\]
Note that a metric $h$ satisfying the above equation also implies that the induced connection $D$ is flat, and therefore $h$ is exactly a harmonic metric as considered in Definition \ref{defn har bun Dol+DR}.

The strategy of the proof of the correspondence has so far been applied in a large variety of similar situations. Among those, we single out pioneering works for the correspondence on noncompact K\"{a}hler manifolds. Mehta and Seshadri \cite{Meht} first provided the analog of the classical Narasimhan--Seshadri correspondence in the context of punctured Riemann surfaces. Later on, Simpson \cite{Simp} developed the geometric analytic tools to study Hermite--Einstein metrics over noncompact curves for the group $\text{GL}_n( \mathbb{C})$, while the higher dimensional generalization to obtain  a complete correspondence between parabolic Higgs bundles and logarithmic integrable connections was due to Biquard for the case of smooth divisors \cite{Biquard}, and Mochizuki for normal crossing divisors \cite{Moc1,Moc2}. We give below the proof of the correspondence for our parahoric situation modeled on the one given in \cite{BGM}, which in turn follows the one in \cite{Biquard}. Since the context is very similar to that of \cite[\S 5]{BGM}, we only exhibit here how the key steps of the proof adapt for our triples $(E,\partial''_E,\phi)$.

\begin{thm}[Kobayashi--Hitchin Correspondence for $G$-Higgs Bundles]\label{thm KH corr Higgs}
	Let $(E,\partial''_E, \phi)$ be a tame $\boldsymbol\theta$-adapted $G$-Higgs bundle of degree zero equipped with a $\boldsymbol\theta$-adapted metric $h_0$. Then, $(E,\partial''_E, \phi)$ admits a $\boldsymbol\theta$-adapted harmonic metric $h$ which is quasi-isometric to $h_0$ if and only if it is $R_{h_0}$-polystable. Moreover, this metric is unique up to automorphisms of Higgs bundles.
\end{thm}

\begin{proof}
We first prove that the polystability condition is necessary. Suppose that there exists a harmonic metric $h$ on $(E,\partial''_E,\phi)$. Since $h$ is quasi-isometric to $h_0$, the $R_{h}$-stability condition is equivalent to the $R_{h_0}$-stability condition. Therefore, we only have to show that the $G$-Higgs bundle is $R_h$-polystable. Let $\sigma$ be a compatible $\boldsymbol\theta$-adapted holomorphic reduction of group structure. With a similar calculation as in \S\ref{subsect ana Hig ana D-mod}, we derive
\begin{align*}
	F_{h_\sigma} = F_{D_\sigma} - 2 F_{D''_\sigma} - [\phi_\sigma , \phi^*_\sigma],
\end{align*}
where $D_\sigma:=\sigma^* D$, $D''_\sigma:=\sigma^* D''$ and the other terms are the same as in  \S\ref{subsect ana deg on bundle} and \S\ref{subsect ana Hig ana D-mod}. Since the reduction $\sigma: X_{\boldsymbol D} \rightarrow E/P$ is holomorphic, the pseudo-curvature $F_{D''_\sigma} = (D''_\sigma)^2=0$ is trivial. Now we consider the term $F_{D_\sigma}$. With a similar calculation as we did in Lemma \ref{lem Chern-Weil}, we get
\begin{align*}
	F_{D_\sigma}=\sigma^* F_D + 2 h( D''(\sigma), D''(\sigma)),
\end{align*}
where $h(\cdot,\cdot)$ denotes the hermitian metric on the bundle $E(\mathfrak{g})\otimes T^*_\mathbb{C} X$ (sections of this bundle are $1$-forms on $E(\mathfrak{g})$) induced from  the hermitian metric $h$ on $E(\mathfrak{g})$ (see Remark \ref{rem hermitian metric}) and $\sigma$ is regarded as a section $X_{\boldsymbol D} \rightarrow E(\mathfrak{u}^-)$ (see \S\ref{subsect deg Higgs and mod}). Since $\partial''_E(\sigma)=0$, we have
\begin{align*}
	F_{D_\sigma}= \sigma^* F_D + 2 h([\phi,\sigma],[\phi,\sigma]),
\end{align*}
and since $h$ is a harmonic metric, we have $F_D=0$, thus
\begin{align*}
	F_{h_\sigma} = 2 h([\phi,\sigma],[\phi,\sigma]) - [\phi,\phi^*].
\end{align*}
Note that the second term $[\phi,\phi^*]$ does not contribute to the degree, and therefore, we have the analytic degree $\deg^{\rm an} E(h, \sigma, \chi) \geq 0$ for an arbitrary nontrivial antidominant character $\chi$. Therefore, the tame $\boldsymbol\theta$-adapted $G$-Higgs bundle is $R_h$-semistable. The analytic degree is zero if and only if $[\phi,\sigma]=0$. In this case, we have a natural lifting $\phi_{\sigma_L}: X \rightarrow E_{\sigma_L}(\mathfrak{l}) \otimes K_{X_{\boldsymbol D}}$ and obtain a well-defined semistable $L$-Higgs bundle $(E_{\sigma_L},\phi_{\sigma_L})$ of degree zero. If $(E_{\sigma_L},\phi_{\sigma_L})$ is stable, then $(E,\partial''_E,\phi,h)$ is $R_h$-polystable. If not, we iterate the arguments above and stop at a stable one.
	
For the converse statement, assume that the triple $(E,\partial''_E, \phi)$ is $R_{h_{0}}$-polystable of degree zero for a $\boldsymbol\theta$-adapted harmonic metric $h_0$. We want to show that $R_{h_{0}}$-polystability provides the existence of a solution to the Hermite--Einstein equation (\ref{Hermite-Einstein eq}) in the space of metrics
	\[\mathcal{H}=\{ h = h_0 e^s, \text{ } s \in \hat{L}^{2,p}_{\delta}(E( \text{Im}(\mathfrak{g})))\},\]
for a small positive $\delta$ and a large positive $p$, where $\hat{L}^{2,p}_{\delta}$ denotes the weighted Sobolev space (see \cite{BGM}) and $\text{Im}(\mathfrak{g})$ denotes the imaginary part of the Lie algebra $\mathfrak{g}$. Note that a metric $h \in \mathcal{H}$ is $\boldsymbol\theta$-adapted by definition. For a pair of metrics $h_0$, $h_0 e^s \in \mathcal{H}$, we consider a Donaldson functional analogous to the one considered by Simpson in \cite{Simp1988}:
\begin{align*}
	M(h_0, h_0 e^s) = \int_{X_{\boldsymbol{D}}} \langle \sqrt{-1} \Lambda F_{h_0},s \rangle_{h_0} + \int_{X_{\boldsymbol{D}}} (\psi(s)(D''(s)),D''(s))_{h_0},
\end{align*}
where $\Lambda F_{h_0}$ is the contraction of $F_{h_0}$ with respect to the K\"ahler form on $X$, $\langle \cdot, \cdot \rangle_{h_0}$ and $(\cdot,\cdot)_{h_0}$ are explained in Remark \ref{rem notation herm}. We are extending the function $\psi(t)=\frac{e^t-t-1}{t^2}$ to sections of $E(\mathfrak{g})$. Evidently, the critical points of $M$ are Hermite--Einstein metrics. Following the scheme of the analogous proof in \cite[\S 5.6]{BGM}, we first remark that we can reduce to the case when the tuple $(E,\partial''_E, \phi)$ is $R_{h_{0}}$-stable of degree zero and that the solution we seek is of the form $h=h_0 e^s$ for a section $s$ in the semisimple part of $\mathfrak{g}$.
	
Now one needs to show that the Donaldson functional $M(h_0 ,h)$ minimizes for $h \in \mathcal{H}$ under the constraint ${{\left\| F(h) \right\|}_{L_{\delta }^{2,p}}}\le B$, for some large constant $B$. Fix a constant $B$ and define
\[{{\mathsf{\mathcal{S}}}^{\infty }}( B )=\left\{ s\in \hat{C}_{\delta }^{\infty }\left( E(\text{Im}(\mathfrak{g})) \right),\,\,{{\left\| F({{h}_{0}}{{e}^{s}}) \right\|}_{L_{\delta }^{p}}}\le B \right\}.\]
In the proof of \cite[Theorem 8.1]{Biquard}, Biquard shows how the existence of a solution $h$ reduces to the following proposition. A similar argument is also given in \cite[Proposition 5.3]{Simp1988}.
\begin{prop}\label{estimates_regularity}
Let $(E,\partial''_E, \phi)$ be a tame $\boldsymbol\theta$-adapted $G$-Higgs bundle over ${X}_{\boldsymbol D}$ equipped with a $\boldsymbol\theta$-adapted metric $h_0$ which is stable of degree zero. Then there exist constants $C$ and $C'$ such that
\begin{equation}\label{estimates_equation}
	{\rm sup} |s| \leq C +  C' M(h_0,h_0 e^s),
\end{equation}
for any section $ s \in {{\mathsf{\mathcal{S}}}^{\infty }}\left( B \right)$.
\end{prop}
	
\begin{proof}
We will show using the method of Uhlenbeck--Yau \cite{UY} that if there are no constants satisfying (\ref{estimates_equation}) then stability of the tuple $(E,\partial''_E, \phi)$ is violated. From \cite[Lemma 8.4]{Biquard} it follows that under this assumption there do not exist constants $C, C'$ such that
\begin{equation*}
	{{\left\| s \right\|}_{L^{1}}}\le C+  C' M(h_0,h_0 e^s),
\end{equation*}
for any section $ s \in {{\mathsf{\mathcal{S}}}^{\infty }}( B )$. Let $C_i$ be real constants with $\lim\limits_{i \rightarrow \infty} C_i = \infty$. Then, there is a sequence of sections $\{s_i\} \subset \mathcal{S}^{\infty}(B)$ such that ${\left\| s_i \right\|}_{L^{1}}\to \infty$ and ${{\left\| s_i \right\|}_{L^{1}}}\ge C_{i} M(h_0,h_0 e^{s_i})$.
		
Set $l_i:={\left\| s_i \right\|}_{L^{1}}$ and define $u_i :=l^{-1}_{i}s_{i}$, so ${\left\| u_i \right\|}_{L^{1}}=1$. Then, after passing to a subsequence, the sections $u_i$ converge weakly and locally in $L^{1,2}$ to a section $u_{\infty}$ of $E(\text{Im}(\mathfrak{g}))\rvert _{X_{\boldsymbol D}}$. This construction of $u_\infty$ gives the inequality
\begin{align*}
M(h_0,h_0 e^{u_\infty}) =  \int_{X_{\boldsymbol{D}}} \langle \sqrt{-1} \Lambda F_{h_0},u_\infty \rangle_{h_0} + \int_{X_{\boldsymbol{D}}} (\psi(u_\infty)(D''(u_\infty)),D''(u_\infty))_{h_0} \leq 0.
\end{align*}
Furthermore, from \cite[Section 5]{Simp1988}, the $L^2$-norm of $D''u_{\infty}$ is finite and the projection on eigenspaces of $\text{ad}(u_{\infty})$ corresponding to nonnegative eigenvalues of $\text{ad}(u_{\infty})$ is zero. This section $u_{\infty}$ is holomorphic with respect to $\partial''_E$ \cite[Lemma 5.4]{BGM}.
		
For any point $x \in X_{\boldsymbol D}$, let $s = u_\infty (x) \in \text{Im}(\mathfrak{g})$. The element $s$ defines a parabolic subgroup $P_s \subseteq G$ together with a strictly anti-dominant character $\chi_\infty: P_s \rightarrow \mathbb{C}^{*}$ (see \cite[\S 2.3]{BGM}). Furthermore, the section $u_{\infty}$ gives a holomorphic reduction of structure group $\sigma_{\infty}: X_{\boldsymbol D} \rightarrow E/P_s$. By \cite[Lemma 2.10]{BGM}, we have
\begin{align*}
	\int_{X_{\boldsymbol{D}}} \langle \sqrt{-1} \Lambda F_{h_0},u_\infty \rangle_{h_0} = \sqrt{-1} \int_{X_{\boldsymbol{D}}}   (\chi_\infty)_* (\sigma_\infty)^* F_{h_0} = 2 \pi \deg^{\rm an} E(h_0,\sigma_\infty, \chi_\infty).
\end{align*}
We now see that
\begin{align*}
	2 \pi \deg^{\rm an} E(h_0,\sigma_\infty, \chi_\infty) = & \int_{X_{\boldsymbol{D}}} \langle \sqrt{-1} \Lambda F_{h_0},u_\infty \rangle_{h_0} \\
	\leq & \int_{X_{\boldsymbol{D}}} \langle \sqrt{-1} \Lambda F_{h_0},u_\infty \rangle_{h_0} - (\varpi(u_\infty)(D''(u_\infty)), D''(u_\infty) )_{h_0}\\
	\leq & \, M(h_0,h_0 e^{u_\infty}) \leq 0,
\end{align*}
where $\varpi : \mathbb{R} \to \mathbb{R}$ is the function with $\varpi (0)=0$ and $\varpi(t)=t^{-1}$ for $t\neq 0$. The second inequality above holds because the function $\psi(t)$ in the Donaldson functional has the same convergence as the function $-\varpi(t)$ when the terms $l_i$, for all $i$, go to infinity. In conclusion, we find a particular reduction of structure group $\sigma_\infty$ and a nontrivial antidominant character $\chi_\infty$ such that the analytic degree is not positive, thus violating the stability of $(E,\partial''_E,\phi)$.
\end{proof}
	
The last step in the proof of Theorem \ref{thm KH corr Higgs} is to show that the harmonic metric $h$ is unique up to Higgs bundle automorphisms. Indeed, if $h$ and $h'$ are two such metrics in $\mathcal{H}$, then we may write $h'=he^s$ for some $s \in \hat{L}^{2,p}_{\delta}(E( \text{Im}(\mathfrak{g})))$. Then, $M(h, he^{ts} )$ is a constant function of $t$. In fact, this is a convex function of $t$, and for any metrics $h, h', h'' \in \mathcal{H}$, it satisfies the relations:
\begin{align*}
	& M(h,h'') = M(h,h') + M(h',h'')\\
	& \frac{d}{dt} M(h,he^{ts}) \rvert _{t=0} = \int_{X_{\boldsymbol{D}}} \langle \sqrt{-1} \Lambda F_h ,s \rangle_{h}.
\end{align*}
From these relations we get that $D''s=0$, and then the element $e^s$ stabilizes the Higgs field $\phi$. Therefore, the element induces an automorphism of the Higgs bundle. This finishes the proof of the theorem.	
\end{proof}

\begin{rem}\label{rem_reduce_to_Simpson}
It is important to point out that the result above does, in fact, reduce to the main theorem of Simpson \cite{Simp} for the case $G=\text{GL}_n(\mathbb{C})$ by taking a faithful representation. Indeed, in \cite[Section 5]{KSZparh} it was shown how the notion of $R$-stability as considered in the proof of Theorem \ref{thm KH corr Higgs} in this article coincides with the stability condition for parabolic Higgs bundles as considered in \cite{Simp}. Note that in \cite[\S 5.2]{BGM}, the authors remark the difficulty in reducing the Kobayashi--Hitchin correspondence they developed to the theorem of Simpson for the stability condition used in that work.
\end{rem}

From Proposition \ref{prop tame ana Higgs and conn}, a tame $\boldsymbol\alpha$-adapted $G$-Higgs bundle $(E,\partial''_E, \phi)$ equipped with an $\boldsymbol\alpha$-adapted harmonic metric $h$ and with residue of the Higgs field $\phi_{\alpha_x}$ around each puncture $x$, corresponds to a tame harmonic $\boldsymbol\beta$-adapted $G$-connection $(V,d')$ with residue $d'_{\beta_x}$equipped with the same $\boldsymbol\alpha$-adapted metric $h$. The proof of Theorem  \ref{thm KH corr Higgs} readily adapts to relate the harmonicity of a metric on a $G$-connection $(V, d', h)$ with the $R_h$-polystability condition: 

\begin{thm}[Kobayashi--Hitchin Correspondence for $G$-connections]\label{cor KH corr mod}
	Let $(V,d')$ be a tame $\boldsymbol\beta$-adapted $G$-connection of degree zero equipped with an $\boldsymbol\alpha$-adapted metric $h_0$. Then, $(V,d')$ admits an $\boldsymbol\alpha$-adapted harmonic metric $h$ which is quasi-isometric to $h_0$ if and only if it is $R_{h_{0}}$-polystable.
\end{thm}



\subsection{Equivalence between the Dolbeault and de Rham categories}

Now we are ready to prove that the functors $\Xi_{\rm Dol}$ and $\Xi_{\rm dR}$ induce an equivalence of categories.
\begin{lem}\label{lem ang=alg Dol (dR)}
	The functors
	\begin{align*}
	\Xi_{\rm Dol}: \mathcal{C}_{\rm Dol}(X_{\boldsymbol D},G,\boldsymbol\theta,\phi_{\boldsymbol\theta}) \rightarrow \mathcal{C}_{\rm Dol}(X,\mathcal{G}_{\boldsymbol\theta},\varphi_{\boldsymbol\theta}) \quad \text{ and } \quad \Xi_{\rm dR}: \mathcal{C}_{\rm dR}(X_{\boldsymbol D},G,\boldsymbol\theta,d'_{\boldsymbol\theta}) \rightarrow \mathcal{C}_{\rm dR}(X,\mathcal{G}_{\boldsymbol\theta},\nabla_{\boldsymbol\theta})
	\end{align*}
	induce equivalences of categories, where $\phi_{\boldsymbol\theta}=\varphi_{\boldsymbol\theta}$ and $d'_{\boldsymbol\theta}=\nabla_{\boldsymbol\theta}$.
\end{lem}

\begin{proof}
	We only give the proof for the functor $\Xi_{\rm Dol}$, and the proof for $\Xi_{\rm dR}$ is similar. By the uniqueness of the harmonic metric (Theorem \ref{thm KH corr Higgs}), the functor $\Xi_{\rm Dol}$ is injective on the objects. Now we consider the surjectivity. Take a logahoric $\mathcal{G}_{\boldsymbol\theta}$-Higgs torsor $(\mathcal{E},\varphi)$ in the category $\mathcal{C}_{\rm Dol}(X,\mathcal{G}_{\boldsymbol\theta},\varphi_{\boldsymbol\theta})$. Restricting this object on $X_{\boldsymbol D}$, we obtain a tame  $\boldsymbol\theta$-adapted $G$-Higgs bundle $(E,\partial''_E,\phi)$. By the construction of the metric in \S\ref{subsect local study}, we obtain a $\boldsymbol\theta$-adapted metric $h_0$ on $E$ (by partition of unity). In total, we construct an element $(E,\partial''_E,\phi,h_0)$ in the category $\mathcal{C}_{\rm Higgs}(X_{\boldsymbol D},G,\boldsymbol\theta,\phi_{\boldsymbol\theta})$, which corresponds to the given parahoric Higgs torsor $(\mathcal{E},\varphi)$ under the functor $\Xi_{\rm Higgs}$. From Propositions \ref{prop analytic=algebraic degree} and \ref{prop ana=alg Higgs and conn}, the tuple $(E,\partial''_E,\phi,h_0)$ is $R_{h_0}$-stable of degree zero. By Theorem \ref{thm KH corr Higgs} again, we obtain a harmonic metric $h$ on $E$, and therefore an element $(E,\partial''_E,\phi,h)$ in the category $\mathcal{C}_{\rm Dol}(X_{\boldsymbol D},G,\boldsymbol\theta,\phi_{\boldsymbol\theta})$. This finishes the proof of this lemma.
\end{proof}

With the same notation as in \S\ref{sect ana Higgs and conn} and \S\ref{sect alg Higgs and conn}, we have:
\begin{thm}\label{thm alg Dol and DR}
The categories $\mathcal{C}_{\rm Dol}(X,\mathcal{G}_{\boldsymbol\alpha},\varphi_{\boldsymbol\alpha})$ and $\mathcal{C}_{\rm dR}(X,\mathcal{G}_{\boldsymbol\beta},\nabla_{\boldsymbol\beta})$ are equivalent.
\end{thm}

\begin{proof}
The theorem follows directly by chasing the diagram below
\begin{center}
\begin{tikzcd}
\mathcal{C}_{\rm Dol}(X_{\boldsymbol D},G,\boldsymbol\alpha,\phi_{\boldsymbol\alpha}) \arrow[rr, "\text{Theorem } \ref{thm ana Dol and DR}", equal] \arrow[d, "\text{Lemma } \ref{lem ang=alg Dol (dR)}"] & &
\mathcal{C}_{\rm dR}(X_{\boldsymbol D},G,\boldsymbol\beta,d'_{\boldsymbol\beta}) \arrow[d, "\text{Lemma } \ref{lem ang=alg Dol (dR)}"]\\
\mathcal{C}_{\rm Dol}(X,\mathcal{G}_{\boldsymbol\alpha},\varphi_{\boldsymbol\alpha}) \arrow[rr, leftrightarrow, dotted]& & \mathcal{C}_{\rm dR}(X,\mathcal{G}_{\boldsymbol\beta},\nabla_{\boldsymbol\beta}),
\end{tikzcd}
\end{center}
where $\phi_{\boldsymbol\alpha}=\varphi_{\boldsymbol\alpha}$ and $d'_{\boldsymbol\beta}=\nabla_{\boldsymbol\beta}$.
\end{proof}

\subsection{Riemann--Hilbert Correspondence}\label{subsect Riemann-Hilbert cat}

Let $\pi_1(X_{\boldsymbol D})$ be the fundamental group of $X_{\boldsymbol D}$, i.e.
\begin{align*}
\pi_1(X_{\boldsymbol D}) = \langle a_i,b_j,c_x \text{ } | \text{ } \prod^g_{i=1} [a_i,b_i] \prod_{x\in \boldsymbol D} c_x = 1  \rangle,
\end{align*}
where the generator $c_x$ corresponds to a loop around the puncture $x$. A \emph{$G$-local system} $\mathscr{L}$ is a $G$-bundle on $X_{\boldsymbol D}$ with parallel transport. Equivalently, fixing a base point $x \in X$, a $G$-local system corresponds to a pair $(\mathscr{L}_x, \rho_x)$, where $\mathscr{L}_x$ is the stalk and $\rho_x : \pi_1(X_{\boldsymbol D}) \rightarrow G$ is a representation of the fundamental group. Although this correspondence depends on the choice of the base point $x$, we can apply the same argument as in \cite[\S 6 The Betti moduli spaces]{Simp3} and identify categories of pairs $(\mathscr{L}_x,\rho_x)$ obtained from distinct choices of base point. Then, a $G$-local system is regarded as a representation of the fundamental group.

Now we consider the relation between $G$-local systems and connections on $G$-bundles with regular singularities on $X_{\boldsymbol D}$. The key point is to study the correspondence locally around each puncture $x \in \boldsymbol D$. We first review the result for $G = {\rm GL}_n(\mathbb{C})$. An arbitrary connection $d + A(z)\frac{dz}{z}$ is gauge equivalent to a connection in the form $d + a \frac{dz}{z}$ \cite[Theorem 5.1.4]{HoTakTan}, where $A(z) \in \mathfrak{gl}_n(R)$ and $a \in \mathfrak{gl}_n(\mathbb{C})$. Based on this result, the corresponding monodromy around the puncture is given by $e^{-2 \pi i a}$, and we can define a representation $\rho$ such that $\rho(c_x)=e^{-2 \pi i a}$. Thus, there is a well-defined correspondence between ${\rm GL}_n(\mathbb{C})$-local systems and connections on vector bundles with regular singularities.

However, for a general complex reductive group $G$, the property above does not hold anymore. The main problem is that, given an arbitrary connection $d + A(z)\frac{dz}{z}$, where $A(z) \in \mathfrak{g}(R)$, it may not be gauge equivalent to one in the form $d+a\frac{dz}{z}$, where $a \in \mathfrak{g}$. Note that here the gauge action is given by $G$ rather than ${\rm GL}_n$. Therefore, a local correspondence between equivalence classes of monodromy data and equivalence classes of connections needs to be studied in detail.

To address this problem, Boalch introduced the parahoric language and established a version of Riemann--Hilbert correspondence, which is called the \emph{tame parahoric Riemann--Hilbert correspondence}. We follow the same notation as in \S\ref{subsect local study} and review his result:

\begin{thm}\cite[Theorem D]{Bo}\label{thm tame parah RH corr Bo}
	There is a canonical bijection between $G_\beta(K)$-orbits of logahoric connections with residue $d'_\beta$ and conjugacy classes of pairs $(M,P)$, where $P$ is a parabolic subgroup of $G$ conjugate to $P_\gamma$ and $M \in P$, such that the Levi factor of $M$ is conjugate to $M_\gamma$.
\end{thm}

Although the above theorem only gives a local correspondence around punctures, it can be extended to a global one naturally. Before we state the result, we introduce the following notions. Let $\boldsymbol\gamma=\{\gamma_x, x \in \boldsymbol D\}$ be a set of weights labelled by punctures in $\boldsymbol D$. Let $P_{\gamma_x}$ be the parabolic subgroup of $G$ defined by the weight $\gamma_x$, and denote by $L_{\gamma_x}$ the Levi subgroup of $P_{\gamma_x}$. Given a collection of elements 
\begin{align*}
M_{\boldsymbol\gamma}=\{M_{\gamma_x}\text{ }|\text{ } M_{\gamma_x} \in L_{\gamma_x}\},
\end{align*}
define ${\rm Hom}(\pi_1(X_{\boldsymbol D}),G, \boldsymbol\gamma, M_{\boldsymbol\gamma})$ to be set of representations $\rho: \pi_1(X_{\boldsymbol D}) \rightarrow G$ such that
\begin{itemize}
	\item $\rho(c_x) \in P$, where $P$ is a parabolic subgroup conjugate to $P_{\gamma_x}$;
	\item the Levi component of $\rho(c_x)$ is conjugate to $M_{\gamma_x}$.
\end{itemize}
Under the natural $G$-action on ${\rm Hom}(\pi_1(X_{\boldsymbol D}),G, \boldsymbol\gamma, M_{\boldsymbol\gamma})$, denote by $\mathcal{C}_{\rm Loc}(X_{\boldsymbol D}, G,\boldsymbol\gamma, M_{\boldsymbol\gamma})$ the corresponding category. An object in $\mathcal{C}_{\rm Loc}(X_{\boldsymbol D}, G,\boldsymbol\gamma, M_{\boldsymbol\gamma})$ is called a \emph{$\boldsymbol\gamma$-filtered $G$-local system (with monodromies $M_{\boldsymbol\gamma}$)} on $X_{\boldsymbol D}$, where the filtered structure comes from the parabolic subgroup $P_{\gamma_x}$ around each puncture $x \in \boldsymbol D$. Let $\mathcal{C}_{\rm Conn}(X,\mathcal{G}_{\boldsymbol\beta},\nabla_{\boldsymbol\beta})$ be the category of logahoric $\mathcal{G}_{\boldsymbol\beta}$-connections on $X$. As a corollary, we obtain the global version of the tame parahoric Riemann--Hilbert correspondence:

\begin{cor}\label{cor tame parah RH corr}
The categories $\mathcal{C}_{\rm Loc}(X_{\boldsymbol D}, G,\boldsymbol\gamma, M_{\boldsymbol\gamma})$ and $\mathcal{C}_{\rm Conn}(X,\mathcal{G}_{\boldsymbol\beta},\nabla_{\boldsymbol\beta})$ are equivalent.
\end{cor}

\subsection{Nonabelian Hodge Correspondence}\label{subsect main result in cat}

In this subsection, we define a stability condition on filtered $G$-local systems and prove a correspondence between stable filtered $G$-local systems on $X_{\boldsymbol D}$ and $R$-stable logahoric connections of degree zero on $X$. Together with results in previous sections, we establish the tame parahoric nonabelian Hodge correspondence at the level of categories.

Let $\mathscr{L}$ be a $\boldsymbol\gamma$-filtered $G$-local system with monodromies $M_{\boldsymbol\gamma}$. Let  $\tau : X_{\boldsymbol D} \rightarrow \mathscr{L}/P$ be a reduction of structure group, where $P$ is a parabolic subgroup of $G$. We say that $\tau$ is \emph{compatible}, if the pullback $\mathscr{L}_\tau$
\begin{center}
	\begin{tikzcd}
	\mathscr{L}_{\tau} \arrow[r, dotted] \arrow[d, dotted] & \mathscr{L} \arrow[d]\\
	X_{\boldsymbol D} \arrow[r, "\tau"]  & \mathscr{L}/P
	\end{tikzcd}
\end{center}
is a $P$-local system. Given a character $\chi$ of $P$, the \emph{degree of the $\boldsymbol\gamma$-filtered $G$-local system $\mathscr{L}$} (with respect to $\tau$ and $\chi$) is defined as
\begin{align*}
\deg^{\rm loc} \mathscr{L}(\tau,\chi) := \langle \boldsymbol\gamma, \chi \rangle.
\end{align*}
Note that as a local system, the degree of the line bundle $\chi_* \mathscr{L}_\tau$ is always zero. Therefore, we omit the term $\chi_* \mathscr{L}_\tau$ in the definition of the degree of filtered local system.

Now given a $\boldsymbol\gamma$-filtered $G$-local system $\mathscr{L}$ with monodromies $M_{\boldsymbol\gamma}$, denote by $(\mathcal{V},\nabla)$ the corresponding  logahoric $\mathcal{G}_{\boldsymbol\beta}$-connection in $\mathcal{C}_{\rm Conn}(X,\mathcal{G}_{\boldsymbol\beta},\nabla_{\boldsymbol\beta})$. By Corollary \ref{cor tame parah RH corr}, there is a one-to-one correspondence between compatible reductions of $(\mathcal{V},\nabla)$ and compatible reductions of $\mathscr{L}$. Furthermore, if $(V,d',h)$ is a tame metrized $\boldsymbol\beta$-adapted $G$-connection corresponding to $(\mathcal{V},\nabla)$, then the one-to-one correspondence also holds for compatible $\boldsymbol\beta$-adapted holomorphic reductions $\sigma$ of $(V,d',h)$ by Lemma \ref{lem red grp struc Higgs and conn}. With respect to the discussion above, we have the following lemma:

\begin{lem}\label{lem loc=ana degree}
	Let $\mathscr{L}$ be a $\boldsymbol\gamma$-filtered $G$-local system in $\mathcal{C}_{\rm Loc}(X_{\boldsymbol D},G,\boldsymbol\gamma,M_{\boldsymbol\gamma})$, and let $(\mathcal{V},\nabla)$ be the corresponding logahoric $\mathcal{G}_{\boldsymbol\beta}$-connection in $\mathcal{C}_{\rm Conn}(X,\mathcal{G}_{\boldsymbol\beta},\nabla_{\boldsymbol\beta})$. Given a compatible reduction $\tau$ of $\mathscr{L}$ and a character $\chi$ of $P$, we have
	\begin{align*}
	\deg^{\rm loc} \mathscr{L}(\tau,\chi) = parh \deg \mathcal{V}(\varsigma, \kappa),
	\end{align*}
	where $\varsigma$ is the compatible reduction of $\mathcal{V}$ (corresponding to $\tau$) and $\kappa$ is the character of $\mathcal{P}_{\boldsymbol \beta}$ (corresponding to $\chi$) by Lemma \ref{lem char}.
\end{lem}

\begin{proof}
	 With the help of Proposition \ref{prop analytic=algebraic degree}, the proof is the same as that of \cite[Lemma 6.5]{Simp}.
\end{proof}

\begin{defn}
	A $\boldsymbol\gamma$-filtered $G$-local system $\mathscr{L}$ is \emph{$R$-stable} (resp. \emph{$R$-semistable}), if for
	\begin{itemize}
		\item any proper parabolic group $P \subseteq G$,
		\item any compatible reduction of structure group $\tau: X_{\boldsymbol D} \rightarrow \mathscr{L}/P$,
		\item any nontrivial anti-dominant character $\chi: P \rightarrow \mathbb{C}^*$, which is trivial on the center of $P$,
	\end{itemize}
	one has
	\begin{align*}
	\deg^{\rm loc} \mathscr{L} (\tau, \chi) > 0, \quad (\text{resp. } \geq 0).
	\end{align*}
\end{defn}

\begin{cor}\label{cor Betti=dR}
	A $\boldsymbol\gamma$-filtered $G$-local system is \emph{$R$-stable} (resp. \emph{$R$-semistable}) if and only if the corresponding logahoric $\mathcal{G}_{\boldsymbol\beta}$-connection is \emph{$R$-stable} (resp. \emph{$R$-semistable}).
\end{cor}

\begin{proof}
	The corollary follows directly from the correspondence of compatible reductions of structure group and the equivalence of degrees.
\end{proof}

Let
\begin{align*}
\mathcal{C}_{\rm B}(X_{\boldsymbol D}, G, \boldsymbol\gamma, M_{\boldsymbol\gamma}) \subseteq \mathcal{C}_{\rm Loc}(X_{\boldsymbol D}, G, \boldsymbol\gamma, M_{\boldsymbol\gamma})
\end{align*}
be the subcategory including all $R$-stable $\boldsymbol\gamma$-filtered $G$-local systems of degree zero, where the subscript {\rm B} is for \emph{Betti}. Here is an informal geometric interpretation of the category $\mathcal{C}_{\rm B}(X_{\boldsymbol D}, G, \boldsymbol\gamma, M_{\boldsymbol\gamma})$. As we discussed above, the category $\mathcal{C}_{\rm Loc}(X_{\boldsymbol D}, G, \boldsymbol\gamma, M_{\boldsymbol\gamma})$ is exactly the set of points on the underlying space of the quotient ${\rm Hom}(\pi_1(X_{\boldsymbol D}),G, \boldsymbol\gamma, M_{\boldsymbol\gamma})/G$. By adding the stability condition of local systems, let
\begin{align*}
{\rm Hom}(\pi_1(X_{\boldsymbol D}),G, \boldsymbol\gamma, M_{\boldsymbol\gamma})^{\rm s} \subseteq {\rm Hom}(\pi_1(X_{\boldsymbol D}),G,\boldsymbol\gamma, M_{\boldsymbol\gamma})
\end{align*}
be the subset of all representations, which correspond to stable $\boldsymbol\gamma$-filtered local systems. Then, the category $\mathcal{C}_{\rm B}(X_{\boldsymbol D}, G, \boldsymbol\gamma, M_{\boldsymbol\gamma})$ can be understood as the set of points on the underlying space of ${\rm Hom}(\pi_1(X_{\boldsymbol D}),G, \boldsymbol\gamma, M_{\boldsymbol\gamma})^{\rm s} /G$.

\begin{rem}
Note that a stable filtered $G$-local system may not correspond to an irreducible representation of the fundamental group and this is not even true for parabolic bundles (that is, even in the case when $G={\rm GL}_n(\mathbb{C})$). Therefore, the quotient ${\rm Hom}(\pi_1(X_{\boldsymbol D}),G, \boldsymbol\gamma, M_{\boldsymbol\gamma})^{\rm s} /G$ cannot be taken as the GIT quotient (in the sense of \cite{Rich}) because it is well-known that a (stable) point in ${\rm Hom}(\pi_1(X_{\boldsymbol D}),G)/\!\!/G$ corresponds to an irreducible representation \cite[Lemma 4.1]{Rich}. It would be interesting to give a GIT construction of this moduli space ${\rm Hom}(\pi_1(X_{\boldsymbol D}),G, \boldsymbol\gamma, M_{\boldsymbol\gamma})^{\rm s} /G$ directly. We are grateful to Carlos Simpson for pointing out this important remark to us.
\end{rem}

\begin{thm}\label{thm main thm in cat}
The categories are equivalent
\begin{align*}
\mathcal{C}_{\rm Dol}(X,\mathcal{G}_{\boldsymbol\alpha},\varphi_{\boldsymbol\alpha}) \cong \mathcal{C}_{\rm dR}(X,\mathcal{G}_{\boldsymbol\beta},\nabla_{\boldsymbol\beta}) \cong \mathcal{C}_{\rm B}(X_{\boldsymbol D}, G,\boldsymbol\gamma, M_{\boldsymbol\gamma}).
\end{align*}
\end{thm}

\begin{proof}
The equivalence of the first two categories follows from Theorem \ref{thm alg Dol and DR}, and the second one follows from Corollary \ref{cor Betti=dR}.
\end{proof}

\begin{rem}
The nonabelian Hodge correspondence established in this article is considered for the categories involving a general complex reductive group $G$. For dealing with real groups, one would need to construct well-defined initial $\boldsymbol \theta$-adapted model metrics $h_0$ on a tame analytic $G$-Higgs bundle that give an approximate solution to the Hermite--Einstein equation near the punctures. In \cite[\S 5.1]{BGM}, the authors provide such models subject to a certain condition on the nilpotent part of the graded pieces of the residue of the Higgs field. A similar condition could also be applied in our parahoric setting and for our notion of stability to provide a special Kobayashi--Hitchin correspondence for real groups, however, since it was not clear to us how to get past this condition and establish the analytic part of the correspondence in general, we did not deal with the real group case here (cf. Remark \ref{comparison_model_real}).
\end{rem}

\section{Tame Parahoric Nonabelian Hodge Correspondence: Moduli Space}\label{sec_naHc moduli spaces}

We show that the nonabelian Hodge correspondence holds not only at the level of categories, but also for the corresponding moduli spaces, which are constructed in \cite{KSZparh}. In this section, all weights are considered to be \emph{rational} because the moduli space construction from \cite{KSZparh} is done for rational weights only.

\subsection{Betti Moduli Problem}

In this subsection, we give the definition of the moduli problem of $\boldsymbol\gamma$-filtered $G$-local systems, of which the Levi factors of monodromies are $M_{\boldsymbol\gamma}$ around punctures. Let
\begin{align*}
\widetilde{\mathcal{M}}_{\rm B}(X_{\boldsymbol D}, G, \boldsymbol \gamma, M_{\boldsymbol \gamma}): ({\rm Sch}/\mathbb{C})^{\rm op} \rightarrow {\rm Sets},
\end{align*}
be a functor where $({\rm Sch}/\mathbb{C})^{\rm op}$ is the (opposite) category of schemes over $\mathbb{C}$ with respect to the \'etale topology. For each $\mathbb{C}$-scheme $S$, the set $\widetilde{\mathcal{M}}_{\rm B}(X_{\boldsymbol D}, G, \boldsymbol \gamma, M_{\boldsymbol \gamma})(S)$ is defined as isomorphism classes of $S$-flat families $\mathscr{L}$ of $\boldsymbol\gamma$-filtered $G$-local systems on $X$ with $M_{\boldsymbol\gamma}$ the Levi factors of monodromies around punctures such that for each point $s \in S$, the fiber $\mathscr{L}|_s$ is $R$-stable. This moduli problem is called the \emph{Betti moduli problem}.

In this paper, we do not construct a moduli space for this moduli problem. Later on, we will prove that this moduli problem is equivalent to the moduli problem for the de Rham moduli space, and therefore, the de Rham moduli space represents both moduli functors by the universal property. Although we do not give the construction for the Betti moduli space directly, we still introduce the notation $\mathcal{M}_{\rm B}(X_{\boldsymbol D}, G, \boldsymbol \gamma, M_{\boldsymbol \gamma})$ for the corresponding moduli space.

\subsection{Dolbeault Moduli Space}

Fixing a topological invariant $\mu$ (see \S\ref{subsect degree zero}), the moduli problem of $R$-stable logahoric $\mathcal{G}_{\boldsymbol\theta}$-Higgs torsors of type $\mu$ on $X$
\begin{align*}
\widetilde{\mathcal{M}}_{\rm Higgs}(X, \mathcal{G}_{\boldsymbol\theta},\mu): {\rm (Sch/\mathbb{C})}^{\rm op}  \rightarrow {\rm Sets}
\end{align*}
is defined as follows. For each $\mathbb{C}$-scheme $S$, the set $\widetilde{\mathcal{M}}_{\rm Higgs}(X,\mathcal{G}_{\boldsymbol\theta},\mu)(S)$ is defined as the collection of isomorphism classes of pairs $(\mathcal{E},\varphi)$  such that
\begin{itemize}
	\item $\mathcal{E}$ is an $S$-flat family of parahoric $\mathcal{G}_{\boldsymbol\theta}$-torsors on $X$;
	\item $\varphi$ is a section of $\mathcal{E}(\mathfrak{g}) \otimes \pi^*_X K_X(D)$, where $\pi_X: X_S \rightarrow X$ is the natural projection;
	\item for each point $s \in S$, the fiber $(\mathcal{E}|_{X \times s},\varphi|_{X \times s})$ is an $R$-stable logahoric $\mathcal{G}_{\boldsymbol\theta}$-Higgs torsor of type $\mu$ on $X$.
\end{itemize}
The authors prove the following result for this moduli problem:
\begin{thm}\cite[Theorem 6.1]{KSZparh}\label{thm exist of parah Higgs moduli}
	There exists a quasi-projective scheme $\mathcal{M}_{\rm Higgs}(X,\mathcal{G}_{\boldsymbol\theta},\mu)$ as the moduli space of the moduli problem $\widetilde{\mathcal{M}}_{\rm Higgs}(X,\mathcal{G}_{\boldsymbol\theta},\mu)$.
\end{thm}

As a special case, we define the moduli problem $\widetilde{\mathcal{M}}_{\rm Dol}(X, \mathcal{G}_{\boldsymbol\theta},\varphi_{\boldsymbol\theta})$ for $R$-stable logahoric $\mathcal{G}_{\boldsymbol\theta}$-Higgs torsors of degree zero with additional residue data $\varphi_{\boldsymbol\theta}$ around punctures. We want to remind the reader that when parahoric degree is zero, then $\mu$ is the trivial element. As a result of Theorem \ref{thm exist of parah Higgs moduli}, we have:

\begin{cor}
	There exists a quasi-projective scheme $\mathcal{M}_{\rm Dol}(X, \mathcal{G}_{\boldsymbol\theta},\varphi_{\boldsymbol\theta})$ as the moduli space of the moduli problem $\widetilde{\mathcal{M}}_{\rm Dol}(X, \mathcal{G}_{\boldsymbol\theta},\varphi_{\boldsymbol\theta})$.
\end{cor}
The moduli space $\mathcal{M}_{\rm Dol}(X, \mathcal{G}_{\boldsymbol\theta},\varphi_{\boldsymbol\theta})$ is called the \emph{Dolbeault moduli space}.

\subsection{De Rham Moduli Space}

Define the moduli problem of $R$-stable logahoric $\mathcal{G}_{\boldsymbol\theta}$-connections of type $\mu$ on $X$
\begin{align*}
\widetilde{\mathcal{M}}_{\rm Conn}(X, \mathcal{G}_{\boldsymbol\theta},\mu): {\rm (Sch/\mathbb{C})}^{\rm op}  \rightarrow {\rm Sets}
\end{align*}
as follows. For each $\mathbb{C}$-scheme $S$, the set $\widetilde{\mathcal{M}}_{\rm Conn}(X,\mathcal{G}_{\boldsymbol\theta},\mu)(S)$ is defined as a collection of pairs $(\mathcal{E},\nabla)$ up to isomorphism such that
\begin{itemize}
	\item $\mathcal{E}$ is an $S$-flat family of parahoric $\mathcal{G}_{\boldsymbol\theta}$-torsors on $X$;
	\item $\nabla$ is a connection on $\mathcal{E}$;
	\item for each point $s \in S$, the fiber $(\mathcal{E}|_{X \times s},\varphi|_{X \times s})$ is an $R$-stable logahoric $\mathcal{G}_{\boldsymbol\theta}$-connections of type $\mu$ on $X$.
\end{itemize}

Similar to the case of logahoric Higgs torsors, we have the following result:
\begin{thm}\label{thm exist of logah Dmod}
	There exists a quasi-projective scheme $\mathcal{M}_{\rm Conn}(X,\mathcal{G}_{\boldsymbol\theta},\mu)$ as the moduli space of the moduli problem $\widetilde{\mathcal{M}}_{\rm Conn}(X, \mathcal{G}_{\boldsymbol\theta},\mu)$.
\end{thm}

\begin{proof}
	Both connections and Higgs bundles can be understood as special cases of $\Lambda$-modules (see \cite{Simp2,Sun202003} for instance). Therefore, the construction of the moduli space for logahoric connections is similar to that of parahoric Higgs torsors (see \cite[Remark 6.11]{KSZparh}).
\end{proof}

Similarly, we can define the moduli problem $\widetilde{\mathcal{M}}_{\rm dR}(X, \mathcal{G}_{\boldsymbol\theta}, \nabla_{\boldsymbol\theta})$ for $R$-stable logahoric $\mathcal{G}_{\boldsymbol\theta}$-connections of degree zero with additional residue data $\nabla_{\boldsymbol\theta}$ around punctures. The corresponding moduli space $\mathcal{M}_{\rm dR}(X, \mathcal{G}_{\boldsymbol\theta},\nabla_{\boldsymbol\theta})$ is called the \emph{de Rham moduli space}.

\begin{cor}
	There exists a quasi-projective scheme $\mathcal{M}_{\rm dR}(X, \mathcal{G}_{\boldsymbol\theta} ,\nabla_{\boldsymbol\theta})$ as the moduli space of the moduli problem $\widetilde{\mathcal{M}}_{\rm dR}(X, \mathcal{G}_{\boldsymbol\theta},\nabla_{\boldsymbol\theta})$.
\end{cor}

\subsection{Identification of the Moduli Spaces}
We first introduce some notation. A moduli problem $\widetilde{\mathcal{M}}: {\rm (Sch/\mathbb{C})}^{\rm op}  \rightarrow {\rm Sets}$ is defined as a contravariant functor. Now we substitute the category ${\rm Sch/\mathbb{C}}$ by the category of complex analytic spaces, and denote by $\widetilde{\mathcal{M}}^{\rm (an)}$ the new moduli problem. Let $\mathcal{M}^{\rm (an)}$ be the moduli space (as a complex analytic space) that co-represents $\widetilde{\mathcal{M}}^{\rm (an)}$. Finally, denote by $\mathcal{M}^{\rm (top)}$ the topological space underlying $\mathcal{M}^{\rm (an)}$. All of the moduli problems considered in this section can be defined naturally on the category of complex analytic spaces, and the corresponding moduli spaces also exist. Now we can state the main theorem.

\begin{thm}\label{thm main thm in mod space}
	There is an isomorphism of complex analytic spaces
	\begin{align*}
	\mathcal{M}^{\rm (an)}_{\rm B}(X_{\boldsymbol D},G, \boldsymbol \gamma, M_{\boldsymbol \gamma} ) \cong \mathcal{M}^{\rm (an)}_{\rm dR}(X, \mathcal{G}_{\boldsymbol \beta}, \nabla_{\boldsymbol \beta}),
	\end{align*}
	and we also have a homeomorphism of topological spaces
	\begin{align*}
	\mathcal{M}^{\rm (top)}_{\rm Dol}(X, \mathcal{G}_{\boldsymbol \alpha}, \varphi_{\boldsymbol \alpha}) \cong \mathcal{M}^{\rm (top)}_{\rm dR}(X, \mathcal{G}_{\boldsymbol \beta}, \nabla_{\boldsymbol \beta}).
	\end{align*}
\end{thm}

Before proving the main theorem, we first extend some categorical equivalences studied in previous sections to a family version. In \S\ref{subsect functor Xi}, we define the functor $\Xi : \mathcal{C}(X_{\boldsymbol D}, G, \boldsymbol\theta) \rightarrow \mathcal{C}(X,\mathcal{G}_{\boldsymbol\theta})$. The given collection of weights $\boldsymbol\theta$ is key data to define the functor. Moreover, this construction also works for families. We only give the description over complex analytic spaces, and refer the reader to \cite[\S 4.2]{Yun} for more details. Let $S$ be a complex analytic space. Denote by $\mathcal{C}(X_{\boldsymbol D}, G, \boldsymbol\theta)(S)$ the category of $S$-flat families of metrized $\boldsymbol\theta$-adapted $G$-bundles on $X_S$, and denote by $\mathcal{C}(X,\mathcal{G}_{\boldsymbol\theta})(S)$ the category of $S$-flat families of parahoric $\mathcal{G}_{\boldsymbol\theta}$-torsors on $X_S$. With the same construction as above, we have a natural functor
\begin{align*}
\Xi: \mathcal{C}(X_{\boldsymbol D}, G, \boldsymbol\theta)(S) \rightarrow \mathcal{C}(X,\mathcal{G}_{\boldsymbol\theta})(S).
\end{align*}
Similarly, the functors $\Xi_{\rm Higgs}$ and $\Xi_{\rm Conn}$ (see \S\ref{sect alg Higgs and conn}) are also well-defined for families (over complex analytic spaces). We will use the notation $\mathcal{C}_{\bullet}(X_{\boldsymbol D},G,\boldsymbol\theta)(S)$ and $\mathcal{C}_{\bullet}(X,\mathcal{G}_{\boldsymbol\theta})(S)$ for the corresponding categories of families over $S$, where $\bullet={\rm Higgs} \text{ or } {\rm Conn}$. A family version of Proposition \ref{prop ana=alg Higgs and conn} is given as follows:

\begin{lem}\label{lem family Dol and DR}
	Let $S$ be a complex analytic space. We have
	\begin{align*}
	\mathcal{C}_{\rm Dol}(X_{\boldsymbol D},G,\boldsymbol\alpha,\phi_{\boldsymbol\alpha})(S) & \cong \mathcal{C}_{\rm Dol}(X,\mathcal{G}_{\boldsymbol\alpha},\varphi_{\boldsymbol\alpha})(S),\\
	\mathcal{C}_{\rm dR}(X_{\boldsymbol D},G,\boldsymbol\beta,d'_{\boldsymbol\beta})(S) & \cong \mathcal{C}_{\rm dR}(X,\mathcal{G}_{\boldsymbol\beta},\nabla_{\boldsymbol\beta})(S).
	\end{align*}
\end{lem}

\begin{proof}
	We only consider the Dolbeault case. Under the functor $\Xi_{\rm Higgs}$, a family $(E,\phi)$ of $R_h$-stable tame harmonic $G$-Higgs bundles on $(X_{\boldsymbol D})_S$ will be sent to a family $(\mathcal{E},\varphi)$ of logahoric $\mathcal{G}_{\boldsymbol\theta}$-Higgs torsors on $X_S$. For each $s \in S$, the fiber $(E_s,\phi_s)$ corresponds to a logahoric $\mathcal{G}_{\boldsymbol\theta}$-Higgs torsor on $X$, which is $R$-stable by Proposition \ref{prop ana=alg Higgs and conn} and of degree zero by Proposition \ref{prop analytic=algebraic degree}. Therefore, $(\mathcal{E},\varphi)$ is a $S$-flat family of $R$-stable logahoric $\mathcal{G}_{\boldsymbol\theta}$-torsors on $X_S$. Consider the other direction. Given an $S$-flat family $(\mathcal{E},\varphi)$ of $R$-stable logahoric $\mathcal{G}_{\boldsymbol\theta}$-Higgs torsors on $X_S$, we obtain a $S$-flat family of tame metrized $\boldsymbol\theta$-adapted $G$-Higgs bundles on $X_{\boldsymbol D}$ by taking the restriction $(E,\phi):=(\mathcal{E},\varphi)|_{X_{\boldsymbol D} \times S}$. With the help of Proposition \ref{prop analytic=algebraic degree} and \ref{prop ana=alg Higgs and conn} again, each fiber of $(E,\phi)$ is an $R_h$-stable tame harmonic $G$-Higgs bundles on $X_{\boldsymbol D}$. This finishes the proof of the proposition.
\end{proof}

The tame parahoric Riemann--Hilbert correspondence (Theorem \ref{thm tame parah RH corr Bo}) also has a family version. We first review the idea of the proof of Theorem \ref{thm tame parah RH corr Bo}. Let $A$ be an element in $\mathfrak{g}_{\beta}(K)$, and denote by $a_0$ the residue of $A\frac{dz}{z}$. Let $a_0=s_0 + n_0$ be the Jordan decomposition of $a_0$, and furthermore, we write $s_0=\nabla_\beta+\sigma_\beta$, where $\nabla_\beta$ is the real part of $s_0$ and $\sigma_\beta$ is the imaginary part. The key part of the proof (see \cite[Proof of Theorem 6]{Bo}) is that given such an element $A \in \mathfrak{g}_{\beta}(K)$, we can find an element $g \in G_\beta(K)$ such that
\begin{align*}
B \frac{dz}{z}= {\rm Ad}(g) A \frac{dz}{z}+ dg \cdot g^{-1},
\end{align*}
where $B=\sum_{i=0} b_i z^i$, $b_0=\nabla_\beta + \sigma_\beta$ and
\begin{align*}
[\nabla_\beta, b_i]=i b_i, \quad [\sigma_\beta, b_i]=0.
\end{align*}

Now let $S$ be an affine complex variety. Let $G_\beta(\mathcal{O}_S)$ be the parahoric group with coefficient in $\mathcal{O}_S \otimes_{\mathbb{C}} K$, and denote by $\mathfrak{g}_\beta(\mathcal{O}_S)$ its parahoric Lie algebra. Proving a family version of Theorem \ref{thm tame parah RH corr Bo} is equivalent to proving the following lemma, and a similar statement in characteristic $p$ is given in \cite[Lemma 4.3]{LS}:
\begin{lem}
	Let $S$ be an affine complex analytic space. Let $A \in \mathfrak{g}_{\beta}(\mathcal{O}_S)$ be such that the residue $a_0 \in \mathfrak{g}$ and the semisimple part of $a_0$ is $\nabla_\beta+\sigma_\beta$. Then, there exists an element $g \in G_\beta(\mathcal{O}_S)$ such that
	\begin{align*}
	B \frac{dz}{z}= {\rm Ad}(g) A \frac{dz}{z}+ dg \cdot g^{-1},
	\end{align*}
	where $B=\sum_{i=0} b_i z^i$, $b_0=\nabla_\beta + \sigma_\beta$ and
	\begin{align*}
	[\nabla_\beta, b_i]=i b_i, \quad [\sigma_\beta, b_i]=0.
	\end{align*}
\end{lem}

\begin{proof}
	The proof of this lemma is the same as the first part of the proof of \cite[Theorem 7.9]{Bo}, and the only difference is that the coefficient is taken from $\mathcal{O}_S$ rather than $\mathbb{C}$. We omit the proof here.
\end{proof}

Thus, a family version of Theorem \ref{thm tame parah RH corr Bo} and Corollary \ref{cor tame parah RH corr} follows directly.

\begin{lem}\label{lem family tame parah RH corr}
	The categories $\mathcal{C}_{\rm Loc}(X_{\boldsymbol D}, G,\boldsymbol\gamma, M_{\boldsymbol\gamma})(S)$ and $\mathcal{C}_{\rm Conn}(X,\mathcal{G}_{\boldsymbol\beta},\nabla_{\boldsymbol\beta})(S)$ are equivalent.
\end{lem}

As a conclusion, we have:

\begin{lem}\label{lem family main thm in cat}
	There is an equivalence of categories
	\begin{align*}
	\mathcal{C}_{\rm dR}(X,\mathcal{G}_{\boldsymbol\beta},\nabla_{\boldsymbol\beta})(S) \cong \mathcal{C}_{\rm B}(X_{\boldsymbol D}, G,\boldsymbol\gamma, M_{\boldsymbol\gamma})(S).
	\end{align*}
\end{lem}

\begin{proof}
	With the same idea as in the proof of Theorem \ref{thm main thm in cat}, this lemma follows from Lemma \ref{lem family Dol and DR} and \ref{lem family tame parah RH corr}.
\end{proof}

Now we are ready to prove the main Theorem \ref{thm main thm in mod space}. The proof is similar to that in \cite[\S 7]{Simp3}. Instead of following the proof step by step, we only give the key ideas.

\begin{proof}[Proof of the first statement]
	Since the moduli space $\mathcal{M}^{\rm (an)}_{\rm dR}(X,\mathcal{G}_{\boldsymbol\beta},\nabla_{\boldsymbol\beta})$ only parameterizes stable logahoric $\mathcal{G}_{\boldsymbol\beta}$-connections, it represents the moduli problem $\widetilde{\mathcal{M}}^{\rm (an)}_{\rm dR}(X,\mathcal{G}_{\boldsymbol\beta},\nabla_{\boldsymbol\beta})$. Thus, it is enough to prove that the moduli problems $\widetilde{\mathcal{M}}^{\rm (an)}_{\rm B}(X_{\boldsymbol D},G, \boldsymbol \gamma, M_{\boldsymbol \gamma} )$ and $\widetilde{\mathcal{M}}^{\rm (an)}_{\rm dR}(X, \mathcal{G}_{\boldsymbol \beta}, \nabla_{\boldsymbol \beta})$ are equivalent. More precisely, for each complex analytic space $S$, we have to show
	\begin{align*}
	\widetilde{\mathcal{M}}^{\rm (an)}_{\rm B}(X_{\boldsymbol D},G, \boldsymbol \gamma, M_{\boldsymbol \gamma} )(S) \cong \widetilde{\mathcal{M}}^{\rm (an)}_{\rm dR}(X, \mathcal{G}_{\boldsymbol \beta}, \nabla_{\boldsymbol \beta})(S).
	\end{align*}
	The equivalence is already given in Lemma \ref{lem family main thm in cat}, and this finishes the proof of the first statement.
\end{proof}

\begin{proof}[Proof of the second statement]
	Under the equivalence of Theorem \ref{thm ana Dol and DR} and Theorem \ref{thm alg Dol and DR}, it is enough to show that the equivalence of categories $\mathcal{C}_{\rm Dol}(X_{\boldsymbol D},G,\boldsymbol\alpha,\phi_{\boldsymbol\alpha})$ and $\mathcal{C}_{\rm dR}(X_{\boldsymbol D},G,\boldsymbol\beta,d'_{\boldsymbol\beta})$ induces a homeomorphism of the corresponding topological spaces. Note that the equivalence of categories is established under the existence of harmonic metrics. Therefore, we only have to show that given a sequence $\{(V_i,h_i)\}_{i \in I}$ of harmonic bundles, there is a harmonic bundle $(V,h)$, to which a subsequence $(V_{i'},h_{i'})_{i' \in I'}$ converges. This is a direct result of \cite[Lemma 8.1]{Simp4}.
\end{proof}

\vspace{2mm}

\textbf{Acknowledgments}.
We would like to kindly thank Oscar Garc\'{i}a-Prada and Mao Sheng for very helpful discussions, as well as Philip Boalch and Carlos Simpson for their comments on a first draft of the paper, which resulted in several improvements and corrections. We also thank Thomas Haines and Qiongling Li for their interest. P. Huang acknowledges funding by the Deutsche Forschungsgemeinschaft (DFG, German Research Foundation) – Project-ID 281071066 – TRR 191. H. Sun is partially supported by NSFC12101243.
\vspace{2mm}

\appendix

\section{Connections and Lie algebroids}\label{sect appendix}
In the appendix, we discuss the notion of connection on principal bundles and torsors. There are many equivalent definitions. In this paper, we use two of them, the usual one as a first order differential operator and the one involving the Lie algebroid of infinitesimal symmetries. There are many good references about the topic and we refer the reader to \cite{Atiyah,Chenzhu1,CraSau,Mac,Torte} for more details.

\subsection{Connections on Principal Bundles -- Analytically}
Let $X$ be a complex manifold. Let $G$ be a connected complex reductive group, and let $E$ be a $G$-bundle on $X$. We consider the following sheaves
\begin{itemize}
	\item $\mathscr{A}^p$: sheaf of $C^{\infty}$ $p$-forms;
	\item $\mathscr{A}^p(\bullet)$: sheaf of $C^{\infty}$ $\bullet$-valued $p$-forms;
	\item $\mathscr{A}^{p,q}$: sheaf of $C^{\infty}$ forms of type $(p,q)$;
	\item $\mathscr{A}^{p,q}(\bullet)$: sheaf of $C^{\infty}$ $\bullet$-valued forms of type $(p,q)$,
\end{itemize}
where $\bullet$ represents a vector bundle or a $G$-bundle on $X$. We use the notation $\mathscr{Z}$ for closed forms. Denote by $d = \partial + \bar{\partial}$ the usual exterior differential operator.

\begin{defn}
	A \emph{connection} on $E$ is a first order differential operator $D: \mathscr{A}^0(E) \rightarrow \mathscr{A}^1(E)$ satisfying the Leibniz rule
	\begin{align*}
	D(fs) = df \otimes s + f D(s),
	\end{align*}
	where $f \in C^{\infty}(U)$ and $s \in \mathscr{A}^0(E)(U)$ for any open set $U$. It is \emph{integrable} if $D \circ D=0$.
\end{defn}
As usual, a connection $D$ can be written as a sum $D=d'+d''$, where $d'$ and $d''$ are the operators of type $(1,0)$ and $(0,1)$ respectively.

There is an equivalent definition of connections in the viewpoint of Lie algebroids. Let $T_X$ be the tangent bundle. Denote by ${\rm At}(E)$  the \emph{Atiyah algebroid} of $E$. A section of ${\rm At}(E)$ is a pair $(v,v')$, where $v \in T_X$ and $v' \in T_E$, such that
\begin{enumerate}
	\item the restriction of $v'$ to $X$ is $v$;
	\item $v'$ is $G$-invariant.
\end{enumerate}
We have a short exact sequence for ${\rm At}(E)$
\begin{align*}
0 \rightarrow E(\mathfrak{g}) \rightarrow {\rm At}(E) \xrightarrow{pr} T_X \rightarrow 0,
\end{align*}
where $pr: {\rm At}(E) \rightarrow T_X$ is the projection map.
\begin{defn}
	A \emph{connection} on $E$ is a map $\nabla: T_X \rightarrow {\rm At}(E)$ such that $pr \circ \nabla = {\rm id}$. It is \emph{integrable} if the map $\nabla$ is a homomorphism of Lie algebroids.
\end{defn}

\begin{exmp}\label{exmp connection analytic}
The second definition is very useful to understand the structure of connections on $G$-bundles locally. We first consider the trivial $G$-bundle $E=G \times X$. Then, the sheaf of connections on $E$ is isomorphic to the sheaf $\mathscr{A}^1(E(\mathfrak{g}))$ via $D \mapsto D - d$, where $d$ is the usual exterior differential operator. For integrable connections, it is isomorphic to $\mathscr{Z}^1(E(\mathfrak{g}))$. Under this isomorphism, the corresponding element of a connection in $\mathscr{A}^1(E(\mathfrak{g}))$ is called the \emph{connection form} in this paper.

Now we consider a general $G$-bundle $E$. The example above actually shows that for any connection $D$ on $E$, it can be regarded as an element of $\mathscr{A}^1(E(\mathfrak{g}))$ locally. More precisely, let $U \subseteq X$ be an open subset such that $E|_U$ is the trivial $G$-bundle. Then, the restriction $D|_U$ corresponds to an element in $\mathscr{A}^1(E(\mathfrak{g}))(U)$. Furthermore, given an element $\phi$ in $\mathscr{A}^1(E(\mathfrak{g}))$, we obtain a connection on $E$ by gluing local data $D|_U+\phi_U$, and this connection is denoted by $D + \theta$.
\end{exmp}

\begin{defn}\label{defn holomorphic struct}
	A \emph{holomorphic structure} on $E$ is a first order differential operator $\partial''_E: \mathscr{A}^{0}(E) \rightarrow \mathscr{A}^{0,1}(E)$ satisfying
	\begin{enumerate}
		\item the Leibniz rule
		\begin{align*}
		\partial''_E (fs) = \bar{\partial}(f) \otimes s + f \partial''_E(s),
		\end{align*}
		where $f \in C^{\infty}(U)$ and $s \in \mathscr{A}^0(E)(U)$ for any open set $U$,
		\item the integrability condition $\partial''_E \circ \partial''_E = 0$.
	\end{enumerate}	
\end{defn}

\subsection{Connections on Torsors -- Algebraically}
Let $X$ be a smooth projective variety over $\mathbb{C}$. Denote by $T_X$ (resp. $\Omega_X$) the tangent sheaf (resp. cotangent sheaf). Let $\mathcal{G}$ be an affine smooth group scheme on $X$ equipped with an integrable connection $\nabla_{\mathcal{G}}:\mathcal{O}_{\mathcal{G}} \rightarrow \mathcal{O}_{\mathcal{G}} \otimes \Omega_X$ compatible with the structure of $\mathcal{G}$. We first review the construction for connections on $\mathcal{G}$-torsors and we refer the reader to \cite[Appendix]{Chenzhu1} for more details. 

Let $\mathcal{E}$ be a $\mathcal{G}$-torsor. An \emph{(integrable) connection} on $\mathcal{E}$ is an (integrable) connection $\nabla: \mathcal{O}_{\mathcal{E}} \rightarrow \mathcal{O}_{\mathcal{E}} \otimes \Omega_X$ such that the diagram
\begin{center}
	\begin{tikzcd}
	\mathcal{O}_{\mathcal{E}} \arrow[rr,"\nabla"] \arrow[d,"a"] & & \mathcal{O}_{\mathcal{E}} \otimes \Omega_X \arrow[d,"a \otimes 1"] \\
	\mathcal{O}_{\mathcal{E}} \otimes \mathcal{O}_{\mathcal{G}} \arrow[rr,"\nabla \otimes 1 + 1 \otimes \nabla_{\mathcal{G}}"] & & (\mathcal{O}_{\mathcal{E}} \otimes \mathcal{O}_{\mathcal{G}}) \otimes \Omega_X
	\end{tikzcd}
\end{center}
commutes, where $a: \mathcal{O}_{\mathcal{E}} \rightarrow \mathcal{O}_{\mathcal{E}} \otimes \mathcal{O}_{\mathcal{G}}$ is the co-action map.

Similar to the analytic case, we have an equivalent definition from the point of view of Lie algebroids. Denote by ${\rm At}(\mathcal{E})$ the Lie algebroid of infinitesimal symmetry of $\mathcal{E}$, of which the sections are pairs $(v,v')$, where $v \in T_X$ and $v' \in T_{\mathcal{E}}$, such that
\begin{enumerate}
	\item the restriction of $v'$ to $\mathcal{O}_X$ is equal to $v$;
	\item $v'$ is $\mathcal{G}$-invariant.
\end{enumerate}
Then, there exists a short exact sequence
\begin{align*}
0 \rightarrow \mathcal{E}(\mathfrak{g}) \rightarrow {\rm At}(\mathcal{E}) \rightarrow T_{X} \rightarrow 0.
\end{align*}
Then, a \emph{connection} on $\mathcal{E}$ is a splitting $\nabla: T_X \rightarrow {\rm At}(\mathcal{E})$ of this exact sequence.

Now we generalize the construction above to \emph{logarithmic connections} \cite[\S 3]{Torte}. Let $\boldsymbol D$ be a smooth divisor on $X$.
\begin{defn}
	A \emph{logarithmic connection (with poles along $\boldsymbol D$)} on $\mathcal{E}$ is defined as a  $\mathbb{C}$-linear map $\nabla: \mathcal{O}_{\mathcal{E}} \rightarrow \mathcal{O}_{\mathcal{E}} \otimes \Omega_X({\rm log}\boldsymbol{D})$ such that
	\begin{enumerate}
		\item $\nabla$ satisfies the Leibniz rule;
		\item the diagram
		\begin{center}
			\begin{tikzcd}
			\mathcal{O}_\mathcal{E} \arrow[rr,"\nabla"] \arrow[d,"a"] & & \mathcal{O}_\mathcal{E} \otimes \Omega_X({\rm log}\boldsymbol{D}) \arrow[d,"a \otimes 1"] \\
			\mathcal{O}_\mathcal{E} \otimes \mathcal{O}_{\mathcal{G}} \arrow[rr,"\nabla \otimes 1 + 1 \otimes \nabla_{\mathcal{G}}"] & & (\mathcal{O}_\mathcal{E} \otimes \mathcal{O}_{\mathcal{G}}) \otimes \Omega_X({\rm log}\boldsymbol{D})
			\end{tikzcd}
		\end{center}
		commutes.
	\end{enumerate}
	A connection $\nabla$ is \emph{integrable} if $\nabla \circ \nabla =0$.
\end{defn}

We define a Lie algebroid ${\rm At}(\mathcal{E})_{\boldsymbol D}$, of which the sections are pairs $(v,v')$, where $v \in T_X(-{\rm log}\boldsymbol{D})$ and $v' \in T_{\mathcal{E}}(-{\rm log}\boldsymbol{D})$, such that
\begin{enumerate}
	\item the restriction of $v'$ to $\mathcal{O}_X$ is equal to $v$;
	\item $v'$ is $\mathcal{G}$-invariant.
\end{enumerate}
Then, we obtain a short exact sequence
\begin{align*}
0 \rightarrow \mathcal{E}(\mathfrak{g}) \rightarrow {\rm At}(\mathcal{E})_{\boldsymbol D} \rightarrow T_{X}(-{\rm log}\boldsymbol{D}) \rightarrow 0.
\end{align*}

\begin{defn}
A \emph{logarithmic connection} on $\mathcal{E}$ is a splitting $\nabla: T_{X}(-{\rm log}\boldsymbol{D}) \rightarrow {\rm At}(\mathcal{E})_{\boldsymbol D}$ of the exact sequence. If $\nabla$ is a homomorphism of Lie algebroids, it is an \emph{integrable connection}.
\end{defn}

Similar to the discussion in Example \ref{exmp connection analytic}, a logarithmic connection on $\mathcal{E}$ can be regarded as a section of $\mathcal{E}(\mathfrak{g}) \otimes \Omega_X({\rm log} \boldsymbol{D})$ locally.

In this paper, we are considering parahoric group schemes. Let $X$ be a smooth projective curve. Fixing a parahoric group scheme $\mathcal{G}_{\boldsymbol\theta}$ on $X$, it is an affine smooth group scheme on $X$ together with a canonical connection $\nabla_{\mathcal{G}_{\boldsymbol\theta}}$ on $\mathcal{G}_{\boldsymbol\theta}$. Thus, the above construction applies directly to parahoric $\mathcal{G}_{\boldsymbol\theta}$-torsors.

\bigskip
\noindent\small{\textsc{Ruprecht-Karls-Universit\"{a}t Heidelberg}\\
		Mathematisches Institut Universit\"{a}t Heidelberg, Im Neuenheimer Feld 205, Heidelberg 69120, Germany}\\
\emph{E-mail address}:  \texttt{pfhwang@mathi.uni-heidelberg.de}
	
\bigskip
\noindent\small{\textsc{Alexander von Humboldt-Stiftung \& Ruprecht-Karls-Universit\"{a}t Heidelberg}\\
Mathematisches Institut Universit\"{a}t Heidelberg, Im Neuenheimer Feld 205, Heidelberg 69120, Germany}\\
\emph{E-mail address}:  \texttt{gkydonakis@mathi.uni-heidelberg.de}
	
\bigskip
\noindent\small{\textsc{Department of Mathematics, South China University of Technology}\\
		381 Wushan Rd, Tianhe Qu, Guangzhou, Guangdong, China}\\
\emph{E-mail address}:  \texttt{hsun71275@scut.edu.cn}
	
\bigskip
	
\noindent\small{\textsc{Department of Mathematics, University of Maryland, College Park}\\
		4176 Campus Drive - William E. Kirwan Hall,
		College Park, MD 20742-4015, USA}\\
\emph{E-mail address}: \texttt{ltzhao@umd.edu}
	
\end{document}